\documentclass[sts]{imsart}

\RequirePackage[OT1]{fontenc}
\RequirePackage{amsthm,amsmath}
\RequirePackage{enumerate}
\RequirePackage[citecolor=blue,urlcolor=blue]{hyperref}
\usepackage[dvipsnames,svgnames]{xcolor}
\usepackage{bbold}
\usepackage[normalem]{ulem}
\usepackage{tikz}
\usepackage{pgfplots}
\usepackage{booktabs}
\usepackage[ruled,vlined]{algorithm2e}
\usepackage{thmtools} 
\usepackage{subcaption}
\usepackage{float}
\usepackage{subcaption}
\usepackage{caption}
\usepackage{enumitem}
\usepackage{thmtools, thm-restate}
\usepackage{lscape}
\usepackage{varioref}
\usepackage{cleveref}
\usepackage[format=plain,labelfont=bf,textfont=it, font=small]{caption}
\usepackage{algorithmic}

\usepackage{graphicx}
\usepackage{amsfonts}
\usepackage{wrapfig}
\usepackage{comment}
\usepackage{tabularx}

\usepackage{tikz}
\usetikzlibrary{arrows.meta, positioning, matrix, fit, backgrounds, calc}

\RequirePackage[authoryear]{natbib}

\startlocaldefs

\newcommand{\Ind}{\ensuremath{{\mathds{1}}}} 
\usepackage{dsfont}
\usepackage{caption}
\usepackage{float}

\newcommand{\distas}[1]{\mathbin{\overset{#1}{\kern\z\sim}}}%
\newsavebox{\mybox}\newsavebox{\mysim}
\newcommand{\distras}[1]{%
  \savebox{\mybox}{\hbox{\kern3pt$\scriptstyle#1$\kern3pt}}%
  \savebox{\mysim}{\hbox{$\sim$}}%
  \mathbin{\overset{#1}{\kern\z@\resizebox{\wd\mybox}{\ht\mysim}{$\sim$}}}%
}

\usepackage[textwidth=3cm]{todonotes}

\newcommand{\Prob}{\ensuremath{{\mathds P}}} 

\newcommand{\indep}{\raisebox{0.05em}{\rotatebox[origin=c]{90}{$\models$}}}
\newcommand{\R}{\ensuremath{{\mathds R}}}
\newcommand{\E}{\ensuremath{{\mathds E}}}
\newcommand{\X}{\ensuremath{{\mathcal X}}}

\newcommand{\PX}{\ensuremath{{P_X}}}
\newcommand{\pX}{\ensuremath{{p}}}

\numberwithin{equation}{section}
\theoremstyle{plain}

\theoremstyle{remark}
\newtheorem{definition}{Definition}[section]

\newtheorem{example}{Example} 
\newtheorem{remark}{Remark}[section]

\endlocaldefs

\begin{document}

\begin{frontmatter}
\title{What is a good imputation under MAR missingness?}
\runtitle{MAR Imputation}

\begin{aug}
\author[A]{\fnms{Jeffrey}~\snm{Näf}\ead[label=e1]{jeffrey.naf@unige.ch}},
\author[B]{\fnms{Erwan}~\snm{Scornet}\ead[label=e2]{erwan.scornet@po-ly-tech-ni-que.edu}}
\and
\author[C]{\fnms{Julie}~\snm{Josse}\ead[label=e3]{julie.josse@inria.fr}}
\address[A]{Research Institute for Statistics and Information Science, University of Geneva\printead[presep={\ }]{e1}.}
\address[B]{Sorbonne Universite and Universite Paris Cite, \\
        CNRS, Laboratoire de Probabilites, \\
Statistique et Modelisation, F-75005 Paris, France\printead[presep={\ }]{e2}.}
\address[C]{Inria, PreMeDICaL Team, University of Montpellier\printead[presep={\ }]{e3}.}
\end{aug}

\begin{abstract}
   Missing values pose a persistent challenge in modern data science. Consequently, there is an ever-growing number of publications introducing new imputation methods in various fields. The present paper attempts to take a step back and provide a more systematic analysis. Starting from an in-depth discussion of the Missing at Random (MAR) condition for nonparametric imputation, we first investigate whether the widely used fully conditional specification (FCS) approach indeed identifies the correct conditional distributions. Based on this analysis, we propose three essential properties an ideal imputation method should meet, thus enabling a more principled evaluation of existing methods and more targeted development of new methods. In particular, we introduce a new imputation method, denoted mice-DRF, that meets two out of the three criteria. We also discuss ways to compare imputation methods, based on distributional distances. Finally, numerical experiments illustrate the points made in this discussion.
\end{abstract}


\end{frontmatter}

\section{Introduction}

In this paper, we discuss general-purpose (multiple) imputation of missing data sets: instead of imputing for a specific estimation goal or target, we focus on imputations that can be used for a wide variety of analyses in a second step. Developing such imputation methods is still an area of active research, as is benchmarking imputations. To categorize the wealth of imputation methods, one usually differentiates between joint modeling methods that impute the data using one (implicit or explicit) model and the fully conditional specification (FCS) where a different model for each dimension is trained \citep{FCS_Van_Buuren2007, VANBUUREN2018}. Joint modeling approaches may be based on parametric distributions \citep{schafer1997analysis}, or on neural networks, such as generative adversarial network (GAN) \cite{GAIN, directcompetitor1, directcompetitor2} or variational autoencoder (VAE) \cite{MIWAE,VAE1,VAE2, VAE3}. In contrast, in the FCS approach, imputation is done one variable at a time, based on conditional distributions \citep[see, e.g.,][]{sequentialapproach0noaccess, sequentialapproach1, sequentialapproach2, greatoverview}. 
The most prominent example of FCS is the multiple imputation by chained equations (MICE) methodology \citep{mice}. 
In this paper, we address three questions.

First, is imputation under MAR actually possible with the FCS approach? Formally, using the so-called pattern-mixture model (PMM, \cite{little_patternmixture}) view of MAR, we prove that the conditional distribution needed to impute a variable $X_j$ is identifiable, in the simplified setting where all other variables were already imputed correctly. Thus, imputation with the FCS approach is feasible in principle. As we will show, this result is non-trivial in the non-parametric missing data framework, due to the distribution shifts across missing patterns, a phenomenon which may occur even in MAR settings. In fact, our discussion emphasizes that the FCS approach to impute one variable at a time is what makes this identification possible and that different (joint modeling) approaches might fail to identify the correct distribution under MAR. In this context, we compare the MAR condition we consider to stronger conditions used in the literature, such as in \cite{directcompetitor1}, \cite{directcompetitor2} and \cite{RMAR_Psychometrika}, and explore their relations in detail. We also compare these MAR definitions with conditions formulated in the graphical modeling literature, for instance in \cite{nabi2020missing, nabi2025define}.

Second, what properties should an ideal (FCS) imputation method have? Despite the previous positive identification result, FCS imputation can be extremely challenging, since distributions of observed variables differ across missing patterns. 
We list three properties an ideal imputation method should meet. In short, it should be a distributional regression method, able to capture non-linear dependencies between covariates, while being robust to \emph{covariate shifts}. This also implies that MAR should be accompanied by an overlap condition that guarantees that imputation is possible. We discuss these issues and existing methods that meet some of these criteria and introduce a new method based on the Distributional Random Forest of \cite{DRF-paper}, denoted ``mice-DRF''. 

Third, how should imputation methods be evaluated? Currently, imputation methods are largely benchmarked and evaluated based on measuring the root mean squared error (RMSE) between the imputed and the underlying true values, see e.g., \cite{stekhoven_missoforest, Waljee2013, GAIN,knn_adv2,metabolomics1, knn_adv1, VAE1, VAE2, awesomebenchmarkingpaper, Dong2021} and many others. However, when comparing imputation methods, one should refrain from using measures such as the RMSE, as already pointed out \citep{VANBUUREN2018, RFimputationpaper}. Indeed, measures like RMSE favor methods that impute conditional means, instead of draws from the conditional distribution. Hence, using RMSE as a validation criterion tends to favor methods that artificially strengthen the dependence between variables and lead to severe biases in parameter estimates and uncertainty quantification. Following \cite{ImputationScores} and others, we emphasize here that imputation is a distributional prediction task and needs to be evaluated as such. Thus, we advocate the use of a distributional metric or score \citep{gneiting, EnergyDistance} between actual and imputed data sets. In particular, we focus on the energy distance \citep{EnergyDistance} which is simple to calculate and does not require to choose any tuning parameters. 

The remainder of the article is organized as follows. In Section \ref{Sec_Main_Results}, we study different MAR conditions and imputations in more detail and present our identification results. In Section \ref{Sec_GraphicalModelscomparison} we draw comparisons to the graphical modeling literature for missing values. We then use these insights to present recommendations for imputation methods, including three properties any ideal imputation method should meet in Section \ref{Sec_Implications}. Here we also discuss the evaluation of imputation methods and introduce the negative energy distance as a measure of imputation quality. Finally, we illustrate the main points of this paper in two empirical examples in Section \ref{Sec_Empirical}. Code to replicate the experiments can be found in \url{https://github.com/JeffNaef/MARimputation}.


\subsection{Related Work} 

Though the literature on missingness is vast, the results and discussions presented in this paper are new to the best of our knowledge. Most of the papers discussing MAR add the additional assumption that the distribution of $X$ (complete observations) and $M\mid X$ (distribution of the missing pattern conditional on complete observations) are parameterized by two distinct sets of parameters, leading to the classical ignorability property of \cite{Rubin_Inferenceandmissing}. This simplifies the analysis and generally avoids the issues we discuss here. For instance, while the FCS and, in particular, the MICE approach has been studied theoretically \citep{RubinLittlebook, MICE_Results, MICE_Results2} assuming this ignorability, the identification problems in this general setting appear to have been largely ignored. Instead, these papers generally focus on the challenging problem of potential incompatibility of the conditional models and analyze the convergence and asymptotic properties of the FCS iterations. Our aim is in a sense much simpler, as we want to know whether the conditional distributions are identifiable under MAR when no assumption on the parametrization is placed. An important exception is given in the recent work of \cite{directcompetitor0}. They appear to be the first to discuss the same identification result under MAR. Our result was derived independently and in the context of a more general discussion of the MAR condition for imputation. In contrast \cite{directcompetitor0} only discuss MAR briefly and then focus on binary data and the No-Self-Censoring condition, the latter being an MNAR situation that neither implies nor is implied by MAR \citep{directcompetitor0}.

Other works use the framework of pattern-mixture models focusing, for example, on the generative adversarial network (GAN) approach:
Both \cite{directcompetitor1} and  \cite{directcompetitor2} make use of the PMM view in their proofs, without explicitly mentioning this, as does the original GAIN paper of \cite{GAIN}. We essentially provide a similar identification result for the FCS or sequential approach under MAR as \cite{directcompetitor1} provide for their GAN-based approach. However, the identification results in \cite{directcompetitor1} and \cite{directcompetitor2} for GAN-based methods rely on stronger MAR conditions, as shown in Section \ref{Sec_FCSunderMAR}. Similarly, \cite{RecoveringProbabilityDistributionsfrommissing} claims the full distribution is recoverable under MAR, but uses a conditional independence condition that is much stronger than the MAR condition we consider. Indeed, graph-based papers concerned with recoverability usually assume variables that are always observed and formulate MAR as conditional independence statements, see e.g \cite{MARthroughconditionalindependence, RMAR_Psychometrika}.\footnote{An interesting recent exception of a weaker MAR definition is given in \cite{nabi2025define}, as we discuss in Section \ref{Sec_GraphicalModelscomparison}.} Following \cite{RMAR_Psychometrika}, we will refer to this condition as RMAR in this paper and emphasize that it is much stronger than the traditional MAR condition of \cite{Rubin_Inferenceandmissing}. Our analysis shows that the variable-by-variable approach of FCS is crucial in identifying the true distribution under MAR.
To the best of our knowledge, we are also the first to propose a list of properties that an imputation method in the FCS framework should have, based on a thorough analysis of the MAR condition. This list complements existing guidelines on general imputation methods with a different focus, see, e.g., \cite[Section 4]{greatoverview}. Finally, when considering the evaluation of imputation methods, we propose to use the energy distance. The idea of using a distributional distance was motivated by \cite{ImputationScores}. Independently, \cite{Naturepaper} propose a similar approach but using the sliced Wasserstein distance \citep{slicedWasserstein}. Their procedure is designed for high-dimensional data and rather complicated as it involves randomly partitioning the data and projecting to the real numbers multiple times. In contrast, we simply propose to calculate the energy distance between the imputed and real data.


\subsection{Notation and Motivating Example}\label{Sec_Notation}


We assume an underlying probability space $(\Omega, \mathcal{A}, \Prob)$ on which all random elements are defined. Throughout, we take $\mathcal{P}$ to be a collection of probability measures on $\R^d \times \{0,1\}^d$, dominated by some $\sigma$-finite measure $\mu$. We denote the true (unobserved) complete data distribution by $P \in \mathcal{P}$, with density $p$. We assume that $n$ i.i.d. samples $(X_i,M_i)_{i=1}^n$, distributed as $(X,M)$, are generated, where $X \in \mathds{R}^d$ corresponds to the complete data and $M \in \{0, 1\}^d$ the associated missing pattern (or missing mask). 
Accordingly, we only observe the masked observations $X_i^*$, where $X_{i,j}^*=\textrm{NA}$ for $M_{i,j}=1$ and $X_{i,j}^*=X_{i,j}$ for $M_{i,j}=0$. For instance, if the complete observation $X_1 = (X_{1,1}, X_{1,2}, X_{1,3})$ is associated to the missing pattern $M_1=(1,0,0)$, then we observe $X_1^* = (\textrm{NA}, X_{1,2},X_{1,3})$, where the first component is missing and the other two are observed. The marginal distribution of $M$ is completely determined by $\Prob(M=m)$ for all $m \in \{0,1\}^d$. We denote by $\X=\{x: \pX(x) > 0 \}$ the support of $X$ and by $\mathcal{M} = \{m \in \{0,1\}^d : \Prob(M=m) > 0 \}$ the support of $M$.

For all $m \in \mathcal{M}$, we also consider the distribution of $X$ in pattern $m$ denoted by $P_{X \mid M=m}$ with density $p(x \mid M=m)$. We also denote the marginal distribution of $X$ as $P_X$ with density $p(x)$.

For any $m \in \{0,1\}^d$ and any $x \in \mathds{R}^d$, we let 
$o(x,m)$ be the subvector of $x$ indexed by the components $j$ such that $m_j=0$ (i.e. the observed components of $x$ if $m$ is the missing pattern associated to $x$). Similarly, we denote by $o^c(x,m)$ 
the subvector of $x$ indexed by the components $j$ such that $m_j=1$ (i.e. the missing components of $x$ if $m$ is the missing pattern associated to $x$). Accordingly, we define $p(o^c(x,m') \mid o(X,m') = o(x,m'), M=m)$, the conditional density of $o^c(X,m')$ conditional on the components $o(X,m')=o(x,m')$ and on the missing pattern $M=m$. Moreover, for all $m, m' \in \mathcal{M}$, we denote the support of $P_{o(X,m) \mid M=m'}$ as $\X_{\mid m'}^{m}$, i.e. $\X_{\mid m'}^{m}=\{x \in \X: p(o(x,m) \mid M=m') > 0\}$.\footnote{To avoid measurability issues, we assume that $p/\pX$ are such that $x \in \X$ implies $\pX(o(x,m)) > 0$ and $x \notin \X_{\mid m'}^{m}$ implies $p(o(x,m) \mid M=m')=0$ for all $m,m'$.} For a set $\mathcal{M}_0 \subset \mathcal{M}$, we also denote $\X_{\mid \mathcal{M}_0}^{m}= \bigcup_{m' \in \mathcal{M}_0} \X_{\mid m'}^{m}$. Finally, regarding scoring imputations, we will take $\| \cdot\|_{2}$ to be the Euclidean distance on $\R^d$ and write expectations as $\E_{\substack{Y \sim H\\X \sim \PX}}[ \| X-Y \|_{2}]$, to clarify over which distributions the expectation is taken. 
Table \ref{Table:Summarization} summarizes the notation.

The following example illustrates the notations and motivates our upcoming discussion.

\begin{example}\label{Example1_first}\label{Example1}

Consider
\begin{align}\label{eq:Example1}
\mathcal{M}= \left\{ (0, 0, 0),  (0, 1, 0),  (1, 0, 0) \right \},
\end{align}
and assume that $X = (X_1, X_2, X_3)$ is uniformly distributed on $\X=[0,1]^3$. Let 
\begin{align*}
    \Prob(M=m_1 \mid X=x) &=  2x_1/3,\\
    \Prob(M=m_2 \mid X=x) &=  2/3-2x_1/3,\\
    \Prob(M=m_3 \mid X=x) &=  1/3.
\end{align*}
Consider now the task of imputing $X_1$. To accomplish this, we might want to learn $X_1 \mid X_{2}, X_3$ in pattern $m_1$ to impute in pattern $m_3$, thus it might make sense to compare the distributions of $X_1 \mid X_{2}, X_3, M=m_1$ and $X_1 \mid X_{2}, X_3, M=m_3$. First it holds that 
\begin{align*}
    &\pX(x_1 \mid X_2=x_2, X_3=x_3, M=m_3)\\
    &=\pX(x_1 \mid X_2=x_2, X_3=x_3).
\end{align*}
In fact, we will see below that this condition characterizes MAR, (\ref{PMMMAR} in Section \ref{subsec:Mardefinitions}). On the other hand,
\begin{align*}
    &p(x_1\mid X_2=x_2,X_3=x_3,  M=m_1)\\
    &=2x_1 \pX(x_1 \mid X_2=x_2, X_3=x_3)\\
    &\neq \pX(x_1 \mid X_2=x_2, X_3=x_3).
\end{align*}
This shows that $\pX(x_1 \mid X_2=x_2, X_3=x_3)$ needed for imputation, is not identifiable from pattern $m_1$. The same argument shows that $\pX(x_1 \mid X_2=x_2, X_3=x_3)$ is also not identifiable from pattern $m_2$. Figure \ref{fig:Example1} illustrates this behavior: It shows the distribution of $X_1$ in different patterns. As the distribution of $(X_2,X_3)$ in the different patterns is always the same, this directly illustrates the change in the conditional distribution of $X_1 \mid X_2, X_3$ when changing from pattern $m_1$ to pattern $m_3$. This example thus questions  whether nonparametric imputation under MAR is possible.
\end{example}

\begin{figure}
    \centering
    \includegraphics[width=1\linewidth]{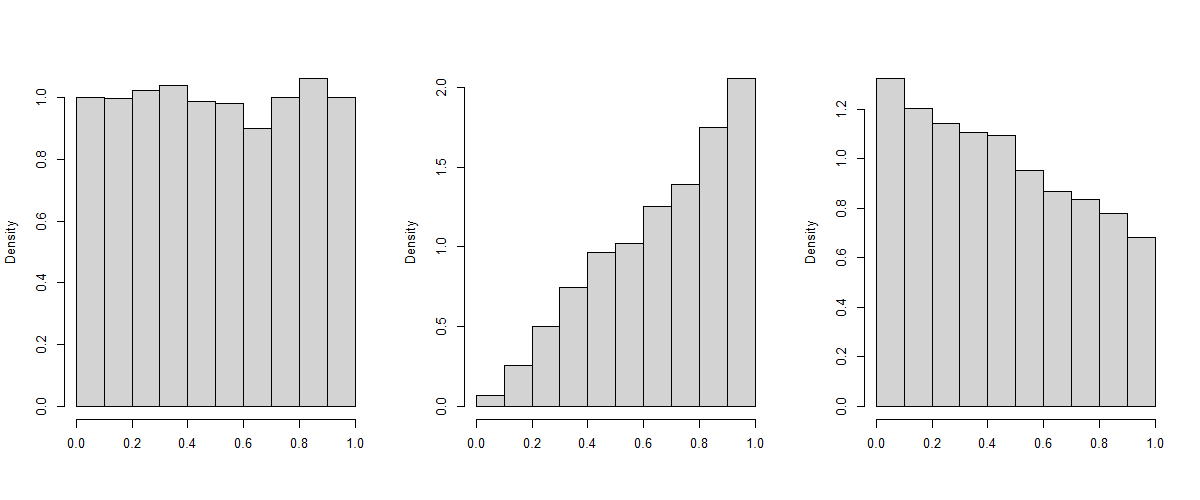}
    \caption{Illustration of Example \ref{Example1_first}. Left: Distribution we would like to impute $X_1 \mid M=m_3$. Middle: Distribution of $X_1$ in the fully observed pattern $(X_1 \mid M=m_1)$. Right: Distribution of $X_1$ in the second pattern $(X_1 \mid M=m_2)$.} 
    \label{fig:Example1}
\end{figure}

\begin{table*}[htbp]
\centering
\caption{Notation Summary for Missing Data Framework}
\label{Table:Summarization}
\begin{tabular}{cl}
\hline
\textbf{Symbol} & \textbf{Description} \\
\hline
$(X, M)$ & Complete data $X$ and missingness pattern $M$\\
$M_j$ & $j$th coordinate of $M$: $M_j = 1$ (masked), $M_j = 0$ (observed) \\
$P$ (resp. $p(x,m)$) & True unobserved joint distribution (resp. density) of $(X,M)$ \\
$(X_i, M_i)_{i=1}^n$ & $n$ samples from $P$ \\
$X_i^*$ & Masked observation: $X_{i,j}^* = \textrm{NA}$ if $M_{i,j} = 1$, $X_{i,j}^* = X_{i,j}$ if $M_{i,j} = 0$ \\
$P_X$ (resp. $p(x)$) & Marginal distribution (resp. density) of $X$ \\
$\mathbb{P}(M = m)$ & Marginal distribution of missingness pattern $M$ \\
$ \X$ (resp. $\mathcal{M}$) & Support of $X$ (resp. $M$)\\
$P_{X|M=m}$ (resp. $p(x \mid M = m)$) & Distribution (resp. density) of $X$ in pattern $m$ \\
$\mathbb{P}(M = m \mid X=x)$ & Probability of pattern $m$ given $x$ \\
$o^c(X, m')$ (resp. $o(X, m')$) & Missing (resp. Observed) components of $X$ under pattern $m'$\\
$p(o^c(x, m') \mid o(X, m')=o(x, m'), M = m)$ & Conditional density of missing part wrt to $m'$ given observed part wrt to $m'$ in pattern $m$ \\
$P_{o(X,m) \mid M=m'}$ & Distribution of observed part with respect to pattern $m$ in pattern $m'$ \\
$ \X_{\mid m'}^{m}$ & Support of $P_{o(X,m) \mid M=m'}$: $\{x \in  \X : p(o(x,m) \mid M = m') > 0\}$ \\
$ \X_{\mid \mathcal{M}_0}^{m}$ & Union of supports: $\bigcup_{m' \in \mathcal{M}_0}  \X_{\mid m'}^{m}$ for $\mathcal{M}_0 \subset \mathcal{M}$ \\
\hline
\end{tabular}
\end{table*}

\section{Sequential Imputation under MAR}\label{Sec_Main_Results}


In this section, we present different definitions of MAR used in the literature, with a focus on imputation. We first study the exact relations between all these definitions, as summarized in Figure~\ref{fig_relation_between_MAR_assumptions}. We then show that if the number of missing values of a pattern $m$ is larger than one, identifying the imputation distribution directly from other patterns is generally not possible under the common MAR definition, which we will formally introduce as \ref{PMMMAR} below. 
However, in Section \ref{Sec_FCSunderMAR}, we show that identifying this imputation distribution is theoretically possible for \ref{PMMMAR} if one variable at a time is imputed, as is done in the FCS approach. Finally, we discuss the issue of unconditional distribution shifts and introduce an overlap condition. Figures \ref{fig_relation_between_MAR_assumptions}, \ref{fig:scoreillustrationOj} and Table \ref{tab:missing_data} each summarize a key takeaway from this section.

\subsection{MAR Definitions}
\label{subsec:Mardefinitions}

In this section, we analyze several different MAR conditions. We present two different settings, the selection model (SM) and the pattern-mixture model (PMM), each one leading to a different set of MAR assumptions.

\paragraph*{Selection Model.} In the \emph{selection model} framework \citep{little_patternmixture}, the joint distribution of $X$ and $M$ is factored as $p(x,m)= \Prob(M=m \mid X=x)   \pX(x)$. 
In this setting, MAR is defined as follows.

\begin{definition}[SM-MAR]    The missingness mechanism is missing at random (MAR) if, for all $m \in \mathcal{M}$, and all $x, \tilde{x} \in \X$ such that $o(x,m)=o(\tilde{x},m)$, we have
    \begin{align}
&\Prob(M=m | X=x) = \Prob(M=m| X=\tilde{x}). \tag{SM-MAR} \label{SMAR}
\end{align}
\end{definition}

\ref{SMAR} is sometimes referred to as ``Always Missing at Random'' \citep[see, e.g.,][]{whatismeant3, directcompetitor1}. A widely used alternative definition of MAR (see e.g., \cite{ourresult, little2019statistical}) is the following.
\begin{definition}[SM-MAR II]
    The missingness mechanism is missing at random (MAR) if,  for all $m \in \mathcal{M}, x\in  \X$, we have
\begin{align}
\Prob(M=m \mid X=x) = \Prob(M=m \mid o(X,m)= o(x,m)). \tag{SM-MAR II} \label{SMARII} 
\end{align}
\end{definition}
 Note that $o(x,m)$ is different for each $m$, and thus neither \ref{SMAR} nor \ref{SMARII} are statements about conditional independence between $M$ and $X$, as remarked in \cite{whatismeant3}. However, as observed by a referee, \ref{SMARII} might be formulated as $\Ind\{M=m\} \indep ~X \mid o(X,m)$ for all $m \in \mathcal{M}$. As such, \ref{SMARII} is intuitive: for any value of $m$, the probability of occurrence of missing pattern $m$  only depends on the observed part of $x$. We can verify that these two definitions are equivalent. 

\begin{restatable}{lemma}{PMMMARzero}\label{amazinglemma}
Condition \ref{SMAR} is equivalent to \ref{SMARII}.
\end{restatable}

\paragraph*{Pattern Mixture Model.} We now turn to the  \emph{pattern-mixture model} (PMM) framework \citep{little_patternmixture}, which is based on the following decomposition $
p(x,m)=p(x\mid M=m) \Prob(M=m)$. The PMM view emphasizes that the conditional distribution of the complete vector $X \mid M=m$ may vary across $m \in \mathcal{M}$. Consequently, learning a distribution of a given pattern $m$ based on another pattern $m'$ can be challenging, as the two distributions may differ drastically. 
A typical example is the Gaussian pattern-mixture model, where  $X \mid M=m \sim N(\mu_m , \Sigma_m),$ so that the distribution in each pattern might follow a different Gaussian distribution. In the PMM setting, \cite{ourresult} proposed the following definition. 

\begin{definition}[PMM-MAR]
    The missingness me-chanism is missing at random (MAR) if, for all $m \in \mathcal{M}$ and $x \in \X_{\mid m}^{m} $, 
\begin{align}
&p(o^c(x ,m ) \mid o(X,m)=o(x ,m ), M =m ) \tag{PMM-MAR}\label{PMMMAR} \\
&= \pX(o^c(x ,m )\mid o(X,m)=o(x ,m )). \nonumber 
\end{align}
\end{definition}

Thus the conditional distribution of the missing part given the observed part $o^c(X ,m ) \mid o(X ,m )$ in pattern $m$ is equal to the conditional distribution, when no information about the pattern is available. 


\begin{restatable}{proposition}{PMMMAR}[\cite{ourresult}]\label{amazinglemma0}
Condition \ref{SMARII} is equivalent to \ref{PMMMAR}.
\end{restatable}

\Cref{amazinglemma0} follows immediately from the fact that \ref{SMARII} is equivalent to $\Ind\{M=m\} \indep X \mid o(X,m)$ for all $m \in \mathcal{M}$.
A stronger, but more interpretable condition, than \ref{PMMMAR} is the conditionally independent MAR (CIMAR), introduced in \cite{directcompetitor1}. 
\begin{definition}
    The missingness mechanism is conditionally independent MAR (CIMAR) if, for all $m, m' \in \mathcal{M}$ and $x \in \X_{\mid m'}^{m}$, we have
    \begin{align}
&p(o^c(x ,m ) \mid o(X ,m )=o(x ,m ), M =m' )\nonumber \\
&= \pX(o^c(x ,m )\mid o(X ,m )=o(x ,m )) \tag{CIMAR} \label{CIMAR}.  
\end{align}
\end{definition}
Contrary to all previous assumptions, \ref{CIMAR} is a conditional independence statement involving $X$ and $M$, as opposed to the indicator functions $\Ind\{M=m\}$, namely that $o^c(X,m) \mid o(X,m)$ is independent of $M$: the distribution of $o^c(X,m) \mid o(X,m)$ remains the same, for all missing patterns $M = m'$. Thus, \ref{CIMAR} allows to learn the distribution of $o^c(X,m) \mid o(X,m)$ from any pattern $m'$ \citep[see][for an application]{ImputationScores}. \ref{CIMAR} might seem unreasonably strong compared to \ref{PMMMAR}. However, we note that it is actually the PMM equivalent to the following prominent condition. Assume $X=(X_{O^c}, X_{O})$, where $O$ is the set of indices of fully observed variables:
\begin{align}\label{Odef}
    O=\{j: m_j = 0 \text{ for all } m \in \mathcal{M} \},
\end{align}
and that for all $x \in  \X$, $m \in \mathcal{M}$:
\begin{align}\label{RMAR}
    \Prob(M=m \mid X=x)=\Prob(M=m \mid X_O= x_{O})\tag{RMAR}.
\end{align}
This is called ``Realistic MAR'' (RMAR) by \cite{RMAR_Psychometrika}, ``MAR+'' by \cite{RMAR2006} and simply ``MAR'' by \cite{Mohan2013, mohan2021graphical}. We note that in this definition we may allow $\Prob(M=m \mid X=x)$ to depend on a subset of $X_{O}$. \Cref{CIMARequivRMAR} below shows that if the missingness mechanism is \ref{RMAR} and there exists a set of variables that is always observed, it is also \ref{CIMAR} and vice-versa.

\ref{CIMAR} in turn is also weaker than MCAR, which requires that  for all  $m \in \mathcal{M}, m' \in \mathcal{M}, x \in  \X$, 
    \begin{align}\label{MCARform}
     p(x \mid M=m ) = p(x \mid M=m' ) 
     \tag{PMM-MCAR}.
\end{align}


\begin{figure*}[h!]
    \centering
\begin{tikzpicture}

\node at (0,0) {$\mathbf{X} = \begin{pmatrix}
x_{1,1} & x_{1,2} & x_{1,3} \\
x_{2,1} & x_{2,2} & x_{2,3} \\
x_{3,1} & x_{3,1} & x_{3,3}
\end{pmatrix}$};

\node at (4,0) {$\mathbf{M} = \begin{pmatrix}
0 & 0 & 0 \\
1 & 0 & 0 \\
1 & 1 & 0
\end{pmatrix}$};

\node at (8,0) {$\mathbf{X}^* = \begin{pmatrix}
x_{1,1} & x_{1,2} & x_{1,3} \\
\textrm{NA} & x_{2,2} & x_{2,3} \\
\textrm{NA} & \textrm{NA} & x_{3,3}
\end{pmatrix}$};

\end{tikzpicture}
    \caption{$\mathbf{X}$ is the assumed underlying full data, $\mathbf{M}$ is the vector of missing indicators and $\mathbf{X}^*$ arises when $\mathbf{M}$ is applied to $\mathbf{X}$. Thus each row of $\mathbf{X}$ (or $\mathbf{X}^*$) is an observation under a different pattern. Under condition \ref{CIMAR}, the distribution of $X_1, X_2 \mid X_3$ is not allowed to change when moving from one pattern to another, though the marginal distribution of $X_3$ is allowed to change. In contrast, under MCAR \eqref{MCARform}, no change is allowed. Under MAR \eqref{PMMMAR}, the only constraint is that the distribution of $X_1, X_2 \mid X_3$ in the third pattern is the same as the unconditional one.}
    \label{fig:illustrationfocond}
\end{figure*}

Figure \ref{fig:illustrationfocond} illustrates these different conditions in a small example.

Another interesting MAR condition is the extended MAR condition, introduced in \cite{directcompetitor1}. For this condition, we need to assume that the fully observed pattern is part of $\mathcal{M}$, which we write as $0 \in \mathcal{M}$.
\begin{definition}
Assume $0 \in \mathcal{M}$. The missingness mechanism is Extended Missing At Random (EMAR), if, \ref{PMMMAR} holds and for $x \in \X_{\mid 0}^{m}$, we have
    \begin{align} 
    &p(o^c(x ,m ) \mid o(X ,m )=o(x ,m ), M =0 ) \nonumber \\
    &= \pX(o^c(x ,m )\mid o(X ,m )=o(x ,m ))\tag{EMAR} \label{EMAR}.
\end{align}
\end{definition}
\ref{EMAR} can be rewritten as \begin{align}
& p(o^c(x ,m ) \mid o(X,m)=o(x ,m ), M =m ) \nonumber \\
= & ~p(o^c(x ,m )\mid o(X,m)=o(x ,m ), M =0 ).  \nonumber  
\end{align}
Clearly, \ref{CIMAR} implies \ref{EMAR} and Example \ref{Example6} in Appendix \ref{Sec_Proofs} demonstrates that \ref{EMAR} is strictly weaker than \ref{CIMAR}. On the other hand, Example \ref{Example1_first} above or Example \ref{interesting_new_Example} below show that \ref{EMAR} is strictly stronger than \ref{PMMMAR}. Thus we have the following.

\begin{restatable}{proposition}{Ordering}\label{Ordering}
    Assume $0 \in \mathcal{M}$. Then \ref{CIMAR} is strictly stronger than \ref{EMAR} which is strictly stronger than \ref{PMMMAR}.
\end{restatable}

Again Proposition \ref{Ordering} might be derived from the fact that we can characterize \ref{RMAR}/\ref{CIMAR} as $M \indep ~ X \mid X_O $   and \ref{EMAR} as for all $m$,$(\Ind\{M=m\}, \Ind\{M=0\}) \indep ~ X \mid o(X,m)$.
We can also establish the following relation between \ref{MCARform}, \ref{CIMAR}, and \ref{RMAR}.

\begin{restatable}{lemma}{CIMARequivRMAR}\label{CIMARequivRMAR}
   Assume $O$ in \eqref{Odef} is not empty. Then \ref{CIMAR} is equivalent to \ref{RMAR} and both are strictly weaker than \ref{MCARform}. If in addition, $0 \in \mathcal{M}$, \ref{EMAR} is equivalent to the following two conditions: \ref{SMARII} and for all $x \in \X$,
   \begin{align}\label{RMARforEMAR}
       \Prob(M=0 \mid X=x)=\Prob(M=0 \mid X_O=x_O).
   \end{align}
   If $O$ is empty, \ref{CIMAR} and \ref{RMAR} simplify to \ref{MCARform}, while $\Prob(M=0 \mid X=x)=\Prob(M=0)$ for \ref{EMAR}.
\end{restatable}

\Cref{fig_relation_between_MAR_assumptions} summarizes the different implications between the above MAR assumptions.

\begin{figure}[h!]
    \centering

\tikzset{every picture/.style={line width=0.75pt}} 

\scalebox{0.8}{\begin{tikzpicture}[x=0.75pt,y=0.75pt,yscale=-1,xscale=1]

\draw   (121,162) .. controls (121,152.61) and (153.46,145) .. (193.5,145) .. controls (233.54,145) and (266,152.61) .. (266,162) .. controls (266,171.39) and (233.54,179) .. (193.5,179) .. controls (153.46,179) and (121,171.39) .. (121,162) -- cycle ;
\draw   (339,163) .. controls (339,153.61) and (371.46,146) .. (411.5,146) .. controls (451.54,146) and (484,153.61) .. (484,163) .. controls (484,172.39) and (451.54,180) .. (411.5,180) .. controls (371.46,180) and (339,172.39) .. (339,163) -- cycle ;
\draw    (266,156) -- (337,156.97) ;
\draw [shift={(339,157)}, rotate = 180.78] [color={rgb, 255:red, 0; green, 0; blue, 0 }  ][line width=0.75]    (10.93,-3.29) .. controls (6.95,-1.4) and (3.31,-0.3) .. (0,0) .. controls (3.31,0.3) and (6.95,1.4) .. (10.93,3.29)   ;
\draw    (338,168) -- (267,167.03) ;
\draw [shift={(265,167)}, rotate = 0.78] [color={rgb, 255:red, 0; green, 0; blue, 0 }  ][line width=0.75]    (10.93,-3.29) .. controls (6.95,-1.4) and (3.31,-0.3) .. (0,0) .. controls (3.31,0.3) and (6.95,1.4) .. (10.93,3.29)   ;
\draw   (117,240) .. controls (117,230.61) and (149.46,223) .. (189.5,223) .. controls (229.54,223) and (262,230.61) .. (262,240) .. controls (262,249.39) and (229.54,257) .. (189.5,257) .. controls (149.46,257) and (117,249.39) .. (117,240) -- cycle ;
\draw   (118,315) .. controls (118,305.61) and (150.46,298) .. (190.5,298) .. controls (230.54,298) and (263,305.61) .. (263,315) .. controls (263,324.39) and (230.54,332) .. (190.5,332) .. controls (150.46,332) and (118,324.39) .. (118,315) -- cycle ;
\draw    (173,223) -- (173,182) ;
\draw [shift={(173,180)}, rotate = 90] [color={rgb, 255:red, 0; green, 0; blue, 0 }  ][line width=0.75]    (10.93,-3.29) .. controls (6.95,-1.4) and (3.31,-0.3) .. (0,0) .. controls (3.31,0.3) and (6.95,1.4) .. (10.93,3.29)   ;
\draw    (174,300) -- (174,259) ;
\draw [shift={(174,257)}, rotate = 90] [color={rgb, 255:red, 0; green, 0; blue, 0 }  ][line width=0.75]    (10.93,-3.29) .. controls (6.95,-1.4) and (3.31,-0.3) .. (0,0) .. controls (3.31,0.3) and (6.95,1.4) .. (10.93,3.29)   ;
\draw    (201,178) -- (201.95,220) ;
\draw [shift={(202,222)}, rotate = 268.7] [color={rgb, 255:red, 0; green, 0; blue, 0 }  ][line width=0.75]    (10.93,-3.29) .. controls (6.95,-1.4) and (3.31,-0.3) .. (0,0) .. controls (3.31,0.3) and (6.95,1.4) .. (10.93,3.29)   ;
\draw    (203,256) -- (203.95,298) ;
\draw [shift={(204,300)}, rotate = 268.7] [color={rgb, 255:red, 0; green, 0; blue, 0 }  ][line width=0.75]    (10.93,-3.29) .. controls (6.95,-1.4) and (3.31,-0.3) .. (0,0) .. controls (3.31,0.3) and (6.95,1.4) .. (10.93,3.29)   ;
\draw   (188,200) .. controls (188,194.75) and (194.04,190.5) .. (201.5,190.5) .. controls (208.96,190.5) and (215,194.75) .. (215,200) .. controls (215,205.25) and (208.96,209.5) .. (201.5,209.5) .. controls (194.04,209.5) and (188,205.25) .. (188,200) -- cycle ; \draw   (191.95,193.28) -- (211.05,206.72) ; \draw   (211.05,193.28) -- (191.95,206.72) ;
\draw   (190,278) .. controls (190,272.75) and (196.04,268.5) .. (203.5,268.5) .. controls (210.96,268.5) and (217,272.75) .. (217,278) .. controls (217,283.25) and (210.96,287.5) .. (203.5,287.5) .. controls (196.04,287.5) and (190,283.25) .. (190,278) -- cycle ; \draw   (193.95,271.28) -- (213.05,284.72) ; \draw   (213.05,271.28) -- (193.95,284.72) ;
\draw   (118,392) .. controls (118,382.61) and (150.46,375) .. (190.5,375) .. controls (230.54,375) and (263,382.61) .. (263,392) .. controls (263,401.39) and (230.54,409) .. (190.5,409) .. controls (150.46,409) and (118,401.39) .. (118,392) -- cycle ;

\draw   (341,240) .. controls (341,230.61) and (373.46,223) .. (413.5,223) .. controls (453.54,223) and (486,230.61) .. (486,240) .. controls (486,249.39) and (453.54,257) .. (413.5,257) .. controls (373.46,257) and (341,249.39) .. (341,240) -- cycle ;

\draw    (400,180) -- (400,221) ;
\draw [shift={(400,223)}, rotate = 270] [color={rgb, 255:red, 0; green, 0; blue, 0 }  ][line width=0.75]    (10.93,-3.29) .. controls (6.95,-1.4) and (3.31,-0.3) .. (0,0) .. controls (3.31,0.3) and (6.95,1.4) .. (10.93,3.29)   ;

\draw    (430,222) -- (430,182) ;
\draw [shift={(430,180)}, rotate = 90] [color={rgb, 255:red, 0; green, 0; blue, 0 }  ][line width=0.75]    (10.93,-3.29) .. controls (6.95,-1.4) and (3.31,-0.3) .. (0,0) .. controls (3.31,0.3) and (6.95,1.4) .. (10.93,3.29)   ;
\draw   (416.74,200.01) .. controls (416.74,194.76) and (422.78,190.51) .. (430.24,190.51) .. controls (437.69,190.51) and (443.74,194.76) .. (443.74,200.01) .. controls (443.74,205.25) and (437.69,209.51) .. (430.24,209.51) .. controls (422.78,209.51) and (416.74,205.25) .. (416.74,200.01) -- cycle ; \draw   (420.69,193.29) -- (439.78,206.72) ; \draw   (439.78,193.29) -- (420.69,206.72) ;

\draw   (341,316) .. controls (341,306.61) and (373.46,299) .. (413.5,299) .. controls (453.54,299) and (486,306.61) .. (486,316) .. controls (486,325.39) and (453.54,333) .. (413.5,333) .. controls (373.46,333) and (341,325.39) .. (341,316) -- cycle ;

\draw    (400,257) -- (400,298) ;
\draw [shift={(400,300)}, rotate = 270] [color={rgb, 255:red, 0; green, 0; blue, 0 }  ][line width=0.75]    (10.93,-3.29) .. controls (6.95,-1.4) and (3.31,-0.3) .. (0,0) .. controls (3.31,0.3) and (6.95,1.4) .. (10.93,3.29)   ;

\draw    (430,299) -- (430,258) ;
\draw [shift={(430,256)}, rotate = 90] [color={rgb, 255:red, 0; green, 0; blue, 0 }  ][line width=0.75]    (10.93,-3.29) .. controls (6.95,-1.4) and (3.31,-0.3) .. (0,0) .. controls (3.31,0.3) and (6.95,1.4) .. (10.93,3.29)   ;
\draw   (416.74,278) .. controls (416.74,272.75) and (422.78,268.5) .. (430.24,268.5) .. controls (437.69,268.5) and (443.74,272.75) .. (443.74,278) .. controls (443.74,283.25) and (437.69,287.5) .. (430.24,287.5) .. controls (422.78,287.5) and (416.74,283.25) .. (416.74,278) -- cycle ; \draw   (420.69,271.28) -- (439.78,284.72) ; \draw   (439.78,271.28) -- (420.69,284.72) ;

\draw   (338,392) .. controls (338,382.61) and (370.46,375) .. (410.5,375) .. controls (450.54,375) and (483,382.61) .. (483,392) .. controls (483,401.39) and (450.54,409) .. (410.5,409) .. controls (370.46,409) and (338,401.39) .. (338,392) -- cycle ;

\draw    (400,333) -- (400,374) ;
\draw [shift={(400,376)}, rotate = 270] [color={rgb, 255:red, 0; green, 0; blue, 0 }  ][line width=0.75]    (10.93,-3.29) .. controls (6.95,-1.4) and (3.31,-0.3) .. (0,0) .. controls (3.31,0.3) and (6.95,1.4) .. (10.93,3.29)   ;

\draw    (430,374) -- (430,334) ;
\draw [shift={(430,332)}, rotate = 90] [color={rgb, 255:red, 0; green, 0; blue, 0 }  ][line width=0.75]    (10.93,-3.29) .. controls (6.95,-1.4) and (3.31,-0.3) .. (0,0) .. controls (3.31,0.3) and (6.95,1.4) .. (10.93,3.29)   ;
\draw   (416.74,354) .. controls (416.74,348.75) and (422.78,344.5) .. (430.24,344.5) .. controls (437.69,344.5) and (443.74,348.75) .. (443.74,354) .. controls (443.74,359.25) and (437.69,363.5) .. (430.24,363.5) .. controls (422.78,363.5) and (416.74,359.25) .. (416.74,354) -- cycle ; \draw   (420.69,347.28) -- (439.78,360.72) ; \draw   (439.78,347.28) -- (420.69,360.72) ;

\draw    (263,386) -- (334,386.97) ;
\draw [shift={(336,387)}, rotate = 180.78] [color={rgb, 255:red, 0; green, 0; blue, 0 }  ][line width=0.75]    (10.93,-3.29) .. controls (6.95,-1.4) and (3.31,-0.3) .. (0,0) .. controls (3.31,0.3) and (6.95,1.4) .. (10.93,3.29)   ;
\draw    (337,398) -- (266,397.03) ;
\draw [shift={(264,397)}, rotate = 0.78] [color={rgb, 255:red, 0; green, 0; blue, 0 }  ][line width=0.75]    (10.93,-3.29) .. controls (6.95,-1.4) and (3.31,-0.3) .. (0,0) .. controls (3.31,0.3) and (6.95,1.4) .. (10.93,3.29)   ;
\draw    (175,374) -- (175,333) ;
\draw [shift={(175,331)}, rotate = 90] [color={rgb, 255:red, 0; green, 0; blue, 0 }  ][line width=0.75]    (10.93,-3.29) .. controls (6.95,-1.4) and (3.31,-0.3) .. (0,0) .. controls (3.31,0.3) and (6.95,1.4) .. (10.93,3.29)   ;
\draw    (202,331) -- (202.48,352.01) -- (202.95,373) ;
\draw [shift={(203,375)}, rotate = 268.7] [color={rgb, 255:red, 0; green, 0; blue, 0 }  ][line width=0.75]    (10.93,-3.29) .. controls (6.95,-1.4) and (3.31,-0.3) .. (0,0) .. controls (3.31,0.3) and (6.95,1.4) .. (10.93,3.29)   ;
\draw   (188.74,351.01) .. controls (188.74,345.76) and (194.78,341.51) .. (202.24,341.51) .. controls (209.69,341.51) and (215.74,345.76) .. (215.74,351.01) .. controls (215.74,356.25) and (209.69,360.51) .. (202.24,360.51) .. controls (194.78,360.51) and (188.74,356.25) .. (188.74,351.01) -- cycle ; \draw   (192.69,344.29) -- (211.78,357.72) ; \draw   (211.78,344.29) -- (192.69,357.72) ;
\draw    (267,310) -- (338,310.97) ;
\draw [shift={(340,311)}, rotate = 180.78] [color={rgb, 255:red, 0; green, 0; blue, 0 }  ][line width=0.75]    (10.93,-3.29) .. controls (6.95,-1.4) and (3.31,-0.3) .. (0,0) .. controls (3.31,0.3) and (6.95,1.4) .. (10.93,3.29)   ;
\draw    (339,322) -- (268,321.03) ;
\draw [shift={(266,321)}, rotate = 0.78] [color={rgb, 255:red, 0; green, 0; blue, 0 }  ][line width=0.75]    (10.93,-3.29) .. controls (6.95,-1.4) and (3.31,-0.3) .. (0,0) .. controls (3.31,0.3) and (6.95,1.4) .. (10.93,3.29)   ;

\draw    (262,234) -- (339,234.97) ;
\draw [shift={(341,235)}, rotate = 180.78] [color={rgb, 255:red, 0; green, 0; blue, 0 }  ][line width=0.75]    (10.93,-3.29) .. controls (6.95,-1.4) and (3.31,-0.3) .. (0,0) .. controls (3.31,0.3) and (6.95,1.4) .. (10.93,3.29)   ;

\draw    (340,246) -- (263,245.03) ;
\draw [shift={(261,245)}, rotate = 0.78] [color={rgb, 255:red, 0; green, 0; blue, 0 }  ][line width=0.75]    (10.93,-3.29) .. controls (6.95,-1.4) and (3.31,-0.3) .. (0,0) .. controls (3.31,0.3) and (6.95,1.4) .. (10.93,3.29)   ;

\draw (149,154) node [anchor=north west][inner sep=0.75pt]   [align=left] { \ \ \ref{PMMMAR}};
\draw (378,155) node [anchor=north west][inner sep=0.75pt]   [align=left] {\ref{SMAR}};
\draw (164,232) node [anchor=north west][inner sep=0.75pt]   [align=left] {\ref{EMAR}};
\draw (167,306) node [anchor=north west][inner sep=0.75pt]   [align=left] {\ref{CIMAR}};
\draw (222,191) node [anchor=north west][inner sep=0.75pt]   [align=left] {Example \ref{Example1}};
\draw (222,270) node [anchor=north west][inner sep=0.75pt]   [align=left] {Example \ref{Example6}};
\draw (152,384) node [anchor=north west][inner sep=0.75pt]   [align=left] {\ref{MCARform}};
\draw (373,384) node [anchor=north west][inner sep=0.75pt]   [align=left] {SM-MCAR};
\draw (392,306) node [anchor=north west][inner sep=0.75pt]   [align=left] {\ref{RMAR}};
\draw (263,330) node [anchor=north west][inner sep=0.75pt]   [align=left] {Lemma \ref{CIMARequivRMAR}};
\draw (263,250) node [anchor=north west][inner sep=0.75pt]   [align=left] {Lemma \ref{CIMARequivRMAR}};

\draw (263,170) node [anchor=north west][inner sep=0.75pt]   [align=left] {Proposition \ref{amazinglemma0}};

\draw (365,232) node [anchor=north west][inner sep=0.75pt]   [align=left] {\ref{PMMMAR}+\eqref{RMARforEMAR}};

\end{tikzpicture}}

\caption{Relationships between the MAR conditions discussed in this paper. An arrow from condition $A$ to condition $B$, encodes that $A$ implies $B$. The definitions are given in Section \ref{subsec:Mardefinitions}.}
\label{fig_relation_between_MAR_assumptions}
\end{figure}

In the presence of missing data, one may resort to imputation, to approximately recover a sample from $\PX$. To impute correctly, one needs to determine the distribution of $o^c(X,m) \mid o(X,m) $, which can be used to impute the missing components $o^c(X,m) \mid M=m$. Clearly under \ref{CIMAR} $o^c(X,m) \mid o(X,m) $ can in principle be determined from any other pattern $m'$, while under \ref{EMAR} this is possible for $m'=0$. On the other hand, it appears not immediately clear what is needed to identify the true conditional distribution under \ref{PMMMAR}. We formalize these insights in the next section.

\subsection{Identifiability under MAR}\label{Sec_FCSunderMAR}

A crucial property of all three MAR definitions (PMM-MAR, CIMAR, EMAR) presented in Section \ref{subsec:Mardefinitions} is that 
\begin{align*}
    &p(o^c(x,m) \mid o(X,m)=o(x,m), M=m)\\
    &=\pX(o^c(x,m) \mid o(X,m)=o(x,m)).
\end{align*}
Thus to impute pattern $m$ successfully, one needs to learn $\pX(o^c(x,m) \mid o(X,m)=o(x,m))$. We are now able to summarize the three different MAR definitions in one result about the ability to identify $\pX(o^c(x,m) \mid o(X,m)=o(x,m))$. We say $\pX(o^c(x,m) \mid o(X,m)=o(x,m))$ is identifiable from a set $\mathcal{M}_0 \subset \mathcal{M}$ of missing patterns, if there exists $(w_{m'}(o(x,m)))_{m' \in \mathcal{M}_0}$, with $\sum_{m' \in \mathcal{M}_0} w_{m'}(o(x,m))=1$, such that the mixture
\begin{align}\label{h0star}
    &h^*(o^c(x,m) \mid o(X,m)=o(x,m))  \\
    &= \sum_{m' \in \mathcal{M}_0} w_{m'}(o(x,m)) \nonumber \\
    & p(o^c(x,m) \mid o(X,m)=o(x,m), M=m'), \nonumber 
\end{align}
satisfies,  for all $x \in \X_{\mid \mathcal{M}_0}^{m}$\footnote{We take the convention that $w_{m'}(o(x,m))=0$, whenever $p(o(x,m)\mid M=m')=0$.}, 
\begin{align}
&\pX(o^c(x,m) \mid o(X,m)=o(x,m)) \nonumber \\
=& ~h^*(o^c(x,m) \mid o(X,m)=o(x,m)). \nonumber    
\end{align} In particular, we say that $\pX(o^c(x,m) \mid o(X,m)=o(x,m))$ is identifiable from a pattern $m' \in \mathcal{M}$, if for all $x \in \X_{\mid m'}^{m}$ 
\begin{align}
& p(o^c(x,m) \mid o(X,m)=o(x,m), M=m') \nonumber \\
=& ~\pX(o^c(x,m) \mid o(X,m)=o(x,m)). \nonumber     
\end{align}
Define in the following
\begin{align*}
    L_{m}=\{m' \in \mathcal{M}: m'_j=0 \text{ for all } j \text{ such that } m_j=1 \}. 
\end{align*}
Thus $L_m$ is the set of patterns $m'$ for which $o^c(x,m) $ is observed.


\begin{restatable}{proposition}{propidentifiabilitymar}\label{prop_identifiability_mar}
   Assume $0 \in \mathcal{M}$ and that $|\mathcal{M}| > 3$. Then for any pattern $m \in \mathcal{M}$, $\pX(o^c(x,m) \mid o(X,m)=o(x,m))$ is  
    \begin{itemize}
        \item[-] identifiable from any other pattern $m' \in L_m$ under \ref{CIMAR},
        \item[-] identifiable from the pattern of fully observed data, $0 \in L_m$, under \ref{EMAR},
        \item[-] is not always identifiable from any single pattern $m' \in L_m$ under \ref{PMMMAR}.
    \end{itemize}
In addition, if $\sum_{j=1}^d m_j  > 1$, $p(o^c(x,m) \mid o(X,m)=o(x,m))$ is not always identifiable from $L_{m}$ under \ref{PMMMAR}.
\end{restatable}
To appreciate Proposition \ref{prop_identifiability_mar}, imagine for any pattern $m'$ in which $o^c(x,m)$ is observed, $X$ was already correctly imputed, such that $X\mid M=m'$ is available for all $m' \in L_m$. In this case, it would still not be possible to identify $\pX(o^c(x,m) \mid o(X, m)=o(x,m))$ under \ref{PMMMAR}, as no mixture of the conditional distribution $p(o^c(x, m) \mid o(X, m)=o(x, m), M=m')$ will recover the correct distribution. This is related to the fact that 
\ref{PMMMAR} still allows for a change in the conditional distributions over different patterns, as the following example illustrates.

\begin{example}\label{interesting_new_Example}
    Consider  $X = (X_1,X_2,X_3)$, where each $X_1, X_2, X_3$ is  marginally uniformly distributed over $[0,1]$, with $(X_1,X_2)$ being dependent and $X_3$ being independent of $(X_1, X_2)$. Moreover, we consider 
    \begin{align*}
        \mathcal{M}&=\{m_1,m_2, m_3, m_4\}\\
        &=\{(0,0,0), (0,1,0), (0,0,1), (1,1,0)\},
    \end{align*}
 and
        \begin{align*}
        &\Prob(M=m_1\mid X=x)=(x_1+x_2)/3,\\ 
        &\Prob(M=m_2\mid X=x)= (1-x_1)/3, \\
        & \Prob(M=m_3\mid X=x)= (1-x_2)/3, \\ 
        &\Prob(M=m_4\mid X=x)= 1/3. 
    \end{align*}
The details of this example are given in Appendix \ref{Sec_AdditionalexampleDetails}. First, this example clearly meets \ref{SMARII}, but not \ref{EMAR} or \ref{CIMAR}. Let us now consider $m=m_4$, such that $L_m=L_{m_4}=\{m_1,m_3\}$. Then $o^c(x,m_4)=(x_1,x_2)$ and $o(x,m_4)=x_3$ and it holds that 
\begin{align*}
    &p(x_1,x_2 \mid X_3=x_3, M=m_1)\\
    = &~ \pX(x_1,x_2 \mid X_3=x_3) (x_1+x_2)/2, \\
    \quad &p(x_1,x_2 \mid X_3=x_3, M=m_3)\\
    = & ~\pX(x_1,x_2 \mid X_3=x_3) (1-x_2)/2,
\end{align*}
so that there is no convex combination such that, for all $x_1, x_2$,  \begin{align}
& \pX(x_1,x_2 \mid x_3) \nonumber \\
= & ~w_{m_1}(x_3)p(x_1,x_2 \mid X_3=x_3, M=m_1) \nonumber\\
 & + w_{m_3}(x_3)p(x_1,x_2 \mid X_3=x_3, M=m_3). \nonumber 
\end{align}
This is also true if there is dependence between $X_1,X_2$ and $X_3$. Thus we cannot learn $\pX(x_1, x_2 \mid X_3=x_3)$ from the patterns where $x_1$, $x_2$ are observed. This is illustrated in Figure \ref{NewExampleIllustration}. While the marginal distribution $\pX(x_1)=\pX(x_1 \mid X_3=x_3)$ and $\pX(x_2)=\pX(x_2 \mid X_3=x_3)$ are uniform, Figure \ref{NewExampleIllustration} illustrates that $\pX(x_1 \mid M \in L_{m_4})$ and $\pX(x_2 \mid M \in L_{m_4})$ are not. 
\end{example}

\begin{figure}
    \centering
    \includegraphics[width=1\linewidth]{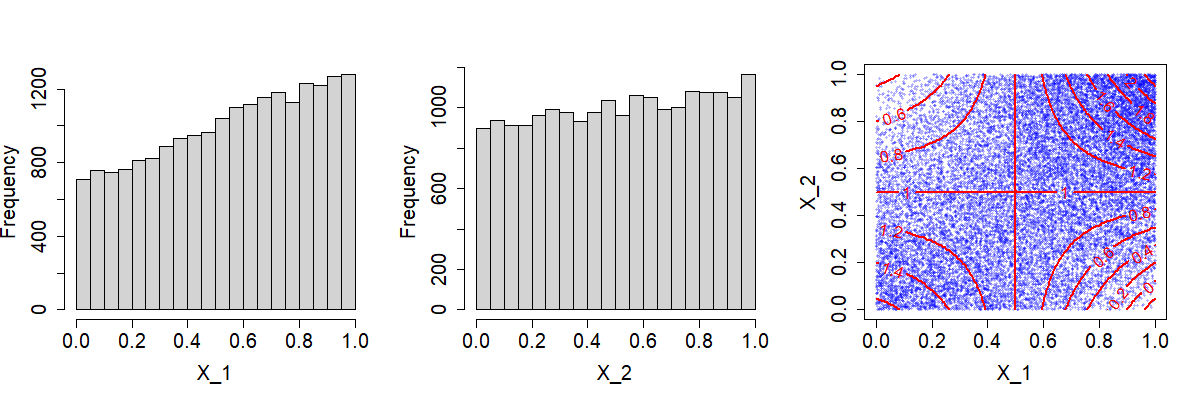}
    \caption{Illustration of the distribution of $(X_1, X_2) \mid M \in L_{m_4}$, Left: Univariate distribution of $X_1 \mid M \in L_{m_4}$, Middle: Univariate distribution of $X_2 \mid M \in L_{m_4}$, Right: Joint distribution with density overlay.}
    \label{NewExampleIllustration}
\end{figure}
On the contrary, under \ref{EMAR}, we are able to impute $o^c(X,m)$ based only on the distribution of the complete input vector. However, such a practice may be difficult due to the low number of complete observations. At the price of the more stringent \ref{CIMAR} assumption, we are able to build correct imputations by leveraging the information of all missing patterns.

Note that the negative result for \ref{PMMMAR} in Proposition \ref{prop_identifiability_mar} assumed $\sum_{j=1}^d m_j > 1$, i.e. the dimension of $o^c(X,m)$ must be strictly larger than  one. Indeed, it turns out that if $o^c(X,m)=X_j$ for some variable $j$, it is possible to identify $\pX(o^c(x,m) \mid o(X,m)=o(x,m))=\pX(x_j \mid X_{-j}=x_{-j})$. More generally, we now show that identification of the correct conditional distributions is possible under \ref{PMMMAR}, if one focuses \emph{on one variable $X_j$ at a time}. To this end, let $L_{j}=\{m \in \mathcal{M}: m_j=0  \},$ be the set of patterns in which $x_j$ is observed. Based on the missing patterns in $L_{j}$, one can build the following mixture distribution, 
\begin{align}\label{hstarminusj}
    & h^*(x_j \mid X_{-j}=x_{-j}) \nonumber \\
      = & \sum_{m \in L_j} \frac{p(  x_{-j}\mid M=m) \Prob(M=m)}{\sum_{m \in L_j} p(  x_{-j}\mid M=m) \Prob(M=m)} \\
      & \quad \times p( x_j \mid X_{-j}=x_{-j}, M=m)\nonumber,
\end{align} 
which is based \textit{only on the missing patterns such that $x_j$ is observed}. Here $x_{-j}$ corresponds to the set of all variables, except the $j$th. Thus the mixture \eqref{hstarminusj} coincides with \eqref{h0star}, if $o^c(x,m)=x_j$ and
\begin{align*}
    w_{m'}(o(x,m))=\frac{p(  x_{-j}\mid M=m') \Prob(M=m')}{\sum_{m \in L_j}p(  x_{-j}\mid M=m)\Prob(M=m)}.
\end{align*}

\begin{restatable}{proposition}{identificationprop}\label{amazingprop}
    Under \ref{PMMMAR}, the predictor $h^*$ defined in \eqref{hstarminusj} satisfies, for all $j \in \{1, \ldots, d\}$, for all $x \in \X$ such that $p(x_{-j} \mid M_j=0) > 0$,
    \begin{align}
        h^*(x_j \mid X_{-j}=x_{-j})=\pX( x_j \mid  X_{-j}=x_{-j}) .
    \end{align}
\end{restatable}

We note that $\Prob(M \in L_j)$ corresponds to the probability that $X_j$ is observed, i.e. $\Prob(M \in L_j)= \Prob(M_j=0)$. To show Proposition \ref{amazingprop} we simply show that
\begin{align*}
       h^*(x_j \mid X_{-j}&=x_{-j}) =p( x_j \mid  X_{-j}=x_{-j}) \\
       & \times \frac{1-\sum_{m \notin L_j} \Prob(M=m \mid X=x)}{1-\sum_{m \notin L_j} \Prob(M=m \mid  X_{-j}=x_{-j} ) }.
\end{align*}
Since the sum in the numerator is not allowed to depend on $x_j$ by \ref{SMARII}, Proposition \ref{amazingprop} follows immediately.

\Cref{amazingprop} shows that the true distribution $\pX( x_j \mid  X_{-j}=x_{-j})$ is indeed identifiable from all available patterns. Intuitively at $X_j$, one can divide the $|\mathcal{M}|$ patterns to two sets, one where $X_j$ is missing, and one where it is observed. Though these two aggregated patterns are mixtures of several patterns $m \in \mathcal{M}$, the MAR condition implies that both aggregated patterns have the same conditional distribution $X_j \mid X_{-j}$, thus allowing to identify the conditional distribution in the pattern where $X_j$ is observed. 

\begin{example}[Example \ref{interesting_new_Example} continued]
Consider the setting of Example \ref{interesting_new_Example}. While we cannot identified $\pX(x_1, x_2 \mid X_3=x_3)$ from the patterns where $x_1$, $x_2$ are observed, it holds that
\begin{align*}
    \pX(x_1 \mid X_2=x_2, X_3=x_3)=h^*(x_1 \mid X_2=x_2, X_3=x_3)\\
    \pX(x_2 \mid X_2=x_2, X_3=x_3)=h^*(x_2 \mid X_2=x_2, X_3=x_3),
\end{align*}
thus if $X_2$ were successfully imputed, $X_1$ could be recovered and vice versa.
\end{example}

We may unify and summarize Propositions \ref{prop_identifiability_mar} and \ref{amazingprop} as follows: Consider a subset $S \subset \{1,\ldots, d\}$, the subset of variables $X_S=(X_j)_{j \in S}$ and the set of all $m'$ such that all variables in $X_S$ are observed: $L_S=\{m': m'_j=0 \text{ for all } j \in S \}$. In Proposition \ref{prop_identifiability_mar}, we considered $X_S=o^c(X,m)$ for a fixed $m$ and, in Proposition \ref{amazingprop}, $X_S=X_j$ for fixed $j$. Then $p(x_S \mid X_{S^c}=x_{S^c})$ is identifiable from any $m' \in L_S$ under \ref{CIMAR}, from $0 \in L_S$ under \ref{EMAR} and identifiable from $L_S$ under \ref{PMMMAR} only under the additional condition that $|S|=1$. This is summarized in Table \ref{tab:missing_data}.

\begin{table}[h]
\centering
\begin{tabular}{lccc}
\toprule
Cardinality & CIMAR & EMAR & MAR \\
\midrule
$|S| \geq 2$ & $\forall m \in L_S$ & $0 \in L_S$ & Not always identifiable \\
$|S| = 1$ & $\forall m \in L_S$ & $0 \in L_S$ & $L_S$ \\
\bottomrule
\end{tabular}
\caption{Summary of Propositions \ref{prop_identifiability_mar} and \ref{amazingprop} for $S \subset \{1,\ldots, d\}$, the subset of variables $X_S=(X_j)_{j \in S}$ and the set of all $m$ such that all variables in $X_S$ are observed. For each condition the table indicates if $p(x_S \mid X_{S^c}=x_{S^c})$ is identifiable from $L_S$.}
\label{tab:missing_data}
\end{table}

\begin{remark}
    To further illustrate these results, consider the following analogy. One could think of different missing value patterns $m$ as different environments, such as different hospitals. Given several hospitals in which $(X_j , X_{-j})$ is observed, we would like to predict a variable $X_j$ from (fully observed) covariates $X_{-j}$ for a new hospital. Under \ref{CIMAR}, covariates $X_{-j}$ can arbitrarily change their distribution from one hospital to the next, but $X_j \mid X_{-j}$ remains the same in all hospitals, making it possible to learn the correct conditional distribution from any other hospital. Under \ref{PMMMAR} on the other hand, $X_j \mid X_{-j}$ may also differ from hospital to hospital. However, taking all hospitals such that $(X_j , X_{-j})$ is observed together, the distribution of $X_j \mid X_{-j}$ in the resulting mixture is the same as $X_j \mid X_{-j}$ needed for the new hospital.
\end{remark}

\paragraph*{Link with the FCS approach.} Imputing one variable at a time is precisely the approach of FCS. The goal of the FCS in general and the MICE approach in particular is to impute by iteratively drawing for all $j \in \{1, \ldots, d\}$ and $t \geq 1$,\begin{align*}
    x_j^{(t+1)} \sim \pX(x_j \mid X_{-j}=x_{-j}^{(t)}),
\end{align*}
where $x_{-j}^{(t)}=\{ x_l^{(t)} \}_{l \neq j}$ are the imputed and observed values of all other variables except $j$ at the iteration $t$. Doing this repeatedly leads to a Gibbs sampler that converges under quite mild conditions \cite[Chapter 10.2.4.]{RubinLittlebook}. Proposition \ref{amazingprop} shows that if we assume to have access to the true distribution $\pX(x_{-j})$, we can identify $\pX(x_j \mid X_{-j}=x_{-j})$  using only the observed values of $X_j$. In particular, this shows that, if we were to start from $X$ and successively impute
each $X_j$ with $h^*(x_j \mid X_{-j}=x_{-j})$, we would preserve the distribution of $X$, providing a stationary solution. As \ref{EMAR} and \ref{CIMAR} are both stronger than \ref{PMMMAR}, Proposition \ref{amazingprop} also holds with \ref{PMMMAR} replaced with either one of those conditions. Thus, Proposition \ref{amazingprop} shows that the FCS approach can identify the true conditional distributions in principle. 

\paragraph*{Block-wise FCS.}
The FCS approach of imputing one variable at a time has been criticized for computational reasons. Indeed, in FCS, $d$ models have to be fitted repeatedly, which can be computationally intensive for large $d$. One way to remedy this would be to use multi-output methods such as DRF to impute variables as blocks, such that, say, 10 variables may be imputed simultaneously. The idea of using block-wise FCS was already discussed in \cite[Chapter 4.7]{VANBUUREN2018}. However, Proposition \ref{prop_identifiability_mar} shows that one might not be able to recover the correct imputation distribution with this approach under \ref{PMMMAR}.

\paragraph*{Identifiability does not imply consistency.} In practice, we do not have access to any true distribution $\pX(x_j \mid X_{-j}=x_{-j}^{(t)})$ for $j \in \{1, \ldots, d\}$, and sample instead from an estimated distribution $p_n(x_j \mid X_{-j}=x_{-j}^{(t)})$, as shown in Algorithm \ref{alg:mice}. This naturally includes estimation error that propagates through the iterations. The characterization of the (finite-sample) error of nonparametric imputation is an open problem. In particular, our identification results, though an important first step, are far from results guaranteeing consistency of the imputation distribution and generally cannot explain the impressive performance of MICE in finite samples.

\begin{algorithm}
\caption{MICE Algorithm}
\label{alg:mice}
\begin{algorithmic}[1]
\REQUIRE Incomplete data matrix $\mathbf{X}$ with missing entries, maximum iterations $T$
\ENSURE Imputed complete data matrix
\STATE Initialize missing values (e.g., mean imputation) to obtain $\mathbf{X}^{(0)}$
\FOR{$t = 1, 2, \ldots, T$}
    \FOR{each variable $j \in \{1, \ldots, d\}$}
        \STATE Let $x_{-j}^{(t)}$ denote all variables except $x_j$ at iteration $t$
        \STATE Fit conditional model $p_n(x_j \mid X_{-j}=x_{-j}^{(t)})$ using observed data
        \STATE Sample imputations: $x_j^{(t+1)} \sim p_n(x_j \mid X_{-j}=x_{-j}^{(t)})$
        \STATE Update $\mathbf{X}^{(t+1)}$ with newly imputed values for variable $j$
    \ENDFOR
\ENDFOR
\RETURN $\mathbf{X}^{(T)}$
\end{algorithmic}
\end{algorithm}

\subsection{Overlap Condition}\label{Sec_overlap}


We note that our identification result in \Cref{amazingprop} only holds for $x_{-j}$ in the support of $P_{X_{-j} \mid M_j=0}$, i.e. $\{x_{-j}: p(x_{-j} \mid M_j=0) > 0\}$. For imputation, we would like to obtain $\pX(x_j \mid X_{-j}=x_{-j})$ for $x_{-j}$ in the support of $P_{X_{-j} \mid M_j=1}$. Thus, we need to ensure that the support of $P_{X_{-j} \mid M=1}$ is nested in the support of $P_{X_{-j} \mid M_j=0}$.

\begin{definition}\label{overlapdef}[Overlap Condition]
    The missingness mechanism meets the overlap condition if, for all $j \in \{1,\ldots,d\},$ for all $x \in  \X$, we have
    \begin{align}\label{OverlapDefinition}
        p(x_{-j} \mid M_j=1) > 0 \implies  p(x_{-j} \mid M_j=0) > 0.
    \end{align}
\end{definition}

This reflects the well-known fact that extrapolation outside of the support of a distribution is not possible in general, see e.g., \cite{pfister2024extrapolationaware}. However, \ref{PMMMAR} alone does not prohibit such extreme distributions shifts, as shown in Example \ref{nonoverlapexample}.


\begin{example}\label{nonoverlapexample}
Consider an example with two patterns, $\mathcal{M}=\{m_1, m_2\}$, with $m_1=(0,0)$, $m_2=(1,0)$ and $X_2|M=m_1 \sim \mathcal{U}([0,1])$, while $X_2|M=m_2 \sim \mathcal{U}([1,2])$. Moreover, assume $X_1=X_2\cdot \varepsilon$, with $\varepsilon \sim \mathcal{U}([0,1])$, such that $X_1 \mid X_2 \sim \mathcal{U}([0,X_2])$.
Since for $(x_1,x_2) \in [0,1]^2=\X_{\mid m_1}^{0}$, $p(x_1 \mid x_2, M=m_1)=\Ind\{0 \leq x_1 \leq x_2\}=p(x_1 \mid x_2)$, and analogously for $(x_1,x_2) \in [0,2]\times [1,2]=\X_{\mid m_2}^{0}$, this case meets \ref{CIMAR} and thus \ref{PMMMAR}, but the overlap condition \eqref{OverlapDefinition} is violated for $j=1$. To impute $X_1$ in pattern $m_2$ one would need to extrapolate from $X_2$ in $[0,1]$ to $X_2$ in $[1,2]$. 
\end{example}


    


Thus, \ref{PMMMAR} is not sufficient for imputation 
and needs to be complemented by an assumption that ensures that the support of $P_{X_{-j} \mid M_j=1}$ is nested in the support of $P_{X_{-j} \mid M_j=0}$. There are different possible ways to phrase sufficient conditions for Condition \eqref{OverlapDefinition}. We opt for the following version:


\begin{restatable}{lemma}{overlaplemma}\label{overlaplemma}
    The missingness mechanism meets the overlap condition in \Cref{overlapdef} if $0 \in \mathcal{M}$ and for all $x \in  \X$, we have
\begin{align}\label{Sufficientforoverlap}
    \Prob(M=0 | X=x ) > 0.  
\end{align}
\end{restatable}

Condition \eqref{Sufficientforoverlap} is sometimes referred to as a positivity assumption and is fundamentally important in the literature on inverse probability weighting (IPW, see, e.g., \cite{Seaman2018-IPW}) and the literature of graphical modeling of missing values (see, e.g., \cite{Mohan2013, mohan2013missing, nabi2025define}). However, in the latter approach, \cite{nabi2020missing, nabi2025define} assume that, for all $m \in \mathcal{M}, x \in \X$,  $\Prob(M=m \mid X=x) > 0$. This is a strictly stronger condition than \eqref{Sufficientforoverlap} if $0 \in \mathcal{M}$. For our purposes, \eqref{Sufficientforoverlap} is sufficient, aligns well with \ref{EMAR} (as shown below) and guarantees \eqref{OverlapDefinition}. For instance, in Example \ref{nonoverlapexample}, it holds that
\begin{align*}
&\Prob(M=m_1 \mid X_1=x_1,X_2=x_2)\\
&=\frac{\Ind\{x_2 \in [0,1]\} \Prob(M=m_1)}{\Ind\{x_2 \in [0,1]\}  \Prob(M=m_1) + \Ind\{x_2 \in [1,2]\}  \Prob(M=m_2) }.
\end{align*}
This is zero for $x_2 \in [1,2]$, thus violating \eqref{Sufficientforoverlap}.
While the condition in \Cref{overlapdef} is a simple requirement, such discussion often seems to be missing in the analysis of imputation. Moreover, even with this assumption, (unconditional) distribution shifts can occur across different missing patterns, even under stringent assumptions such as \ref{CIMAR}, as highlighted in \Cref{Example2} below.

\begin{example}\label{Example2}
    Consider the following Gaussian mixture model for two patterns $m_1=(0,0)$ and $m_2=(1,0)$: 
    \begin{align*}
        (X_1, X_2) \mid M=m_1 &\sim N\left(\begin{pmatrix} 0\\0 \end{pmatrix}, \begin{pmatrix}
            2 & 1\\
            1 & 1
        \end{pmatrix}\right) \\
        (X_1, X_2) \mid M=m_2 &\sim N\left(\begin{pmatrix} 5\\5 \end{pmatrix}, \begin{pmatrix}
            2 & 1\\
            1 & 1
        \end{pmatrix}\right).
    \end{align*}
An equivalent formulation is to generate first $X_2$ from a mixture of univariate Gaussian distributions $N(0,1)$ and $N(5,1)$ and then setting $X_1=X_2+\varepsilon$, with $\varepsilon \sim N(0,1)$, with $X_1$ being missing when $X_2$ is generated from $N(5,1)$.
In this example, \ref{CIMAR} and thus \ref{PMMMAR} holds. Moreover, it can be checked that Assumption \ref{Sufficientforoverlap} is also met. However, the distribution of $X_2$ in pattern $m_1$ ($N(0,1)$) is heavily shifted compared to pattern $m_2$ ($N(5,1)$). Consequently, an estimation method that is able to accurately learn $P_{X_1 \mid X_2}$, for $X_2 \sim N(0,1)$, also needs to be able to extrapolate to $P_{X_1 \mid X_2}$, for $X_2 \sim N(5,1)$. Figure \ref{fig:overlapillustration} illustrates the overlap and nonoverlap case in two examples, while Figure \ref{fig:Example2} shows how nonparametric methods struggle in this example for finite samples. 
\end{example}

\begin{remark}
    We note that \eqref{overlapdef} is also implied by $\Prob(M_j=0 \mid X_{-j}=x_{-j}) \geq \delta > 0$ for all $x_{-j}$ similar to the overlap condition for binary treatments in the causal inference literature, with $M_j$ being the treatment indicator. There are methods available to test the overlap condition in binary treatments, for instance in \cite{overlap1}, that might be applicable for imputation. However, the validity of these methods under missing values would need to be carefully assessed.
\end{remark}

\begin{figure}
    \centering
    \includegraphics[width=1\linewidth]{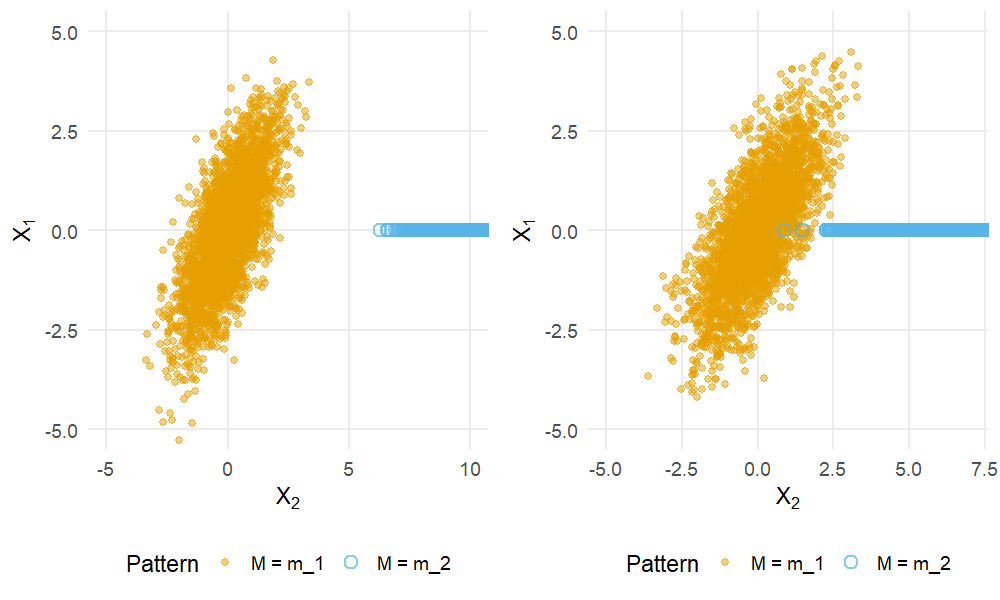}
    \caption{Left: An example where \eqref{overlapdef} is not met. Right: An example where \eqref{overlapdef} is met. In both cases, fully observed point are drawn in dark yellow, while only $X_2$ is shown in the pattern with missing $X_1$ in blue. In both cases $X_1 \mid X_2$ remains the same.}
    \label{fig:overlapillustration}
\end{figure}

\begin{figure}
    \centering

    \begin{tikzpicture}[scale=0.8, every node/.style={scale=0.8},
    box/.style={draw, rectangle, minimum width=2.5cm, minimum height=1.8cm, align=center},
    smallbox/.style={draw, rectangle, minimum width=2cm, minimum height=1cm, align=center},
    bluebox/.style={draw=blue, rectangle, minimum width=2cm, minimum height=1cm, align=center},
    arrow/.style={-{Stealth[length=5pt]}, blue, thick}
]

\node[box] (observed) at (-2,2.5) {\small{It is possible to identify}\\
\small{$p(x_j \mid X_{-j}=x_{-j})$ for all $x_{-j}$ }\\
\small{ such that $p(x_{-j} \mid M_j=0) > 0$}\\
\small{under \ref{PMMMAR}.}}; 
\node[box, draw=blue] (imputed) at (-2,0) { \small{It is possible to impute }\\
\small{with the above distribution} \\
\small{if $p(x_{-j} \mid M_j=1) > 0$ implies}\\
\small{ $p(x_{-j} \mid M_j=1) > 0$ \eqref{overlapdef}}};

\matrix (data) [matrix of math nodes, row sep=0.2cm, column sep=0.8cm, nodes={minimum width=1.5cm}, font=\large] at (4,1.5) {
  x_{1,j} & x_{1,-j} \\
  x_{2,j} & x_{2,-j} \\
  \vdots & \vdots \\
  \texttt{NA} & x_{\ell,-j} \\
  \texttt{NA} & x_{\ell+1,-j} \\
  \vdots & \vdots \\
};

\draw[arrow, black, thick] ($(data.west)+(0,1)$) -- (observed);
\draw[arrow] (imputed)--($(data.west)+(0,-1.5)$);

\node at ($(data.north west)+(-0.5,-2.7)$) {\scalebox{2.2}{$\left(\vphantom{\begin{matrix}x_{1,j}\\x_{2,j}\\\vdots\\x_{\ell+1,j}\\\vdots\end{matrix}}\right.$}};
\node at ($(data.north east)+(0.5,-2.7)$) {\scalebox{2.2}{$\left.\vphantom{\begin{matrix}x_{1,j}\\x_{2,j}\\\vdots\\x_{\ell+1,j}\\\vdots\end{matrix}}\right)$}};



\draw[blue, thick] ($(data.west)-(-0.3,0)$)-- ++(-0.2,0) -- ++(0,-2.6) -- ++(0.2,0);
\draw[blue, thick] ($(data.east)+(-0.3,0)$)-- ++(0.2,0) -- ++(0,-2.6) -- ++(-0.2,0);

\draw[black, thick] ($(data.west)-(-0.3,-2.6)$)-- ++(-0.2,0) -- ++(0,-2.6) -- ++(0.2,0);
\draw[black, thick] ($(data.east)+(-0.3,2.6)$)-- ++(0.2,0) -- ++(0,-2.6) -- ++(-0.2,0);

\end{tikzpicture}

    \caption{Conceptual illustration of the conclusions of Section \ref{Sec_FCSunderMAR} and \ref{Sec_overlap} for \ref{PMMMAR}. Section \ref{Sec_FCSunderMAR} with Proposition \ref{amazingprop} shows that if $P_{X_{-j} \mid M_j=0}$ is available, $\pX(x_j \mid X_{-j}=x_{-j})$ can be identified for all $x_{-j}$ such that $ p(x_{-j} \mid M_j =0) > 0$. Section \ref{Sec_overlap} adds the overlap condition \eqref{overlapdef} which is necessary to generalize from $X_{-j} \mid M_j=0$ to $X_{-j} \mid M_j=1$.}
    \label{fig:scoreillustrationOj}
\end{figure}

\section{Comparison to the literature on graphical models} \label{Sec_GraphicalModelscomparison}

Identifiability under missingness was studied at length in the literature on graphical modeling \cite{Mohan2013, mohan2013missing, mohan2021graphical, bhattacharya20b, nabi2020missing, nabi2025define}. 

Contrary to our case so far, where we worked directly with the (unobserved) full distribution of $(X,M)$ and studied identification of $X_j \mid X_{-j}$, this literature asks the questions whether the full distribution $P$ is recoverable from observed data $(X^*,M)$. We therefore use the term ``recoverable'', as in \cite{Mohan2013, mohan2021graphical} to differentiate it from our arguably simpler perspective above. It is well known \citep[see, e.g.,][]{nabi2020missing} that, the full law $P$ is recoverable iff \eqref{Sufficientforoverlap} holds and
\begin{align}\label{fullrecoverability}
    \Prob(M=m \mid X=x) \text{ is recoverable for all } m \in \mathcal{M}, x \in \X.
\end{align}
%
Assume now that $0 \in \mathcal{M}$ and that there exists a set of variables $X_O$ that are always observed. In the graphical modeling literature,  ``MAR'' was traditionally taken to be \ref{RMAR} \citep[see, e.g.,][]{Mohan2013, mohan2021graphical}, and in this case, $\Prob(M=m \mid X=x)=\Prob(M=m \mid X_O=x_O)$, so that \eqref{fullrecoverability} holds and $P$ is recoverable. Since \ref{CIMAR} is equivalent to \ref{RMAR}, this also means that $P$ is recoverable under \ref{CIMAR}.

Now consider the case of \ref{EMAR}. As shown in Lemma \ref{CIMARequivRMAR}, \ref{EMAR} is equivalent to $\Prob(M=0 \mid X=x)=\Prob(M=0 \mid X_O=x_O)$. This is also true if $O=\emptyset$, as in this case $\Prob(M=0 \mid x)=\Prob(M=0)$. This leads to the following proposition.

\begin{restatable}{proposition}{EMARprop}\label{EMARrecoverability}
Assume $0 \in \mathcal{M}$. Under \ref{EMAR} and positivity \eqref{Sufficientforoverlap}, the full law $P$ is recoverable.
\end{restatable}

Thus, $P$ is recoverable under \ref{EMAR}. Furthermore, an interesting recent graphical definition of MAR is given in \cite{nabi2025define}, which is much more general than \ref{RMAR}: In our notation, for a permutation $\{\pi(j)\}_{j=1\ldots, d}$ of $\{1\ldots, d\}$, we consider a factorization:
\begin{align}\label{graphicalMAR0}
    \Prob(M=m \mid X)=\prod_{j=1}^{d} \Prob(M_{\pi(j)}=m_{\pi(j)} \mid \text{Pa}(M_{\pi(j)}) ),
\end{align}
where $\text{Pa}(M_{\pi(j)})$ collects all elements $(X, M_{\pi(1)}, \ldots,$ $ M_{\pi(j-1)})$ that are ``parents'' of $M_{\pi(j)}$. We then require that, there exists a permutation $\pi$ such that for all $j$ and all $m$,
\begin{align}\label{graphicalMAR}
\Prob(M_{\pi(j)}=m_{\pi(j)} \mid \text{Pa}(M_{\pi(j)}) ) 
\end{align}
depends on $X$ only through $X^*$.  
In other words, $\Prob(M_{\pi(j)}$ $=m_{\pi(j)} \mid \text{Pa}(M_{\pi(j)}) )$ may depend on $X_{j}$ only on the event $M_j=0$. This essentially corresponds to \ref{PMMMAR} and the condition that the joint distribution $P$ can be factored according to a Directed Acyclic Graph (DAG). It can be shown that \eqref{graphicalMAR} holds for instance in Example \ref{interesting_new_Example}, where \ref{EMAR} does not hold. Considering the setting of Example \ref{interesting_new_Example}, and $\pi(j)=j$ for all $j$, $\Prob(M_{1}=1 \mid \text{Pa}(M_{1}) )=1/3$,
\begin{align*}
    \Prob(M_{2}=1 \mid \text{Pa}(M_{2}) )=
    \begin{cases}
        1 & \text{ if } M_1=1\\
        (1-x_1)/2& \text{ if } M_1=0.
    \end{cases}
\end{align*}
Thus, $\Prob(M_{2}=1 \mid \text{Pa}(M_{2}) )$ depends on $X_1$ only if $M_1=1$, i.e. only through $X_1^*$. Similarly, $\Prob(M_{3}=1 \mid \text{Pa}(M_{3}) )$ does not depend on $X$. We provide details in Appendix \ref{Sec_AdditionalexampleDetails}. This is connected to the fact that, knowing the dependency structure, it is straightforward to devise a strategy to recover the full distribution in this example: Since $X_3$ is independent uniform, one could recover $P_{X_3}$, as $\pX(x_3)=\pX(x_3 \mid M_3=0)$. This makes it possible to recover $h^*(x_2 \mid x_1, x_3)$, without having to impute $X_1$, making it possible to recover $P$. 


However, this strategy and Condition \eqref{graphicalMAR} is based on knowledge of the (conditional) independence structure. Moreover, it is possible to construct examples that meet \ref{PMMMAR}, but not \eqref{graphicalMAR} as the following adaptation of Example \ref{interesting_new_Example} shows:

\begin{example}\label{interesting_new_Example_adaptation}
        Consider variables $(X_1,X_2,X_3)$, each marginally distributed as a uniform between $[0,1]$, with $(X_1,X_2)$ being dependent and $X_3$ being independent of $(X_1, X_2)$. Moreover, we consider $$\mathcal{M}=\{m_1,m_2, m_3\}=\{(0,0,0), (0,1,0), (1,0,0)\}$$ and
    \begin{align*}
        &\Prob(M=m_1\mid X=x)=(x_1+x_2)/3,\\
        &\Prob(M=m_2\mid X=x)=(2-x_1)/3\\
        &\Prob(M=m_3\mid X=x)=(1-x_2)/3.
    \end{align*}
    It can be shown that in this example, Condition \eqref{Sufficientforoverlap} holds for $x \in  \X=[0,1]^3$. However, this example is again not \ref{EMAR}, so that $X_{1} \mid X_{2}, X_3$ cannot be learned from $m_1$ and similarly with $X_2 \mid X_1, X_3$. In fact, contrary to Example \ref{interesting_new_Example}, it can be shown that $\Prob(M_1=1 \mid M_2=0, X=x)$ depends on $x_1$ and $\Prob(M_2=1 \mid M_1=0, X=x)$ depends on $x_2$, making it impossible to find a factorization that meets \eqref{graphicalMAR}. Details are given in Appendix \ref{Sec_AdditionalexampleDetails}. This is again directly connected to the fact that there appears to be no clear strategy on how to recover $X_1$ without first imputing $X_2$ and vice-versa. 
\end{example} 

As such, as already mentioned in \cite{nabi2025define}, Condition \eqref{graphicalMAR} is strictly stronger than \ref{SMARII}. This is because ``A MAR missingness mechanism, according to Rubin’s definition, requires that for any given missingness pattern, the missingness is independent of the missing
values given the observed values. These types of restrictions are with respect
to missingness patterns [...] This missingness mechanism cannot be represented via
graphs'' \cite{nabi2025define}. As such, even in the general terminology of \cite{nabi2025define}, Example \ref{interesting_new_Example_adaptation} would be an MNAR example, and more advanced graphical tools are necessary to show identification, even when the dependency structure is perfectly known. On the other hand, Proposition \ref{amazingprop} still holds in this example, and in Section \ref{Sec_Empirical}, we empirically demonstrate that FCS imputation methods are capable of reproducing the distribution of $X$ accurately.

\section{Implications for Imputation}\label{Sec_Implicationsoverall}\label{Sec_Implications}

In Section \ref{Sec_Main_Results}, we proved the identification of $p( x_j \mid  X_{-j}=x_{-j})$ for all $j$, based on all missing data patterns under \ref{SMAR}. While this result is a first step to create methods that replicate the data distribution $\PX$, in practice it remains to estimate $\pX(x_j \mid X_{-j}=x_{-j})$ in the FCS iterations. From now on, we will refer to any method that estimates $\pX(x_j \mid X_{-j}=x_{-j})$ for all $j$, as an imputation method. In this section, we present three essential properties that an ideal imputation method should satisfy. We then discuss which properties are met by classical FCS imputation strategies and introduce a new imputation approach, denoted mice-DRF.  


The goal of each iteration of an FCS algorithm is to estimate the conditional distributions $\pX(x_j \mid X_{-j}=x_{-j})$ for all $j \in \{1, \hdots, d\}$. Thus, an imputation method is intrinsically a distributional estimator (Requirement \ref{distregressionmethod} below). This stands in contrast to methods that only estimate and impute by conditional means and medians, and may be as simple as using a Gaussian distribution with estimated conditional mean (denoted as mice-norm.nob below). However, in order to accurately estimate the conditional distributions $\pX(x_j \mid X_{-j}=x_{-j})$, an imputation method should be able to capture complex (potentially non-linear) interactions in the data (Requirement \ref{flexible}). Moreover, with Example \ref{Example2} in mind, an imputation method should be able to handle distributional shifts in the covariates (Requirement \ref{shifts}), under overlap \eqref{OverlapDefinition}.


Consequently, in the FCS framework under \ref{PMMMAR} and \eqref{OverlapDefinition}, an imputation method should
\begin{enumerate}[label=(\arabic*)]
    \item\label{distregressionmethod} be a distributional regression method
    \item\label{flexible} be able to capture nonlinearities and interactions in the data
    \item\label{shifts} be able to deal with distributional shifts in the covariates. 
\end{enumerate}


We note that without overlap \eqref{OverlapDefinition}, point \ref{shifts} cannot be met using nonparametric methods, as extrapolation will only be possible under parametric assumptions, see e.g. \cite{pfister2024extrapolationaware}.

\begin{remark}
    \cite{greatoverview} (Section 4) also discusses best practices for general imputation methods. Their points partly overlap with ours. In particular, they suggest that ``imputations should reflect uncertainty about missing values'' (corresponding to \ref{distregressionmethod}) and ``imputation models should be as flexible as possible'' (corresponding to \ref{flexible}). As we note in Section \ref{Sec_Discussion}, they also emphasizes that for multiple imputation the uncertainty of the imputation model should be considered. This is not met by the imputation methods presented here and is an open problem for nonparametric imputation. While this becomes less consequential in large samples, this additional uncertainty is needed for reliable uncertainty quantification with multiple imputation using Rubin's Rules. 
\end{remark}

We now describe some of the state-of-the-art imputation methods, based on the benchmark analysis of papers by \cite{Waljee2013, RFimputationpaper, awesomebenchmarkingpaper, wang2022deep, ImputationScores}.

We first consider two very simple methods, the \emph{Gaussian} and \emph{regression} imputations. Following the naming convention of the \texttt{mice} \textsf{R} package, we also denote the former as \emph{mice-norm.nob} and the latter as \emph{mice-norm.predict}. Both fit a linear regression of $X_j$ onto $X_{-j}$. The Gaussian imputation then imputes $X_j$ by drawing from a Gaussian distribution, while the regression imputation simply uses the conditional mean. Given that linear regression is known to extrapolate well when the true underlying model is linear, \emph{mice-norm.predict} meets \ref{shifts}, while \emph{mice-norm.nob} meets \ref{distregressionmethod} and \ref{shifts}. A widely-used method in a variety of fields is the \emph{missForest} of \cite{stekhoven_missoforest}. In this method, $X_j$ is regressed on $X_{-j}$ with a Random Forest (RF, \cite{breiman2001random}) and then imputed with the conditional mean in each iteration. As such, missForest, only meets \ref{flexible}. In contrast, \emph{mice-cart} and \emph{mice-RF} \citep{CARTpaper0, CARTpaper1} use one or several trees respectively, but sample from the leaves to obtain the imputation of $X_j$, approximating draws from the conditional distribution. Thus these methods satisfy \ref{distregressionmethod}, in addition to \ref{flexible}. As such, they may be inheriting the accuracy of missForest, while providing draws from the conditional distribution. However, they are ultimately not designed for the task of distributional regression, using a splitting criterion that is tailored to estimate the conditional expectation. To this end, the Distributional Random Forest (DRF) was recently introduced in \cite{DRF-paper}, with a splitting criterion adapted to conditional distribution estimation. We thus define a new imputation method, denoted \emph{mice-DRF}, that regresses $X_j$ onto $X_{-j}$ and then imputes by sampling from the distributional regression estimator. As DRF is a forest-based distributional method, mice-DRF meets \ref{distregressionmethod} and \ref{flexible}. However, as any local averaging estimate (kernel methods, tree-based methods, nearest neighbors), DRF still generalizes poorly outside of the training set \citep[see, e.g.,][]{malistov2019gradient}, i.e. Requirement \ref{shifts} is not met. Table \ref{summarytable2} summarizes the properties met by the different MICE methods considered in this paper.

Figure \ref{fig:Example2} illustrates the behavior of different imputation strategies for Example \ref{Example2}. As the Gaussian imputation fits a regression in pattern $m_1$ and then draws from a conditional Gaussian distribution given the estimated parameters, it is the ideal method in this setting and indeed captures the distribution very well. For the nonparametric methods, DRF, as a distributional method, performs better than mice-cart. However, it still fails to deal with the covariate shift, centering around 2, when it should center around 5. 

Thus, while previous analysis suggests that forest-based methods such as mice-cart, mice-RF, and likely also mice-DRF may be some of the most successful methods currently available, finding an (FCS) imputation method that satisfies \ref{distregressionmethod}--\ref{shifts} is still an open problem. However, these three ideal properties might help guide the development of new imputation methods. As each new imputation method should be carefully tested against the wealth of imputation methods that could be considered, including joint modeling approaches such as GAIN \citep{GAIN}, the ability to benchmark imputation methods is crucial. Thus, we now turn to the question of how to evaluate the performance of different imputation strategies. Requirement \ref{distregressionmethod} suggests that distributional distances or scores should be used instead of classic predictive metrics such as RMSE. Recall that $\PX$ is the marginal distribution of $X$ and assume that a sample from $\PX$ is observed, i.e. the complete data is available, as is the case in benchmarking studies. In order to evaluate the performance of an imputation strategy that produces a distribution $H$, we compute the (negative) energy distance between imputed and real data:
\begin{align*}
    \tilde{d}^2(H,\PX)&=2 \E_{\substack{X \sim H\\Y \sim \PX}}[ \| X-Y \|_{2}] - \E_{\substack{X \sim H\\ X' \sim H}}[ \| X-X' \|_{2}] \nonumber \\
   & -\E_{\substack{Y \sim \PX\\ Y' \sim \PX}}[ \| Y-Y' \|_{2}],
\end{align*}
 where $\| \cdot \|_{2}$ is the Euclidean metric on $\R^d$. Given samples from $\PX$ and a sample of imputed points this can be readily estimated, see e.g., \cite{EnergyDistance}. 
\begin{table}[]
\centering
\begin{tabular}{@{}llll@{}}
\toprule
Method                                  & \ref{distregressionmethod}                                            & \ref{flexible}                                            & \ref{shifts}                                                                                                                                \\ \midrule
\multicolumn{1}{|l|}{missForest}        & \multicolumn{1}{l|}{}                          & \multicolumn{1}{l|}{\checkmark} & \multicolumn{1}{l|}{}                                        \\ \midrule
\multicolumn{1}{|l|}{mice-cart}         & \multicolumn{1}{l|}{\checkmark} & \multicolumn{1}{l|}{\checkmark} & \multicolumn{1}{l|}{}                           \\ \midrule
\multicolumn{1}{|l|}{mice-RF}           & \multicolumn{1}{l|}{\checkmark} & \multicolumn{1}{l|}{\checkmark} & \multicolumn{1}{l|}{}                                                 \\ \midrule
\multicolumn{1}{|l|}{mice-DRF}          & \multicolumn{1}{l|}{\checkmark} & \multicolumn{1}{l|}{\checkmark} & \multicolumn{1}{l|}{}                            \\ \midrule
\multicolumn{1}{|l|}{mice-norm.nob}     & \multicolumn{1}{l|}{\checkmark} & \multicolumn{1}{l|}{}                          & \multicolumn{1}{l|}{\checkmark}   \\ \midrule
\multicolumn{1}{|l|}{mice-norm.predict} & \multicolumn{1}{l|}{}                          & \multicolumn{1}{l|}{}                          & \multicolumn{1}{l|}{\checkmark}   \\ \bottomrule
\end{tabular}
\caption{Properties \ref{distregressionmethod}--\ref{shifts} met by different imputation methods. Following the naming convention of the \texttt{mice} \textsf{R} package, ``mice-norm.nob'' refers to the Gaussian imputation, while ``mice-norm.predict'' refers to the regression imputation.}
\label{summarytable2}
\end{table}

\begin{figure}
    \centering
    \includegraphics[width=1\columnwidth]{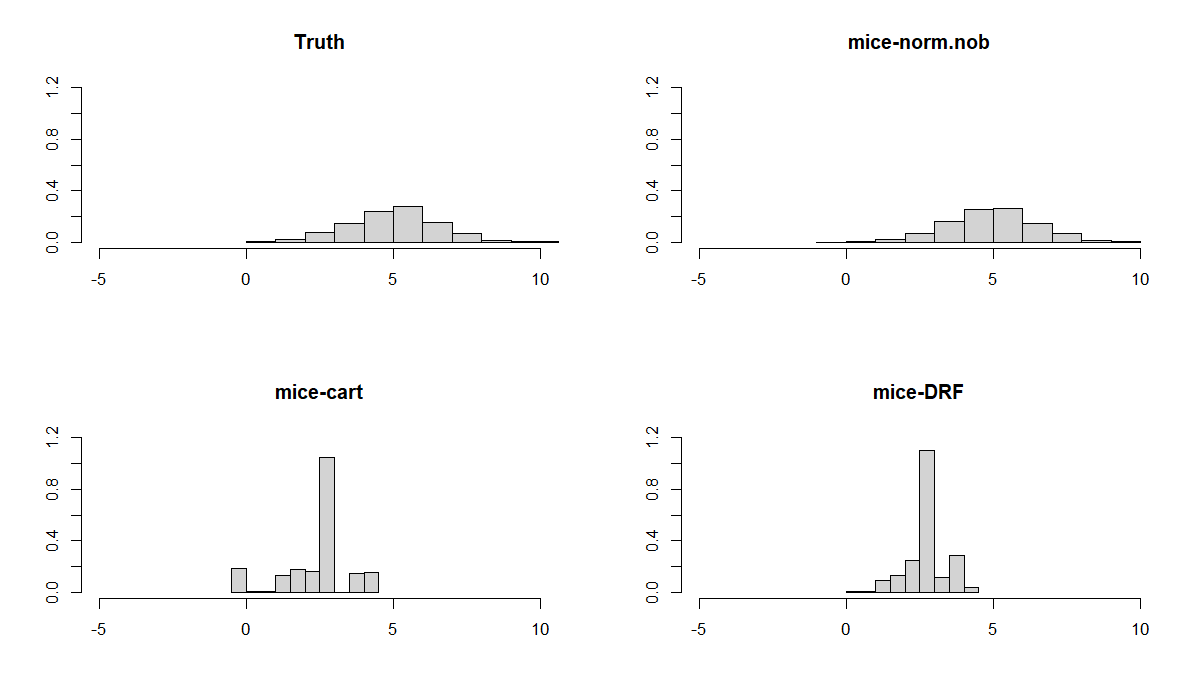}
    \caption{The true distribution against a draw from different imputation procedures for imputing $X_1$ in Example \ref{Example2}.}
    \label{fig:Example2}
\end{figure}

\section{Numerical Illustration} \label{Sec_Empirical}


The goal of this section is to illustrate the concepts discussed in this paper on both simulated and real data, that is:
\begin{itemize}
    \item We study Example \ref{interesting_new_Example_adaptation} with a rather difficult MAR mechanism, that is not \ref{EMAR} and where it is unclear whether $\PX$ can be recovered, and show empirically that FCS imputation recovers $\PX$ well.
    \item We illustrate why treating imputation as a distributional prediction problem is crucial, in particular, how methods like missForest can fail on downstream tasks.
    \item Following this line of thinking, we (i) introduce mice-DRF for distributional imputation and (ii) use the energy distance to score imputation quality.
\end{itemize}

We note that for the new mice-DRF, we essentially implement the mice-RF algorithm described in \cite{CARTpaper1}, with the traditional Random Forest exchanged by the Distributional Random Forest. The \texttt{mice} package provides a very convenient interface through which new regression methods can be added to the mice routine. An \textsf{R}-package implementing this method can be found on \url{https://github.com/KrystynaGrzesiak/miceDRF}.

\paragraph{Imputation methods.}  We empirically evaluate the performances of the following FCS methods
\begin{itemize}
    \item mice-cart
    \item mice-DRF
    \item missForest
    \item regression imputation, named mice-norm.predict (see \textsf{R}-package \texttt{mice} \citep{mice})
    \item Gaussian imputation, named mice-norm.nob (see \textsf{R}-package \texttt{mice} \citep{mice}) 
\end{itemize}
We also compare two deep learning strategies
\begin{itemize}
    \item GAIN \citep{GAIN} 
    \item MIWAE \citep{MIWAE}.
\end{itemize}
All methods are used with their default hyperparameter values. We note that this should not be meant to be an exhaustive list, but simply a collection of imputation methods to illustrate the discussions of this paper.

\paragraph*{Evaluation.} 
To evaluate the imputation methods, we calculate the (negative) energy distance between the true and imputed data sets, using the \texttt{energy} \textsf{R}-package \citep{energypackage}. To illustrate the discussion in this paper, we also add the negative RMSE. In all examples, we standardize the scores over the 10 repetitions to lie in $(-1,0)$.

\subsection{Example \ref{interesting_new_Example_adaptation}}

In this section we consider Example \ref{interesting_new_Example_adaptation}, albeit with three instead of one independent variable added. That is, we consider 3 patterns:
\begin{align*}
    m_1 = (0, 0, 0, 0, 0),\ \ m_2 = (0, 1, 0, 0, 0),\ \ m_3 = (1, 0, 0, 0, 0),
\end{align*}
with
\begin{align*}
       &\Prob(M=m_1 \mid X=x)  = (x_1+x_2)/3\\
    &\Prob(M=m_2 \mid X=x)= (2-x_1)/3\\
    &\Prob(M=m_3 \mid X=x)  = (1-x_2)/3.
\end{align*}

We draw $n=5000$ i.i.d. copies of $(X,M)$, apply the missingness mechanism to obtain $(X^*, M)$, impute and then compare the imputation to the complete data using the energy distance. Figure \ref{fig:Application_0_Scores} shows the results, with the ``true'' imputation, sampling from the correct conditional distributions included. As expected, the true imputation is ranked highest in terms of energy distance, closely followed by mice-cart mice-DRF and mice-norm.nob. In fact, though this is not visible in Figure \ref{fig:Application_0_Scores}, the energy distance between imputation and truth is of the same order for mice-DRF, mice-cart and the true imputation, indicating that the true underlying distribution is well recovered by mice-DRF and mice-cart, despite the complex setting. Despite this success all three distributional mice imputations attain the worst score in terms of (negative) RMSE. Interestingly, missForest, being an FCS method that imputes only conditional expectations, still ranks higher than the joint methods GAIN and MIWAE.

To further demonstrate how remarkable the strong performance of (distributional) FCS imputation methods is in this example, consider a typical M-Estimation task, such as estimation of (conditional) quantiles. Clearly, complete case analysis, i.e., only considering $M=m_1$ will lead to a biased estimate in this example. Instead, M-Estimation can be used by considering the observed parts of each pattern, see e.g., \cite{Mestimatormissingvalues} or Appendix \ref{Sec_likelihoodignorarbility}. In fact, a special case of this is the classical Maximum Likelihood estimator under missingness introduced in \cite{Rubin_Inferenceandmissing}. However, contrary to MLE this approach is not generally consistent for M-Estimators outside of MCAR \citep{Mestimatormissingvalues}. To achieve consistent estimation, one might consider propensity weighting methods that reweigh the M-Estimation problem using estimates of $\Prob(M=m \mid x)$, (see, e.g., \cite{inverseweightingoverview, Shpitser_2016, MARinverseweighting, CANTONI2020, Malinsky2022} among others). However, as outlined in Example \ref{interesting_new_Example_adaptation}, there is no clear way to recover $\Prob(M=m \mid x)$ for any $m \in \{m_1, m_2, m_3\}$ without iterative imputation in this example. Finally, we note that imputing by a non-distributional method like missForest also should lead to a bias. Figure \ref{fig:quantileResults} demonstrates these points for a very simple M-Estimator: the 0.1-quantile of $X_1$ over 20 iterations for (1) the mice-cart and mice-DRF imputations, (2) the missForest imputation and (3) M-Estimation considering the observed variables in each pattern, which here simply boils down to calculating the quantile of of all observed $X_1$, $X_1 \mid M_1=0$. First, for (3) there is a clear bias stemming from the fact that the 0.1 quantile of $X_1 \mid M_1=0$ is approximately $0.106$ instead of $0.1$, as shown in Appendix \ref{Sec_AdditionalexampleDetails}. Second, the ordering of the imputation methods reflects the ordering found in Figure \ref{fig:Application_0_Scores}, with missForest having the highest bias. This discussion serves to illustrate that in this MAR example, distributional FCS imputation achieves an estimation accuracy that is not trivial to achieve with other methods.

\begin{figure}
    \centering
    \includegraphics[width=1\linewidth]{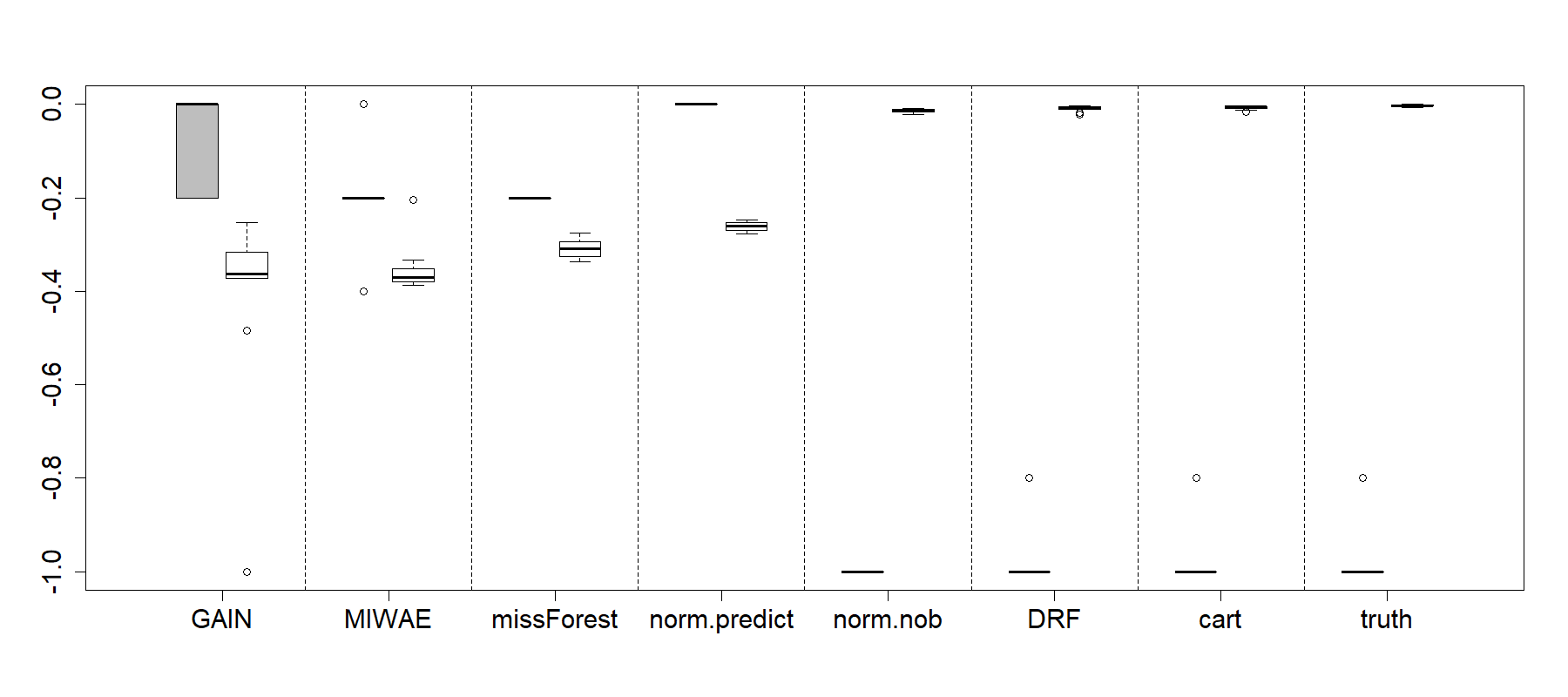}
    \caption{Standardized scores (the higher the better) for Example \ref{interesting_new_Example_adaptation}. Methods are ordered according to the negative energy distance. For each method, negative RMSE (gray, left) and the negative energy distance (white, right) are shown.}
    \label{fig:Application_0_Scores}
\end{figure}

\begin{figure}
    \centering
    \includegraphics[width=1\linewidth]{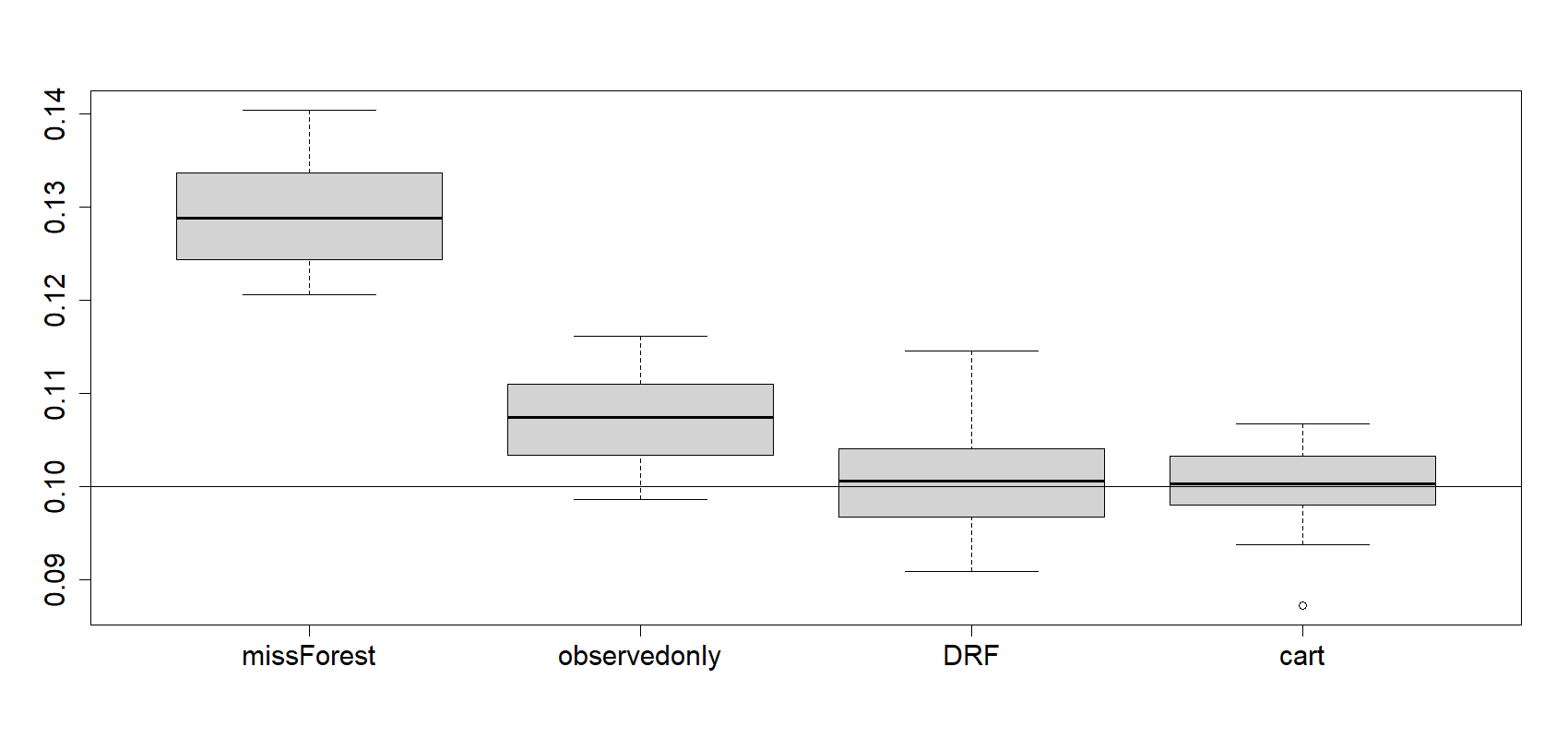}
    \caption{Estimation of the 0.1-Quantile of $X_1$ using (1) the mice-cart and mice-DRF imputations, (2) the missForest imputation and (3) the quantile of of all observed $X_1$, $X_1 \mid M_1=0$.}
    \label{fig:quantileResults}
\end{figure}

\subsection{Air Quality Data}\label{Sec_Airquality}

Imputation is usually not done in isolation, but to apply one or several downstream tasks in a second step. In the last section, we considered the estimation of a univariate quantile. In this section, we would like to compare the imputations on a more complex downstream task, namely calculating the Wasserstein distance (see e.g., \cite{ComputationalOT}), to another fully observed data set. Calculating distributional distances such as the Wasserstein distance between datasets is a very common task in Machine Learning, used for instance in testing \citep{Wassersteintesting}, GANs \citep{WassersteinGAN}, domain adaptation \citep{WassersteinDomainadaptation} and variational inference \citep{WassersteinVariationalInference}. We use the air quality data set obtained from \url{https://github.com/lorismichel/drf/tree/master/applications/air_data/data/datasets/air_data_benchmark2.Rdata}. This is a pre-processed version of the data set that was originally obtained from the website of the Environmental Protection Agency (\url{https://aqs.epa.gov/aqsweb/airdata/download_files.html}). For a detailed description of the data set, we refer to \cite[Appendix C.1]{DRF-paper}. The data set contains a total of 50'000 observations with 11 dimensions, 6 numerical and 5 binary. The goal of this example is to consider a simple MCAR mechanism with a relatively high proportion of missing values in the first six dimensions of numerical variables. That is, we simply let $M_1, \ldots, M_6$ be independent Bernoulli random variable, each with probability of missingness $0.3$, and $M_j=0$ for $j=7\ldots, 11$.

The wealth of data allows us to redraw two data set of 2'000 observations $B=10$ times to get an idea of the variation. We then introduce the MCAR missingness in one of the data sets and impute. In a second step, we calculate the Wasserstein distance between the imputed and the other (complete) data set and compare this to the Wasserstein distance computed with complete data. Figure \ref{fig:Application_1_Scores} shows the (standardized) energy distance, RMSE and the absolute difference in the Wasserstein distance obtained from the imputation and the complete data. The ordering induced by the energy distance matches well with the performance on the Wasserstein downstream task. Both show mice-cart and mice-DRF first and GAIN last. In addition, both score missForest third and, in particular, better than norm.nob. This might be due to the fact that missForest, while predicting instead of drawing from a conditional distribution, still models the nonlinearities in the data relatively well, a feat the Gaussian-based norm.nob cannot achieve. This shows on a very common task that the choice of an imputation is crucial and that the energy distance might be a good measure for benchmarking.

Surprisingly in this example, mice-cart is quite successful in terms of RMSE. On the other hand, RMSE scores missForest again highest, simply because it estimates the conditional means well. This mirrors previous analysis on real data with artificially generated missing values that found missForest to perform well (see e.g., \cite{Waljee2013, RFimputationpaper, awesomebenchmarkingpaper} among others) and indicates that this might be entirely due to the use of RMSE.

\begin{figure}
    \centering
    \includegraphics[width=1\linewidth]{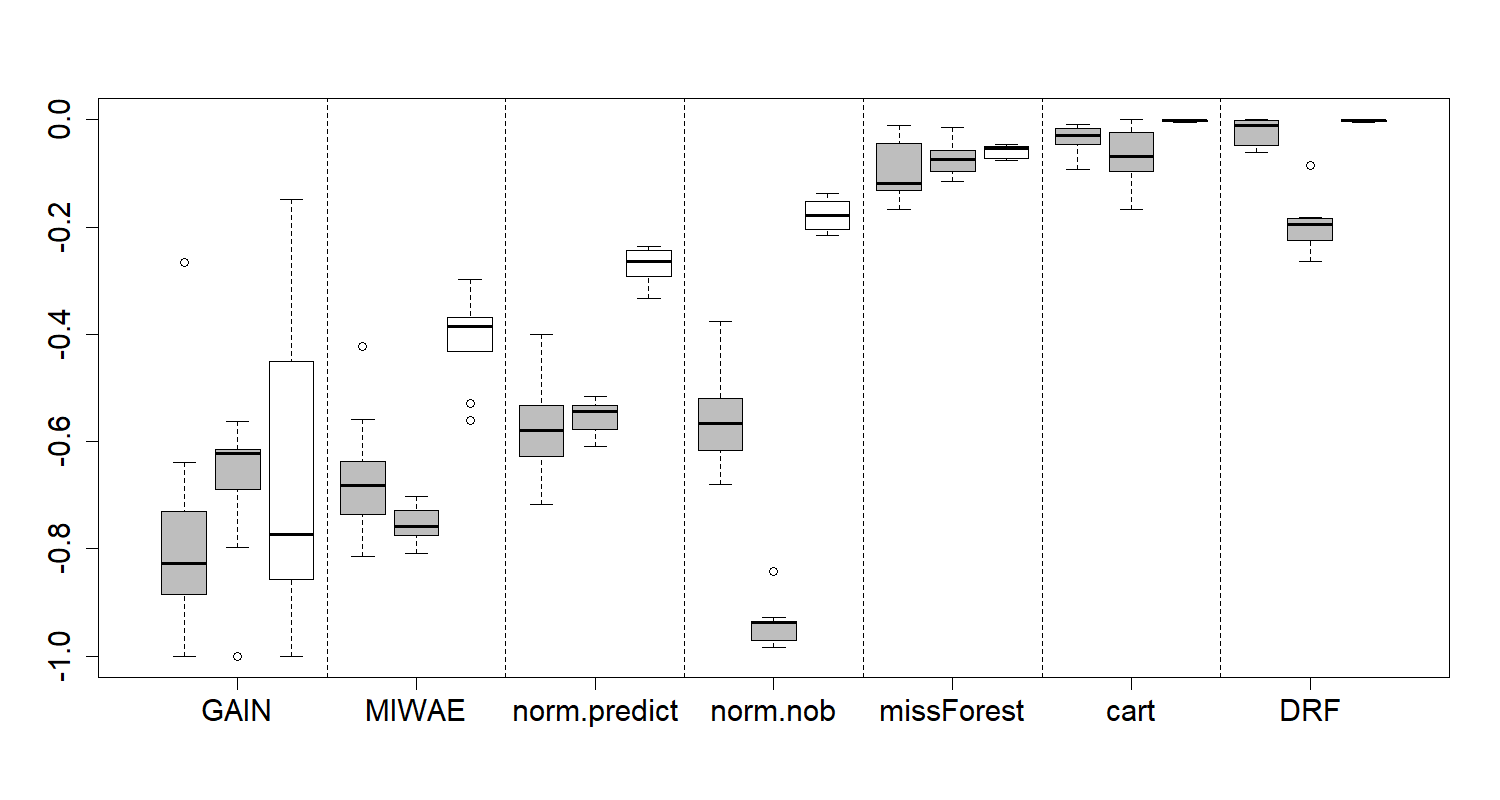}
    \caption{Standardized scores (the higher the better) for the air quality data example. Methods are ordered according to the negative energy distance. For each method, the negative energy distance (white, right), RMSE (gray, middle) and negative difference to the Wasserstein distance calculated on full data (dark gray, left) are shown.}
    \label{fig:Application_1_Scores}
\end{figure}

\subsection{Summary of Results}

We first note that our analysis is by no means a thorough comparison of imputation methods, but simply a numerical illustration of the points made in this paper. As such, the four examples considered in this section, as well as the analysis in Appendix \ref{Sec_GAINMIWAE}, indicate that:
\begin{itemize}
\item[(I)] Complex (conditional) distribution shifts are possible under \ref{PMMMAR} and methods that rely on stronger assumptions such as \ref{EMAR} may fail even in seemingly simple examples. In line with our results, FCS imputation methods that meet (1) and (2), such as mice-DRF and mice-cart, are able to recover the distributions remarkably well.
\item[(II)] Using the energy distance to evaluate imputation method reliably manages to identify the ``best'' imputation in our examples. This is in contrast to RMSE, which often results in an undesirable ordering.
\item[(III)] However, none of the methods is able to reliably deal with strong distributional shifts and nonlinearity, showing once again that better imputation methods need to be found.

\end{itemize}



\section{Discussion}\label{Sec_Discussion}

This paper gives a more systematic discussion of FCS imputation. We analyse and generalize the MAR condition for imputation and, based on this analysis, propose three essential properties an ideal imputation method should meet, as well as a principled way of ranking imputation methods. We conclude with four important points. 

\paragraph*{RMSE should not be used.} RMSE is not a sensible way of evaluating imputations. Dropping RMSE as an evaluation method likely has important implications. For instance, the recommendation of papers to use single imputation methods such as k-NN imputation \citep{knn_adv1} or missForest \citep{Waljee2013, Tang2017-on} appears to rest entirely on the use of RMSE. Even well-designed paper benchmarking imputation methods such as \cite{awesomebenchmarkingpaper} use RMSE. Nonetheless, there appear to be only a handful of recent papers that at least consider different evaluation methods, for instance, \cite{OTimputation,RFimputationpaper, wang2022deep}. Indeed, the problems of RMSE appear to be rediscovered in different fields. For instance,  \cite{RFimputationpaper}  empirically emphasize that, while missForest achieves the smallest RMSE, parameter estimations in linear regression are severely biased. Similarly, \cite{wang2022deep} discusses some problems with using RMSE in the machine learning literature. In contrast, GAN-based approaches recognize the objective of drawing imputations from the respective conditional distributions and naturally use the pattern-mixture modeling approach. However, despite having the right objective, these papers again use RMSE to compare the imputation quality of their method to competitors \citep[see, e.g.,][]{GAIN}.

\paragraph*{New imputation methods are needed.} The problem of imputation is by no means solved. Though there is a set of promising imputation methods with mice-cart, mice-RF, and mice-DRF, there is room for improvement, especially concerning the ability to deal with covariate shifts. In particular, Appendix \ref{Sec_Gaussmixmodelnonlinar} shows an example with distribution shifts and nonlinear relationships for which all methods fail. In addition, when considering multiple imputation, we note that none of the studied nonparametric methods is able to include \emph{model uncertainty}. However, this would technically be needed for correct uncertainty quantification with multiple imputation using Rubin's Rules, see e.g., \cite{greatoverview}. Though mice-rf of \cite{CARTpaper1} attempts to account for model uncertainty using several trees, this is only a heuristic solution.

\paragraph*{Further MAR generating mechanisms may need to be considered.}
It appears intuitive that the combination of distributional shifts and nonlinear relationships is widespread in real data. At the same time, the success of forest-based methods such as missForest and mice-cart in benchmark papers suggests that current ways of introducing MAR might not produce enough distribution shifts in general. For instance, \cite{ImputationScores} analyzed a range of data sets using the standard MAR mechanism of the ampute function implementing the procedure of \cite{ampute}, as we did in Section \ref{Sec_Airquality}. Their analysis also showed mice-cart consistently in first place. Thus, tweaking the approach of \cite{ampute} to produce MAR data with distribution shifts, might be an avenue for further research. In this context, the \ref{CIMAR} assumption might be useful, as MAR examples with distribution shifts can easily be generated.\\


\begin{appendix}

\section{Additional Empirical Considerations}\label{Sec_GAINMIWAE}

\subsection{Uniform Example}\label{PMMMARExample}

In Example \ref{Example1_first} in Section \ref{Sec_Notation} we introduced another simple example for which \ref{PMMMAR} holds but \ref{EMAR} does not. We consider here a similar setting: We assume $X=(X_1, \ldots, X_5)$ are independently uniformly distributed on $[0,1]$. We further specify three patterns
\begin{align*}
    m_1 = (0, 0, 0, 0, 0),\ \ m_2 = (0, 1, 0, 0, 0),\ \ m_3 = (1, 0, 0, 0, 0),
\end{align*}
with
\begin{align*}
        &\Prob(M=m_1 \mid X=x) = x_1/3\\
    &\Prob(M=m_2 \mid X=x) = 2/3-x_1/3\\
    &\Prob(M=m_3 \mid X=x)  = 1/3.
\end{align*}
This clearly meets \ref{PMMMAR}. However, it holds that, 
\begin{align*}
    &p(x_1 \mid X_2=x_2, X_3=x_3, M=m_1)\\
    &=p(x_1) \\
    &\neq p(x_1 \mid X_2=x_2, X_3=x_3, M=m_3)\\
    &= 2x_1 p(x_1),
\end{align*}
and \ref{EMAR} does not hold. In particular, while simply subsampling the observed $X_1$ would be enough to impute well in this example, trying to learn $p(x_1 \mid X_2=x_2, \ldots X_5=x_5)$ in the fully observed pattern will lead to a biased imputation, as shown in Figure \ref{fig:Example1}.\footnote{We note that in Example \ref{Example1} we consider only three variables. However, adding two more independent variables helps stabilize the imputation methods and gives more consistent results in our experiments, especially for GAIN and MIWAE.}

We draw $n=5000$ i.i.d. copies of $(X, M)$ and first study the imputation of $X_1$ for mice-cart and GAIN in Figure \ref{fig:X1imputation}. Interestingly, GAIN appears to impute in a way that one would expect from a method requiring \ref{EMAR}. Although this is not perfect, the boxplot of $X_1$ for the GAIN imputation seems closer to the density $2x_1 p(x_1)$, shown in the middle of Figure \ref{Example1}, than the correct $p(x_1)$. Indeed \cite{directcompetitor2} prove that GAIN identifies the correct distributions under \ref{RMAR}, which they (erroneously) call MAR. Moreover, \cite{directcompetitor1} state that ``Empirically, we often find GAN imputations working quite well for missing data under MAR or even MNAR...''. This example shows that this is not generally true and may hint at the fact that in these cases the mechanism was actually \ref{EMAR} or even \ref{CIMAR}. The conditional GAN (CGAN) of \cite{directcompetitor1}, which is shown to identify the true conditional distributions under \ref{EMAR} will likely suffer the same issues. On the other hand, the imputation of mice-cart seems very reasonable. This illustrates again both the confusion about the term ``MAR'' in the literature as well as the advantage FCS imputation has over other approaches in case of \ref{PMMMAR}.

Figure \ref{fig:Application_2_Scores} shows the negative energy distance and RMSE for this example. While there is some considerable variation, especially for GAIN, the negative energy distance recognizes mice-cart, mice-DRF and norm.nob as the best methods.\footnote{Since norm.nob draws from a Gaussian, this might be surprising, however a quick plot reveals that a Gaussian with mean 0.5 and the right variance gives a close approximation to a uniform distribution on [0,1].} RMSE, on the other hand, puts MIWAE and missForest in the first and second place and assigns poor scores to mice-DRF and mice-cart. Although this may be reasonable for MIWAE, missForest essentially just imputes the mean of $X_1$ and $X_2$. Thus, it entirely misrepresents the variation in the distribution, which clearly shows the misleading effect RMSE can have.

\begin{figure}
    \centering
    \includegraphics[width=1\linewidth]{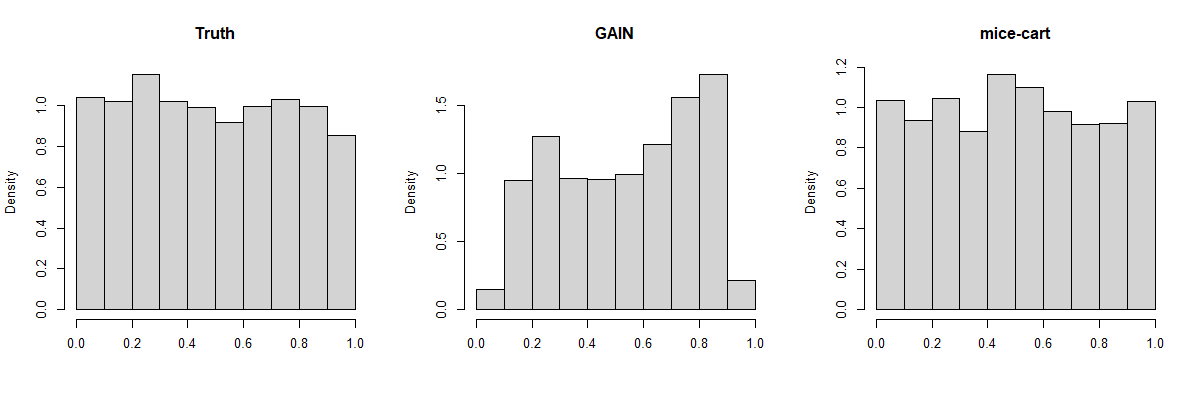}
    \caption{Imputation of $X_1$.}
    \label{fig:X1imputation}
\end{figure}

\begin{figure}
    \centering
    \includegraphics[width=1\linewidth]{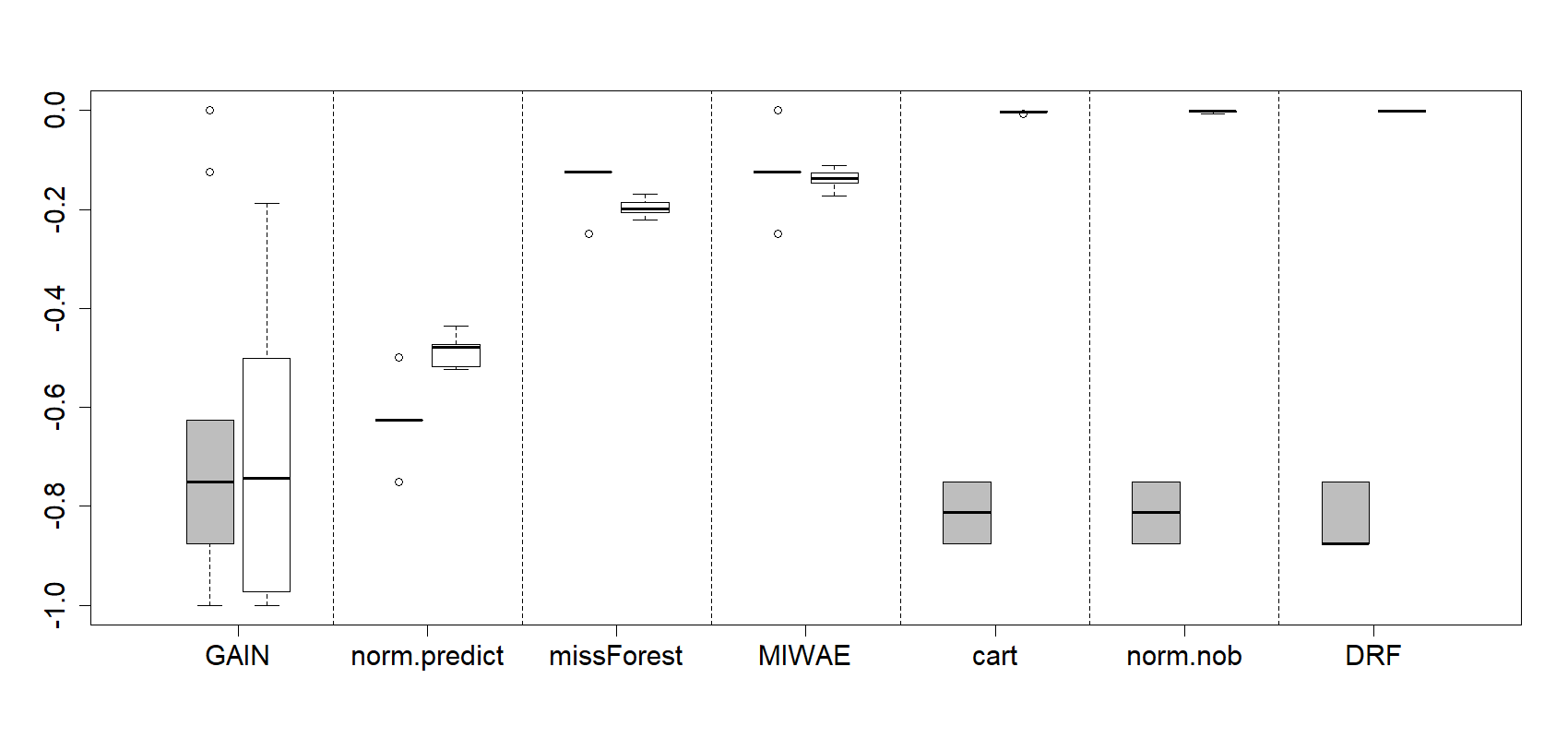}
    \caption{Standardized scores (the higher the better) for the uniform example. Methods are ordered according to the negative energy distance. For each method, negative RMSE (gray, left) and the negative energy distance (white, right) are shown.}
    \label{fig:Application_2_Scores}
\end{figure}

\subsection{Gaussian Mixture Model} \label{Sec_Gaussmixmodel}


\begin{figure}
    \centering
    \includegraphics[width=1\linewidth]{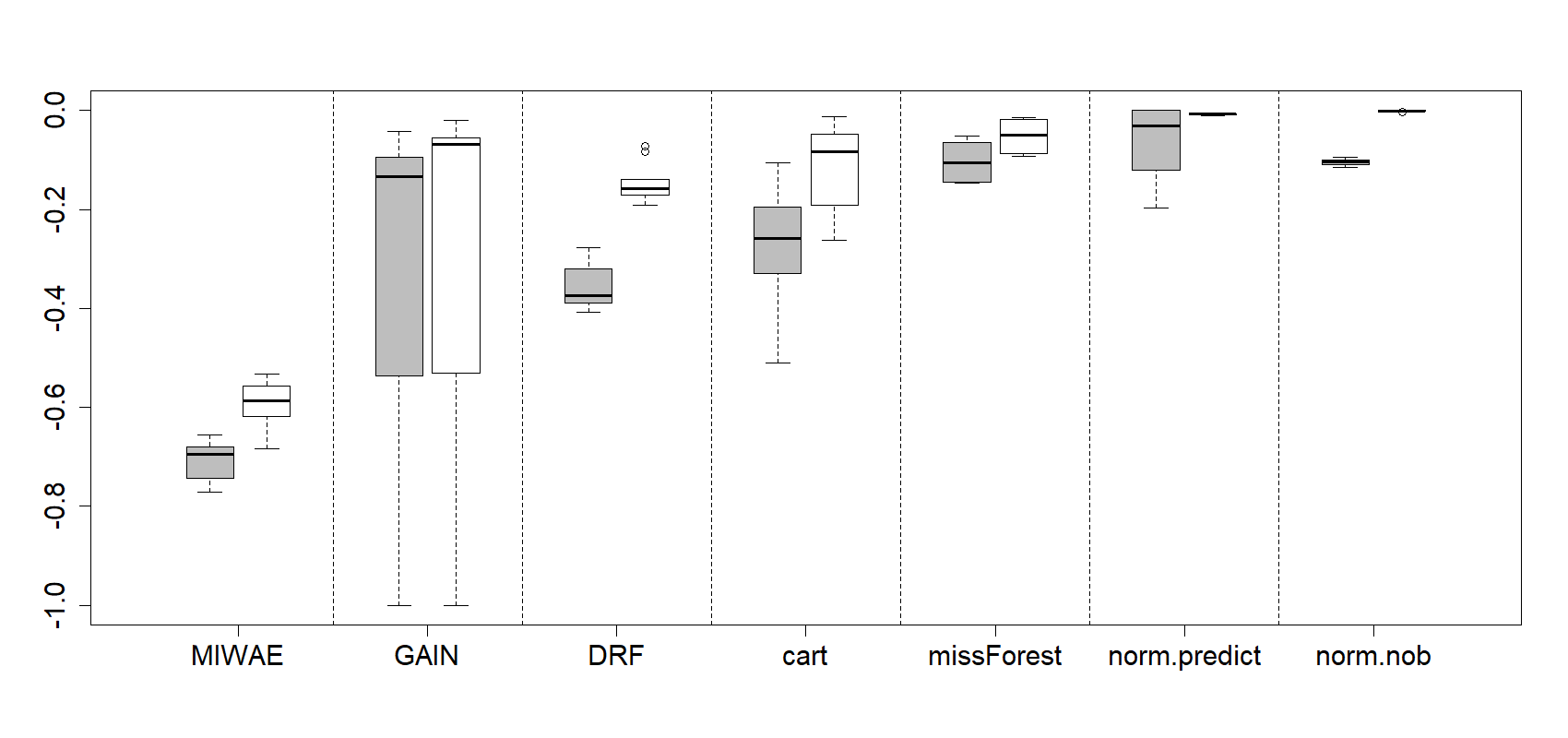}
    \caption{Standardized scores (the higher the better) for the Gaussian mixture model. Methods are ordered according to the negative energy distance. For each method, negative RMSE (gray, left) and the negative energy distance (white, right) are shown.}
    \label{fig:Application_3_Scores}
\end{figure}

We first turn to a Gaussian Mixture model to be able to put more emphasis on distribution shifts under MAR. In particular, we simulate the distribution shift of Example \ref{Example2} in higher dimensions. We take $d=6$ and 3 patterns,
\begin{align*}
    m_1 &= (1, 0, 0, 0, 0 ,0)  \\
    m_2 &= (0, 1, 0, 0, 0 ,0) \\
    m_3 &= (0, 0, 1, 0, 0 ,0)
\end{align*}
The last three columns of fully observed variables, denoted $X_{O}^*$, are all drawn from three-dimensional Gaussians with means $(5,5,5)$, $(0,0,0)$ and $(-5,-5,-5)$ respectively, and a Toeplitz covariance matrix $\Sigma$ with $\Sigma_{i,j}=0.5^{|i-j|}$. Thus there are relatively strong mean shifts between the different patterns. To preserve MAR, the (potentially unobserved) first three columns are built as 
\begin{align*}
    X_{O^c}^* = \mathbf{B}  X_{O}^* + \begin{pmatrix}
        \varepsilon_1\\
        \varepsilon_2\\
        \varepsilon_3
    \end{pmatrix} ,
\end{align*}
 where $\mathbf{B}$ is a $3 \times 3$ matrix of coefficients, $(\varepsilon_1, \varepsilon_2, \varepsilon_3)$ are independent $N(0,4)$ random errors and $O=\{4,5,6\}$ is the index of fully observed values. The increased variance increases the overlap between the distributions compared to Example \ref{Example2}. For $\mathbf{B}$ we copy the vector $(0.5,1,1.5)$ three times, such that $\mathbf{B}$ has identical rows.
The data is thus Gaussian with linear relationships, but there is a strong distribution shift between the different patterns. However, this distributional shift only stems from the observed variables, leaving the conditional distributions of missing given observed unchanged, as in Example \ref{Example2}. Consequently, it can be shown that the missingness mechanism meets \ref{CIMAR} and is thus MAR.

For each pattern, we generate 500 observations, resulting in $n=1.500$ observations and around $17\%$ of missing values. In this example, we expect that the imputation able to adapt to shift in covariates will perform well, even if they are not able to catch complex dependencies between variables. Indeed, we note that $P$ corresponds to the Gaussian imputation (mice-norm.nob) with the (unknown) true parameters. As such, a good evaluation method should rank mice-norm.nob highest. In contrast, the forest-based scores should have the worst performance here, as they may not able to deal properly with the distribution shift. 
On the other hand, they might still be deemed better than mice-norm.predict, which only imputes the regression prediction.

Results for the (standardize) negative energy distance and RMSE are given in Figure \ref{fig:Application_3_Scores}. The ordering induced by the negative energy distance is as expected, with mice-norm.nob in first place, and forest-based methods struggling. However, the increased variance appears to increase the performance of forest-based methods and missForest performs surprisingly well in terms of negative energy distance, though still markedly worse than norm.nob and mice-norm.predict. Finally, even though this example is \ref{CIMAR}, both MIWAE and GAIN perform worse than the forest-based methods in terms of energy distance, indicating that they also struggle with the distribution shift. On the other hand, GAIN performs quite well in some cases, but suffers from an enormous variance. This may be stabilized with better parameter tuning. RMSE in turn, ranks mice-norm.predict as the best method. Thus, RMSE is again unable to identify the method that best replicates $P$ and instead chooses the one that best estimates the conditional expectation.

\subsection{Mixture Model with Nonlinear Relationships}\label{Sec_Gaussmixmodelnonlinar}

We now turn to a more complex version of the model in Section \ref{Sec_Gaussmixmodel} to add nonlinear relationships to the distributional shifts. This example should indicate that the search for successful imputation methods is by no means complete.

Using the same missingness pattern, and Gaussian variables $X_O$ we use a nonlinear function $f$ for the conditional distribution:
\begin{align}\label{nonlinearshiftexample}
    X_{O^c}^* =  f(X_{O}^*) + \begin{pmatrix}
        \varepsilon_1\\
        \varepsilon_2\\
        \varepsilon_3
    \end{pmatrix},
\end{align}
with
\begin{align*}
    &f(x_1,x_2,x_3)\\
    &=(x_3 \sin(x_1 x_2), x_2 \cdot \Ind\{x_2 > 0\}, \arctan(x_1) \arctan(x_2)).
\end{align*}
This introduces nonlinear relationships between the elements of $X_{O^c}^*$ and $X_O^*$, though the conditional distribution of $X_{O^c}^* \mid X_{O}^*$ is still Gaussian and the missingness mechanism is CIMAR. For each pattern, we generate 500 observations, resulting in $n=1'500$ observations and around $17\%$ of missing values.

In this example, the ability to generalize is important, and so is the ability to model nonlinear relationships. Accordingly, this is a very difficult example and there is no clear ideal method. Despite the distribution shift, mice-DRF performs best in terms of negative energy distance, directly followed by mice-cart. This relates to earlier findings in \cite{ImputationScores} and others that mice-cart might be one of the most promising methods. missForest again scores surprisingly high, though it performs markedly worse than mice-cart. GAIN is put last, but again displays very high variability. Finally, the RMSE is surprisingly close to the ranking of negative energy distance, though it again scores missForest highest.

\begin{figure}
    \centering
    \includegraphics[width=1\linewidth]{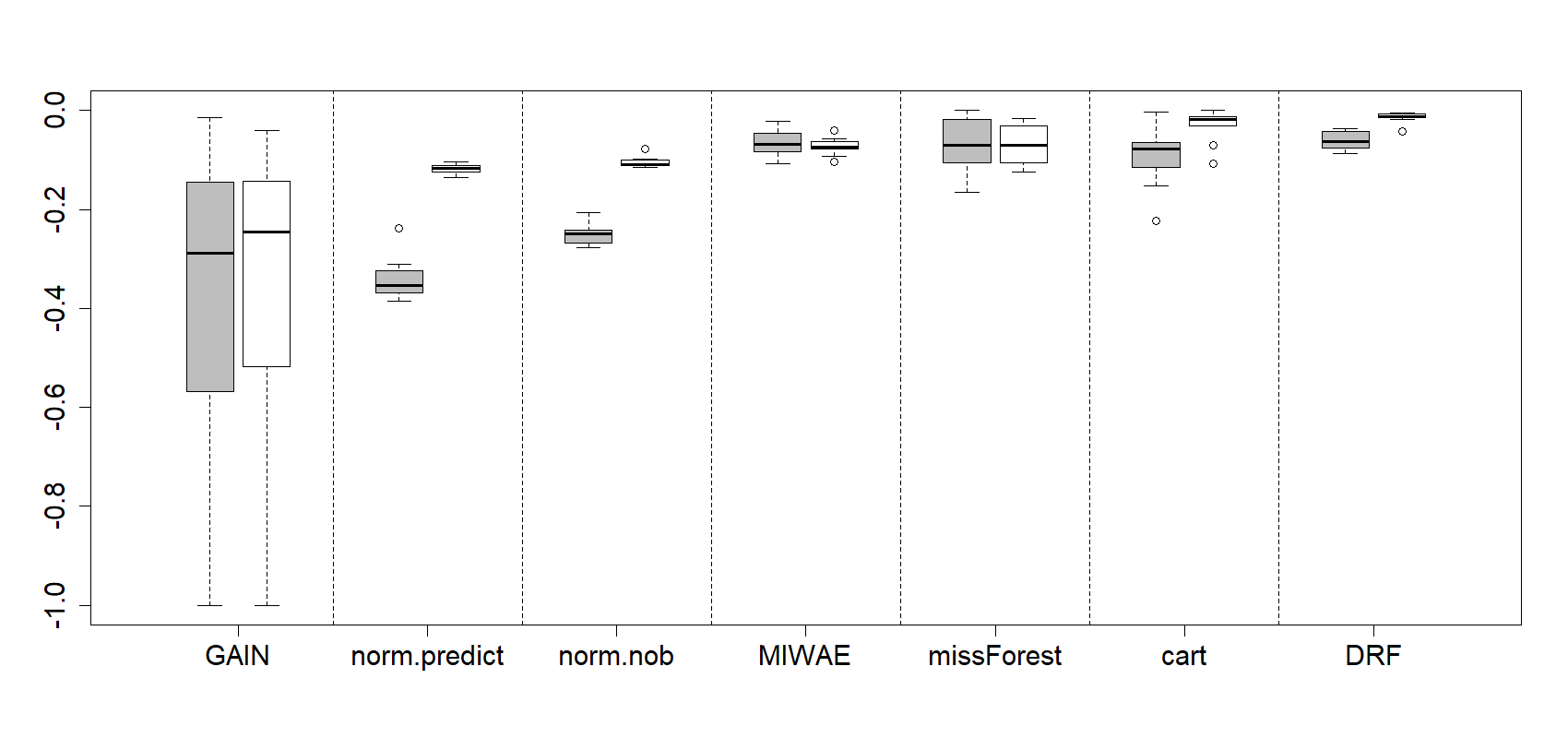}
    \caption{Standardized scores (the higher the better) for the nonlinear mixture model. Methods are ordered according to the negative energy distance. For each method, negative RMSE (gray, left) and the negative energy distance (white, right) are shown.}
    \label{fig:Application_4_Scores}
\end{figure}

\section{Proofs and Additional Results}

In this section, we provide additional results and collect the proofs of the results not shown in the main paper. We first provide details on Example \ref{interesting_new_Example} and \ref{interesting_new_Example_adaptation} in the main text.

\subsection{Details on Examples \ref{interesting_new_Example} and \ref{interesting_new_Example_adaptation}}\label{Sec_AdditionalexampleDetails}

Here we provide details on Example \ref{interesting_new_Example} and \ref{interesting_new_Example_adaptation}, whereby the latter is also used in Section \ref{Sec_Empirical}. In both cases, we use a Farlie-Gumbel-Morgenstern (FGM) copula ( see e.g., \cite{Nelsen2006}), that is, $X_1$, $X_2$ and $X_3$ are marginally uniform on [0,1], and
\begin{align*}
    \pX(x_1,x_2)&=1 + (2x_1 - 1)(2x_2 - 1) \Ind\{(x_1,x_2) \in [0,1]^2\}\\
     \pX(x_1 \mid X_2=x_2)&=1 + (2x_1 - 1)(2x_2 - 1) \Ind\{x_1 \in [0,1]\}\\
     \pX(x_2 \mid X_1=x_1)&=1 + (2x_1 - 1)(2x_2 - 1) \Ind\{x_2 \in [0,1]\}
\end{align*}
We can moreover derive that,
\begin{align*}
    \E[X_2 \mid X_1]=\frac{1}{2} + \frac{1}{3} \left(X_1-\frac{1}{2}\right)
\end{align*}
We then recall that for Example \ref{interesting_new_Example},
    \begin{align*}
        \mathcal{M}&=\{m_1,m_2, m_3, m_4\}\\
        &=\{(0,0,0), (0,1,0), (0,0,1), (1,1,0)\}
    \end{align*}
 and
       \begin{align*}
        &\Prob(M=m_1\mid X=x)=(x_1+x_2)/3\\ 
        &\Prob(M=m_2\mid X=x)= (1-x_1)/3 \\
        & \Prob(M=m_3\mid X=x)= (1-x_2)/3 \\ 
        &\Prob(M=m_4\mid X=x)= 1/3. 
    \end{align*}

For Example \ref{interesting_new_Example_adaptation}, we chose:
 $$\mathcal{M}=\{m_1,m_2, m_3\}=\{(0,0,0), (0,1,0), (1,0,0)\}$$ and
    \begin{align*}
        &\Prob(M=m_1\mid X=x)=(x_1+x_2)/3,\\
        &\Prob(M=m_2\mid X=x)=(2-x_1)/3\\
        &\Prob(M=m_3\mid X=x)=(1-x_2)/3.
    \end{align*} 

To detail the properties of the two examples, we will use the relation
\begin{align}\label{usefulrelation}
    &p( o(x,m)^c \mid o(X,m)=o(x,m), M \in L)  \\
    &=\pX( o(x,m)^c \mid o(X,m)=o(x,m)) \nonumber\\
    &\frac{\sum_{m' \in L}\Prob(M=m' \mid X=x)}{\sum_{m' \in L}\Prob(M=m' \mid o(X,m)=o(x,m))}\nonumber,
\end{align}
for any $L \subset \mathcal{M}$.

\paragraph*{Example \ref{interesting_new_Example}} Using the above relation, we immediately can verify that
\begin{align*}
    &p(x_1,x_2 \mid X_3=x_3, M=m_1)\\
    &=\pX(x_1,x_2 \mid X_3=x_3) (x_1+x_2)/2\\ 
    &p(x_1,x_2 \mid X_3=x_3, M=m_3)\\
    &=\pX(x_1,x_2 \mid X_3=x_3) (1-x_2)/2 
\end{align*}
and
\begin{align*}
    h^*(x_1 \mid X_2=x_2, X_3=x_3)=p(x_1 \mid X_2=x_2, X_3=x_3),
\end{align*}
as $\Prob(M=m_1 \mid X=x) + \Prob(M=m_2 \mid X=x)+\Prob(M=m_3 \mid X=x)=1/3$ by construction. We can then consider the permutation $\pi(j)=j$. Then $\Prob(M_1=1 \mid X=x)=\Prob(M=m_4 \mid X=x)=1/3$. Similarly, for the last element: $\Prob(M_3=1 \mid X=x, M_1=m_1, M_2=m_2)=\Ind\{m_1=0,m_2=0\}$. Finally for $M_2$,
\begin{align*}
    &\Prob(M_2=1 \mid X=x, M_1=m_1)\\
    &=\begin{cases}
        (1-x_1)/2, \text{ if } m_1=0\\
        1, \text{ if } m_1=1.
    \end{cases}
\end{align*}

\paragraph*{Example \ref{interesting_new_Example_adaptation}}
We first note that
    \begin{align*}
        &\Prob(M=m_1\mid X=x)=(x_1+x_2)/3,\\
        &\Prob(M=m_2\mid X=x)=(2-x_1)/3\\
        &\Prob(M=m_3\mid X=x)=(1-x_2)/3.
    \end{align*}
As such, we can verify that
\begin{align*}
    \Prob(M_1=1 \mid M_2=0, M_3=0, X=x)&=(1-x_2)/(1+x_1)\\
    \Prob(M_2=1 \mid M_1=0, M_3=0, X=x)&=(2-x_1)/(2+x_2).
\end{align*}
Note that both of these expressions depend on the missing variables. Thus, for any permutation $\pi$, $\Prob(M_{\pi(j)}=1 \mid \text{Pa}(M_{\pi(j)} )$ depends on $X_{\pi(j)}$ and \eqref{graphicalMAR} does not hold.

Finally, when considering the joint distribution of $(X_1,X_2)$ when $M_1=1$, we may use \eqref{usefulrelation} once again to obtain: 
\[
p(x_1,x_2 \mid M_1=0) = \pX(x_1,x_2) \frac{2}{5}(2+x_2).
\]
To obtain the distribution of $X_1 \mid M_1=0$, we integrate this expression over $x_2$, leading to
\begin{align*}
    &p(x_1 \mid M_1=0) \\
    &=2/5 \pX(x_1) (2+  \int x_2 \pX(x_2 \mid x_1) d x_2)\\
    &=2/5 \pX(x_1) (2+ \E[X_2 \mid X_1=x_1])\\
    &=2/5 \pX(x_1) (7+x_1)/3\\
    &=2/5 (7+x_1)/3 \Ind\{x_1 \in [0,1]\}.
\end{align*}
This distribution has cdf,
\begin{align*}
F(x_1)=\int_{0}^{x_1} \pX(\tilde{x}_1) d\tilde{x}_1=(14x_1 + x_1^2)/15.   
\end{align*}
Inverting this expression, we obtain the quantile,
\begin{align}
    Q(\alpha)=-7+\sqrt{49+15 \alpha}.
\end{align}
In particular, this is around $0.106$, when $\alpha=0.1$.

\subsection{Proofs}\label{Sec_Proofs}

As mentioned above, to avoid measurability issues, we assume that $p/\pX$ are such that $x \in \X$ implies $\pX(o(x,m')) > 0$ and $x \notin \X_{\mid m}^{m'}$ implies $p(o(x,m') \mid M=m)=0$ for all $m,m'$. Moreover, in the proofs we shorten $p(o^c(x,m) \mid o(X,m)=o(x,m))$ to $p(o^c(x,m) \mid o(x,m))$ and similarly with $\Prob(M=m \mid X=x)$ and other expression, for convenience.


We consider the first two results together:

\PMMMARzero*

\PMMMAR*

\begin{proof}
   
We show that \ref{SMAR} implies \ref{PMMMAR}. Recall that \ref{SMAR} means that for $x, \tilde{x} \in \X$ such that $o(x,m)=o(\tilde{x},m)$, $\Prob(M=m | x) = \Prob(M=m| \tilde{x}) $. First consider, $x \in \X_{\mid m}^{m}$ and $\tilde{x} \in \X$ such that $o(x,m)=o(\tilde{x},m)$. Since, $p(o(x,m) \mid M=m) = p(o(\tilde{x},m) \mid M=m) > 0$,  $x \in \X_{\mid m}^{m}$ implies that $\tilde{x} \in \X_{\mid m}^{m}$. Thus,
\begin{align}\label{proofMARCondtrue2}
    &\Prob(M=m | x) = \Prob(M=m| \tilde{x}) \Longleftrightarrow  \\
    & \frac{p(x  |M =m )\Prob(M =m )}{p(x )} \nonumber\\
    &= \frac{p(\tilde{x}  |M =m )\Prob(M =m )}{{p(\tilde{x} )}} \Longleftrightarrow \nonumber\\
    & \frac{p(o(x ,m ), o^c(x ,m ) \mid M  =m )}{p(o(\tilde{x} ,m ), o^c(\tilde{x} ,m ) \mid M  =m )} \nonumber\\
    &= \frac{p(o(x , m ), o^c(x ,m ))}{p(o(\tilde{x} , m ), o^c(\tilde{x} ,m ))} \Longleftrightarrow \nonumber \\
    & \frac{p(o^c(x ,m ) \mid o(x ,m ), M =m )}{p(o^c(x ,m )\mid o(x ,m ))} \nonumber\\
    &= \frac{p(o^c(\tilde{x} ,m ) \mid o(x ,m ), M =m )}{p(o^c(\tilde{x} ,m )\mid o(x ,m ))}  \Longleftrightarrow\nonumber\\
    & p(o^c(x ,m ) \mid o(x ,m ), M =m ) \nonumber\\
    &= \frac{p(o^c(\tilde{x} ,m ) \mid o(x ,m ), M =m )}{p(o^c(\tilde{x} ,m )\mid o(x ,m ))} p(o^c(x ,m )\mid o(x ,m ))\nonumber
\end{align}
Integrating \eqref{proofMARCondtrue2} with respect to the missing part of $x$, $o^c(x,m)$ only shows that 
\[
\frac{p(o^c(\tilde{x} ,m ) \mid o(x ,m ), M =m )}{p(o^c(\tilde{x} ,m )\mid o(x ,m ))}=1,
\]
and thus also \ref{PMMMAR}. Similarly, starting from \ref{PMMMAR}, the same arguments backwards imply that for $x \in \X_{\mid m}^m$, $\tilde{x} \in \X$ such that $o(x,m)=o(\tilde{x},m)$, $\Prob(M=m | x) = \Prob(M=m| \tilde{x}) $. 

On the other hand, for $x \in  \X \setminus \X_{\mid m}^{m}$, $\tilde{x} \in \X$, with $o(x,m)=o(\tilde{x},m)$, $p(o(x,m) \mid M=m)=p(o(\tilde{x},m) \mid M=m)=0$, so that $p(\tilde{x} \mid M=m)=p(x \mid M=m)=0$ by the assumption at the beginning of this section. Thus
\begin{align*}
    0=p(x \mid M=m)=\Prob(M=m \mid x)\frac{p(x)}{\Prob(M=m)}\\
    0=p(\tilde{x} \mid M=m)=\Prob(M=m \mid \tilde{x})\frac{p(\tilde{x})}{\Prob(M=m)}
\end{align*}
showing that $\Prob(M=m \mid x)=\Prob(M=m \mid \tilde{x})=0$ and proving \eqref{SMAR}.

This shows that \ref{SMAR} and \ref{PMMMAR} are equivalent. \cite{ourresult} show that \ref{SMARII} is also equivalent to \ref{PMMMAR}, which proves the result.  
\end{proof}

\Ordering*

\begin{proof}
   Clearly if \ref{CIMAR} implies \ref{EMAR} and \ref{EMAR} implies \ref{PMMMAR}.

    We already studied examples showing that \ref{PMMMAR} may hold, even if \ref{EMAR} and \ref{CIMAR} do not, for instance Examples \ref{Example1_first} and \ref{interesting_new_Example}. Example \ref{Example6} below gives an example where \ref{EMAR} holds but not \ref{CIMAR}.
\end{proof}

\begin{example}\label{Example6}[\ref{EMAR} is strictly weaker than \ref{CIMAR}]
    Consider
\begin{align}\label{eq:Example3}
\mathcal{M}= \left\{\begin{pmatrix} 0 & 0 & 0 & 0\end{pmatrix}, \begin{pmatrix}1 & 0 & 0 & 0\end{pmatrix}, \begin{pmatrix} 1 & 0 & 1 & 0\end{pmatrix}, \begin{pmatrix} 0 & 1 & 1 & 0\end{pmatrix} \right \},
\end{align}
and $(X_1, X_2, X_3, X_4)$ are independently uniformly distributed on $[0,1]$. We construct the example such that $\Prob(M=m_1 \mid x)$ and $\Prob(M=m_4 \mid x)$ depend on the always observed $x_4$, while $\Prob(M=m_2 \mid x)$ and $\Prob(M=m_3 \mid x)$ depends on $x_2$, which is not observed in pattern $m_4$. This will make the example \ref{EMAR}, but not \ref{CIMAR}. More specifically, let
\begin{align*}
    &\Prob(M=m_1 \mid x)= \frac{x_4}{2}, \ \ \Prob(M=m_4 \mid x)= \frac{1}{2} - \frac{x_4}{2}  \\
    &\Prob(M=m_2 \mid x)= \frac{x_2}{2}, \ \  \Prob(M=m_3 \mid x)=\frac{1}{2} - \frac{x_2}{2}.
\end{align*}
By construction it holds that, with $m_1$ being the fully observed pattern and $m \in \{m_2,m_3, m_4\}$,
\begin{align*}
    &\Prob(M=m_1 \mid o(x,m) )\\
    &= \int \Prob(M=m_1 \mid x_4 ) p(o^c(x,m) \mid o(x,m)) d o^c(x,m)\\
    &=\Prob(M=m_1 \mid x_4 ),
\end{align*}
as $x_4$ is not part of $o^c(x,m)$ for any pattern $m$. Thus,
\begin{align*}
    &p(o^c(x,m) \mid o(x,m), M=m_1 )\\
    &=p(o^c(x,m) \mid o(x,m) ) \frac{\Prob(M=m_1 \mid x )}{\Prob(M=m_1 \mid o(x,m) )} \\
    &=p(o^c(x,m) \mid o(x,m) ),
\end{align*}
and \ref{EMAR} holds. On the other hand,
\begin{align*}
    &\Prob(M=m_2 \mid o(x,m_4))\\
    &= \int \Prob(M=m_2 \mid x_2 ) p(x_2, x_3 \mid x_1,x_4) dx_2 d x_3\\
    &=  \frac{1}{4}. 
\end{align*}
Thus,
\begin{align*}
    &p(o^c(x,m_4) \mid o(x,m_4), M=m_2 )\\
    &= p(o^c(x,m_4) \mid o(x,m_4) ) \frac{\Prob(M=m_2 \mid x )}{\Prob(M=m_2 \mid o(x,m_4) )}\\
    &= 2 x_2 p(o^c(x,m_4) \mid o(x,m_4) )\\
    &\neq p(o^c(x,m_4) \mid o(x,m_4) )\\
    &= p(o^c(x,m_4) \mid o(x,m_4), M=m_4 ),
\end{align*}
showing that the example is not \ref{CIMAR}.
\end{example}

\CIMARequivRMAR*

\begin{proof}

First, we notice that under \ref{CIMAR} for all $m, m' \in \mathcal{M}$, $x \in \X_{\mid m'}^{m}$,
\begin{align*}
    &\pX(o^c(x,m) \mid o(x,m)) \frac{\Prob(M=m' \mid x)}{\Prob(M=m' \mid o(x,m))}\\
    &=p(o^c(x,m) \mid o(x,m), M=m')\\
    &=\pX(o^c(x,m) \mid o(x,m)).
\end{align*}
Moreover, if $x \in \X \setminus \X_{\mid m}^{m'}$, we have
\[
0=p(o(x,m) \mid M=m') \implies p(x \mid M=m')=0,
\]
so that 
\begin{align*}
    P(M=m' \mid x)&=p(x \mid M=m') \frac{\Prob(M=m')}{\pX(x)}\\
    &=0\\
    &=P(M=m' \mid o(x,m)).
\end{align*}
Thus \ref{CIMAR} implies that for all $m, m' \in \mathcal{M}$, and all $x \in \X$, $\Prob(M=m' \mid x)=\Prob(M=m' \mid o(x,m))$. Consequently if a variable $X_j^*$ is observed in pattern $m'$, but not in pattern $m$, \ref{CIMAR} implies that $\Prob(M=m' \mid x)$ is not allowed to depend on $x_j$. Thus, $\Prob(M=m' \mid x)$ can only depend on variables that are observed in all patterns $m$, i.e. subsets of $X_{O}^*$. This shows that \ref{CIMAR} implies \ref{RMAR}.

Similarly since under \ref{RMAR}, $x_{O}$ is part of $o(x,m)$ for all $m$,
\begin{align*}
    &\Prob(M=m' \mid o(x,m))\\
    &=\int \Prob(M=m' \mid x) p(o^c(x,m) \mid o(x,m)) do^c(x,m)\\
    &=\Prob(M=m' \mid x_{O}) \int  p(o^c(x,m) \mid o(x,m)) do^c(x,m)\\
    &=\Prob(M=m' \mid x_{O}).
\end{align*}
Thus, for all $x \in \X_{\mid m'}^{m}$,
\begin{align*}
    &p(o^c(x,m) \mid o(x,m), M=m')\\
    &=p(o^c(x,m) \mid o(x,m))\frac{\Prob(M=m' \mid x)}{\Prob(M=m' \mid o(x,m))} \\
    &=p(o^c(x,m) \mid o(x,m))\\
    &=p(o^c(x,m) \mid o(x,m))\frac{\Prob(M=m \mid x)}{\Prob(M=m \mid o(x,m))}\\
    &=p(o^c(x,m) \mid o(x,m), M=m),
\end{align*}
and \ref{CIMAR} holds.

For \ref{EMAR}, taking $m \in \mathcal{M}$ and $m'=0$, the same arguments show that $\Prob(M=0\mid x)=\Prob(M=0\mid x_O)$ for all $x \in \X$.

If $O$ is empty, then under \ref{RMAR}, $\Prob(M=m \mid x)=\Prob(M=m \mid x_{O})=\Prob(M=m)$ for all $m$, proving that \ref{CIMAR} simplify to \ref{MCARform}. Similarly, for \ref{EMAR}, $\Prob(M=0 \mid x)=\Prob(M=m)$.

For nonempty $O$, Example \ref{Example2} in the main text shows an example that is \ref{CIMAR} but not \ref{MCARform}, showing that, in this case, \ref{CIMAR} is strictly weaker than \ref{MCARform}.

\end{proof}

\propidentifiabilitymar*

\begin{proof} 
By definition of \ref{CIMAR} and \ref{EMAR} it directly follows that for all $x \in \X_{\mid m'}^{m}$,
\begin{align*}
p(o^c(x,m) \mid o(x,m), M=m') = p(o^c(x,m) \mid o(x,m)),
\end{align*}
for all $m' \in \mathcal{M}$ and, for all $x \in \X_{\mid 0}^{m}$,
\begin{align*}
p(o^c(x,m) \mid o(x,m), M=0) = p(o^c(x,m) \mid o(x,m)),
\end{align*}
respectively. Example \ref{Example1} shows that under \ref{PMMMAR}, 
\begin{align*}
p(o^c(x,m) \mid o(x,m), M=m') = p(o^c(x,m) \mid o(x,m)),
\end{align*}
might not hold for any $m' \neq m$. Finally Example \ref{Example5} shows that there exists a MAR situation, where for some $m \in \mathcal{M}$,
\begin{align*}
    &h^*(o^c(x,m) \mid o(x,m)) \\
    &= \sum_{m' \in L_m} w_{m'}(o(x,m)) p(o^c(x,m) \mid o(x,m), M=m'),
\end{align*}
is not equal to $p(o^c(x,m) \mid o(x,m))$ for any set of weights $w_{m'}(o(x,m))$.
\end{proof}

\begin{example}\label{Example5}
    
Consider
\begin{align}\label{eq:Example5}
\mathcal{M}= \left\{\begin{pmatrix} 0 & 0 & 0 & 0\end{pmatrix}, \begin{pmatrix} 0 & 0 & 1 & 0\end{pmatrix}, \begin{pmatrix} 0 & 1 & 0 & 0\end{pmatrix}, \begin{pmatrix}1 & 1 & 0 & 0\end{pmatrix}  \right \},
\end{align}
and $(X_1, X_2, X_3, X_4)$ are independently uniformly distributed on $[0,1]$. We further specify that 
\begin{align*}
    \Prob(M=m_1 \mid x) &= \Prob(M=m_1 \mid x_1,x_2) = (x_1+x_2)/8\\
    \Prob(M=m_2 \mid x) &= \Prob(M=m_2 \mid x_2) = 1/4-x_2/8\\
        \Prob(M=m_3 \mid x) &= \Prob(M=m_3 \mid x_1) = 1/4-x_1/8\\
    \Prob(M=m_4 \mid x) &= \Prob(M=m_4) = 2/4.
\end{align*}
Again, $\Prob(M=m \mid x)= \Prob(M=m \mid o(x,m))$ and thus the MAR condition \ref{SMAR} holds. Consider now $m=m_4$, such that $o^c(x,m)=(x_1,x_2)$. Then it holds that
\begin{align*}
    &p(o^c(x,m) \mid o(x,m), M=m_1)\\
    &=\frac{1}{2}\left(x_1 + x_2\right) p(x_1,x_2 \mid x_3, x_4)\\
    &p(o^c(x,m) \mid o(x,m), M=m_2)\\
    &=(2-x_2) p(x_1,x_2 \mid x_3, x_4)
\end{align*}
Consider the mixture as in \eqref{h0star},
\begin{align*}
   & h^*(o^c(x,m) \mid o(x,m)) \\
    &= \sum_{m' \in L_m} w_{m'}(o(x,m)) p(o^c(x,m) \mid o(x,m), M=m'),
\end{align*}
with $\sum_{m' \in L_m} w_{m'}(o(x,m))=1$. Then for $ h^*(o^c(x,m) \mid o(x,m))= p(o^c(x,m) \mid o(x,m))$ to hold, it must hold that for all $x_1, x_2$,
\begin{align*}
    w_{m_1}(x_3,x_4) \frac{1}{2}\left(x_1 + x_2\right) +  w_{m_2}(x_3,x_4) (2-x_2)=1.
\end{align*}
It is impossible to find a set of weights that meet this condition for all $x_1, x_2$ simultaneously. 
\end{example}

\identificationprop*

\begin{proof}
Recall that 
\begin{align} 
        L_{j}=\{m \in \mathcal{M}: m_j=0  \}.
    \end{align}
As all previous variables have been imputed and $x_j$ is observed, it is thus possible to identify the full distribution $p(x \mid M=m)$ for all $m \in L_j$. Thus, we learn the mixture of joint distributions
\begin{align*}
  p(x_j, x_{-j} \mid M \in L_j)&=\frac{1}{C} \sum_{m \in L_j} \Prob(M=m) \cdot p(x \mid M=m )\\
  &=\frac{1}{C} \sum_{m \in L_j} \Prob(M=m \mid x) \cdot p( x),
\end{align*}
where $C$ is a constant such that $p(x_j, x_{-j}\mid M \in L_j)$ integrates to 1. Integrating $p(x_{j}, x_{-j}\mid M \in L_j)$ over $x_{j}$, we obtain similarly
\begin{align*}
     p( x_{-j}\mid M \in L_j) = \frac{1}{C} \sum_{m \in L_j} \Prob(M=m \mid  x_{-j}) \cdot p(   x_{-j})
\end{align*}
Thus for $x_{-j}$ such that $p(x_{-j} \mid M_j=0)=p( x_{-j}\mid M \in L_j)>0$:
\begin{align*}
    h^*(x_j \mid x_{-j}) &= \frac{p(x_j, x_{-j}\mid M \in L_j)}{p( x_{-j}\mid M \in L_j)}\\
    &=\frac{\sum_{m \in L_j} \Prob(M=m \mid x) \cdot p( x)}{\sum_{m \in L_j} \Prob(M=m \mid  x_{-j}) \cdot p(  x_{-j})}\\
    &=p( x_j \mid  x_{-j}) \frac{\sum_{m \in L_j} \Prob(M=m \mid x)}{\sum_{m \in L_j} \Prob(M=m \mid  x_{-j} ) }.
\end{align*}
It only remains to show that
\begin{align}\label{whatwewant0}
    \frac{\sum_{m \in L_j} \Prob(M=m \mid x)}{\sum_{m \in L_j} \Prob(M=m \mid x_{-j}) }=1.
\end{align}
Indeed, we note that for any $m \in L_j^c$,
\begin{align*}
    \Prob(M=m \mid x)=\Prob(M=m \mid x_{-j}),
\end{align*}
by \ref{SMARII}. Consequently,
\begin{align*}
    1= \sum_{m \in L_j} \Prob(M=m \mid x) + \sum_{m \in L_j^c} \Prob(M=m \mid  x_{-j}),
\end{align*}
so that
\begin{align*}
    \sum_{m \in L_j} \Prob(M=m \mid x) &= 1- \sum_{m \in L_j^c} \Prob(M=m \mid x_{-j})\\
    &= \sum_{m \in L_j} \Prob(M=m \mid  x_{-j}),
\end{align*}
and thus \eqref{whatwewant0} indeed holds.
\end{proof}

\overlaplemma*

\begin{proof}

First, $p(x_{-j} \mid M_j=1) > 0$ implies $p(x_{-j}) \geq p(x_{-j} \mid M_j=1) \Prob(M_j=1) > 0$. Moreover,
\begin{align*}
p(x_{-j} \mid M_j=0)&=\Prob(M_j=0 \mid x_{-j}) \frac{p(x_{-j})}{\Prob(M_j=0)} \\
&=\sum_{m \in L_j} \Prob(M=m \mid x_{-j}) \frac{p(x_{-j})}{\Prob(M_j=0)}.
\end{align*}
Thus if \eqref{Sufficientforoverlap} holds,
\begin{align*}
    p(x_{-j} \mid M_j=1) > 0 &\implies p(x_{-j}) \\
    &\implies p(x_{-j} \mid M_j=0) > 0.
\end{align*}
    
\end{proof}

\EMARprop*

\begin{proof}
    As shown in Lemma \ref{CIMARequivRMAR}, \ref{EMAR} is equivalent to $\Prob(M=0 \mid x)=\Prob(M=0 \mid x_O)$. Thus, $\PX$ is recoverable under \ref{EMAR}. We show that for all $m$, $\Prob(M=m\mid x)$ is identifiable. Indeed,
\begin{align*}
    \Prob(M=m\mid x)=p(x \mid M=m) \frac{\Prob(M=m)}{\pX(x)}.
\end{align*}
We already know that $\pX(x)$ is identifiable from above and so is $\Prob(M=m)$. Moreover,
\begin{align*}
    &p(x \mid M=m)\\
    &=p(o^c(x,m) \mid o(x,m), M=m) p(o(x,m) \mid M=m)\\
    &=p(o^c(x,m) \mid o(x,m), M=0) p(o(x,m) \mid M=m)
\end{align*}
Now, $p(o^c(x,m) \mid o(x,m), M=0)$ is identified because it is in the observed pattern, while $p(o(x,m) \mid M=m)$ is observed in pattern $m$ and thus identified.

Since $P(M=m\mid x)$ is identified for all $m$, $p(x,m)$ is identified through
\begin{align*}
    p(x,m)=\frac{p(x,M=0)}{\Prob(M=0 \mid x)} \Prob(M=m \mid x),
\end{align*}
\end{proof}

\subsection{Ignorability in Maximum Likelihood Estimation}\label{Sec_likelihoodignorarbility}

\subsubsection{The traditional view on MLE}

In the context of MLE, it has long been established \citep{Rubin_Inferenceandmissing} that the missing mechanism is ignorable under MAR and an additional condition. To formalize this assume $p$ is parametrized by a vector $\theta$. Moreover, assume the conditional distribution of $M \mid x$ is parametrized by $\phi$. Then we can rewrite the MAR definition in \ref{SMAR} slightly, as in \cite{Rubin_Inferenceandmissing, whatismeant3}:  
 \begin{align}\label{FinaltrueMAR}
&\Prob_{\phi}(M=m | x) = \Prob_{\phi}(M=m| \tilde{x}) \text{ for all } m \in \mathcal{M}  \\
&\text{ and } x, \tilde{x} \text{ such that } o(x,m)=o(\tilde{x},m) \nonumber.
\end{align}
As so far, $\phi$ and $\theta$ are not restricted to be finite-dimensional, this could in principle be assumed without loss of generality, such that \eqref{FinaltrueMAR} is indeed the same as condition \ref{SMAR}. In the following, we will assume that $\theta$ is finite-dimensional. Let $\Omega_{\theta}$ be the space of $\theta$, $\Omega_{\phi}$ the space of possible $\phi$ and $\Omega_{\theta, \phi}$ the joint space of the parameters. The crucial additional condition is that:
\begin{align}\label{crucialadditionalcond}
  \Omega_{\theta, \phi}=\Omega_{\theta} \times \Omega_{\phi}.  
\end{align}
This just means that $\phi$ is distinct from $\theta$, so that $\Prob_{\phi}(M=m | x)$ does not depend on $\theta$ \citep{Rubin_Inferenceandmissing, whatismeant, whatismeant3}. In this case, we can rederive the classical ignorability result for MAR in a likelihood context: Consider the likelihood for a pattern $m$,
\begin{align*}
    \mathcal{L}(\theta; o(x,m))&=p_{\theta, \phi}(o(x,m) , M=m)\\
    &= \int p_{\theta, \phi}(x , M=m) do^c(x,m).
\end{align*}
That is, $\mathcal{L}(\theta; o(x,m))$ is the joint density of the observed values with respect to pattern $m$, and $M=m$, seen as a function of $\theta$. Under \eqref{FinaltrueMAR} and \eqref{crucialadditionalcond} it can be checked that
\begin{align*}
    &\int p_{\theta, \phi}(x , M=m) do^c(x,m) \\ 
    &= \Prob_{\phi}(M=m \mid o(x,m)) p_{\theta}^*(o(x,m))\\
    &= c(o(x,m))  p_{\theta}(o(x,m)),
\end{align*}
with $c(o(x,m))$ not depending on $\theta$. Consequently, it is possible to ignore the missingness mechanism (and potential distribution shifts) in a likelihood setting due to (a) the assumption of distinct parameters $\theta, \phi$ \eqref{crucialadditionalcond} and (b) the nature of MLE. In particular, even though the distribution $p_{\theta, \phi}^*(o(x,m), M=m)$ is not the same as the $p_{\theta}^*(o(x,m))$, it is \emph{essentially} the same from an MLE perspective: We can therefore simply maximize $p_{\theta}^*(o(x,m))$ over $\theta$ to get the MLE.

\subsubsection{Modern Results on Consistency}
In a somewhat overlooked paper, \cite{MLEConsistencyunderMissing} showed that Condition \eqref{crucialadditionalcond} is actually necessary only for efficiency reasons, but that even without it, the MLE is $\sqrt{n}$-consistent. The key argument for this in our context is showing that $\E[\log(p_{\theta}(o(X,M)))]$ is minimized at $\theta^*$, where expectation is taken over $p_{\theta, \phi}(x,m)=\Prob_{\phi}(M=m \mid x) p_{\theta}(x)$. As can be seen in \cite[page 3180]{MLEConsistencyunderMissing}, this argument crucially relies on the log and the fact that $\log(x) \leq x-1$.\footnote{The argument is not reproduced here due to space constraints.}
We note that this is thus deeply related to the fact that the MLE implicitly minimizes the KL-Divergence and these properties do not hold for other loss functions. Indeed, as discussed in \cite{Mestimatormissingvalues} and in Section \ref{Sec_Empirical}, M-Estimation besides MLE is in general not consistent under MAR. This underscores the messages of this paper that as soon as one deviates from the MLE framework, accounting for MAR is not trivial in general.

\end{appendix}


\begin{funding}
This work is part of the DIGPHAT project which was supported by a grant from the French government, managed by the National Research Agency (ANR), under the France 2030 program, with reference ANR-22-PESN-0017.
\end{funding}



\bibliographystyle{imsart-nameyear} 
\bibliography{biblio}       

\begin{thebibliography}{74}

\bibitem[\protect\citeauthoryear{Ambrogioni
  et~al.}{2018}]{WassersteinVariationalInference}
\begin{binproceedings}[author]
\bauthor{\bsnm{Ambrogioni},~\bfnm{Luca}\binits{L.}},
  \bauthor{\bsnm{G\"{u}\c{c}l\"{u}},~\bfnm{Umut}\binits{U.}},
  \bauthor{\bsnm{G\"{u}\c{c}l\"{u}t\"{u}rk},~\bfnm{Yagmur}\binits{Y.}},
  \bauthor{\bsnm{Hinne},~\bfnm{Max}\binits{M.}},
  \bauthor{\bsnm{Maris},~\bfnm{Eric}\binits{E.}} \AND \bauthor{\bparticle{van}
  \bsnm{Gerven},~\bfnm{Marcel A.~J.}\binits{M.~A.~J.}}
(\byear{2018}).
\btitle{Wasserstein variational inference}.
In \bbooktitle{Proceedings of the Thirty-Second International Conference on
  Neural Information Processing Systems}.
\bseries{Proceedings of Machine Learning Research}
\bpages{2478–2487}.
\end{binproceedings}
\endbibitem

\bibitem[\protect\citeauthoryear{Anil~Jadhav and Ramanathan}{2019}]{knn_adv1}
\begin{barticle}[author]
\bauthor{\bsnm{Anil~Jadhav},~\bfnm{Dhanya~Pramod}\binits{D.~P.}} \AND
  \bauthor{\bsnm{Ramanathan},~\bfnm{Krishnan}\binits{K.}}
(\byear{2019}).
\btitle{Comparison of Performance of Data Imputation Methods for Numeric
  Dataset}.
\bjournal{Applied Artificial Intelligence}
\bvolume{33}
\bpages{913-933}.
\end{barticle}
\endbibitem

\bibitem[\protect\citeauthoryear{Arjovsky, Chintala and
  Bottou}{2017}]{WassersteinGAN}
\begin{binproceedings}[author]
\bauthor{\bsnm{Arjovsky},~\bfnm{Martin}\binits{M.}},
  \bauthor{\bsnm{Chintala},~\bfnm{Soumith}\binits{S.}} \AND
  \bauthor{\bsnm{Bottou},~\bfnm{L{\'e}on}\binits{L.}}
(\byear{2017}).
\btitle{{W}asserstein Generative Adversarial Networks}.
In \bbooktitle{Proceedings of the 34th International Conference on Machine
  Learning}.
\bseries{Proceedings of Machine Learning Research}
\bvolume{70}
\bpages{214--223}.
\end{binproceedings}
\endbibitem

\bibitem[\protect\citeauthoryear{Bertsimas, Pawlowski and
  Zhuo}{2018}]{knn_adv2}
\begin{barticle}[author]
\bauthor{\bsnm{Bertsimas},~\bfnm{Dimitris}\binits{D.}},
  \bauthor{\bsnm{Pawlowski},~\bfnm{Colin}\binits{C.}} \AND
  \bauthor{\bsnm{Zhuo},~\bfnm{Ying~Daisy}\binits{Y.~D.}}
(\byear{2018}).
\btitle{From Predictive Methods to Missing Data Imputation: An Optimization
  Approach}.
\bjournal{Journal of Machine Learning Research}
\bvolume{18}
\bpages{1--39}.
\end{barticle}
\endbibitem

\bibitem[\protect\citeauthoryear{Bhattacharya et~al.}{2020}]{bhattacharya20b}
\begin{binproceedings}[author]
\bauthor{\bsnm{Bhattacharya},~\bfnm{Rohit}\binits{R.}},
  \bauthor{\bsnm{Nabi},~\bfnm{Razieh}\binits{R.}},
  \bauthor{\bsnm{Shpitser},~\bfnm{Ilya}\binits{I.}} \AND
  \bauthor{\bsnm{Robins},~\bfnm{James~M.}\binits{J.~M.}}
(\byear{2020}).
\btitle{Identification In Missing Data Models Represented By Directed Acyclic
  Graphs}.
In \bbooktitle{Proceedings of The 35th Uncertainty in Artificial Intelligence
  Conference}
(\beditor{\bfnm{Ryan~P.}\binits{R.~P.}~\bsnm{Adams}} \AND
  \beditor{\bfnm{Vibhav}\binits{V.}~\bsnm{Gogate}}, eds.).
\bseries{Proceedings of Machine Learning Research}
\bvolume{115}
\bpages{1149--1158}.
\end{binproceedings}
\endbibitem

\bibitem[\protect\citeauthoryear{Bonneel et~al.}{2015}]{slicedWasserstein}
\begin{barticle}[author]
\bauthor{\bsnm{Bonneel},~\bfnm{Nicolas}\binits{N.}},
  \bauthor{\bsnm{Rabin},~\bfnm{Julien}\binits{J.}},
  \bauthor{\bsnm{Peyr{\'e}},~\bfnm{Gabriel}\binits{G.}} \AND
  \bauthor{\bsnm{Pfister},~\bfnm{Hanspeter}\binits{H.}}
(\byear{2015}).
\btitle{Sliced and Radon Wasserstein Barycenters of Measures}.
\bjournal{Journal of Mathematical Imaging and Vision}
\bvolume{51}
\bpages{22-45}.
\end{barticle}
\endbibitem

\bibitem[\protect\citeauthoryear{Breiman}{2001}]{breiman2001random}
\begin{barticle}[author]
\bauthor{\bsnm{Breiman},~\bfnm{Leo}\binits{L.}}
(\byear{2001}).
\btitle{Random forests}.
\bjournal{Machine learning}
\bvolume{45}
\bpages{5--32}.
\end{barticle}
\endbibitem

\bibitem[\protect\citeauthoryear{Burgette and Reiter}{2010}]{CARTpaper0}
\begin{barticle}[author]
\bauthor{\bsnm{Burgette},~\bfnm{Lane~F.}\binits{L.~F.}} \AND
  \bauthor{\bsnm{Reiter},~\bfnm{Jerome~P.}\binits{J.~P.}}
(\byear{2010}).
\btitle{{Multiple Imputation for Missing Data via Sequential Regression
  Trees}}.
\bjournal{American Journal of Epidemiology}
\bvolume{172}
\bpages{1070-1076}.
\end{barticle}
\endbibitem

\bibitem[\protect\citeauthoryear{{Van Buuren}}{2018}]{VANBUUREN2018}
\begin{bbook}[author]
\bauthor{\bsnm{{Van Buuren}},~\bfnm{Stef}\binits{S.}}
(\byear{2018}).
\btitle{Flexible Imputation of Missing Data. Second Edition}.
\bpublisher{Chapman \& Hall/CRC Press}.
\end{bbook}
\endbibitem

\bibitem[\protect\citeauthoryear{{Van Buuren} and
  Groothuis-Oudshoorn}{2011}]{mice}
\begin{barticle}[author]
\bauthor{\bsnm{{Van Buuren}},~\bfnm{Stef}\binits{S.}} \AND
  \bauthor{\bsnm{Groothuis-Oudshoorn},~\bfnm{Karin}\binits{K.}}
(\byear{2011}).
\btitle{{mice}: Multivariate Imputation by Chained Equations in {R}}.
\bjournal{Journal of Statistical Software}
\bvolume{45}
\bpages{1-67}.
\end{barticle}
\endbibitem

\bibitem[\protect\citeauthoryear{Cantoni and {de Luna}}{2020}]{CANTONI2020}
\begin{barticle}[author]
\bauthor{\bsnm{Cantoni},~\bfnm{Eva}\binits{E.}} \AND \bauthor{\bsnm{{de
  Luna}},~\bfnm{Xavier}\binits{X.}}
(\byear{2020}).
\btitle{Semiparametric inference with missing data: Robustness to outliers and
  model misspecification}.
\bjournal{Econometrics and Statistics}
\bvolume{16}
\bpages{108-120}.
\end{barticle}
\endbibitem

\bibitem[\protect\citeauthoryear{{\'C}evid et~al.}{2022}]{DRF-paper}
\begin{barticle}[author]
\bauthor{\bsnm{{\'C}evid},~\bfnm{Domagoj}\binits{D.}},
  \bauthor{\bsnm{Michel},~\bfnm{Loris}\binits{L.}},
  \bauthor{\bsnm{Näf},~\bfnm{Jeffrey}\binits{J.}},
  \bauthor{\bsnm{Meinshausen},~\bfnm{Nicolai}\binits{N.}} \AND
  \bauthor{\bsnm{B\"{u}hlmann},~\bfnm{Peter}\binits{P.}}
(\byear{2022}).
\btitle{Distributional Random Forests: Heterogeneity Adjustment and
  Multivariate Distributional Regression}.
\bjournal{Journal of Machine Learning Research}
\bvolume{23}
\bpages{1--79}.
\end{barticle}
\endbibitem

\bibitem[\protect\citeauthoryear{Daniel~Malinsky and
  Tchetgen}{2022}]{Malinsky2022}
\begin{barticle}[author]
\bauthor{\bsnm{Daniel~Malinsky},~\bfnm{Ilya~Shpitser}\binits{I.~S.}} \AND
  \bauthor{\bsnm{Tchetgen},~\bfnm{Eric J.~Tchetgen}\binits{E.~J.~T.}}
(\byear{2022}).
\btitle{Semiparametric Inference for Nonmonotone Missing-Not-at-Random Data:
  The No Self-Censoring Model}.
\bjournal{Journal of the American Statistical Association}
\bvolume{117}
\bpages{1415--1423}.
\end{barticle}
\endbibitem

\bibitem[\protect\citeauthoryear{Deng, Han and
  Matteson}{2022}]{directcompetitor1}
\begin{barticle}[author]
\bauthor{\bsnm{Deng},~\bfnm{Grace}\binits{G.}},
  \bauthor{\bsnm{Han},~\bfnm{Cuize}\binits{C.}} \AND
  \bauthor{\bsnm{Matteson},~\bfnm{David~S.}\binits{D.~S.}}
(\byear{2022}).
\btitle{Extended missing data imputation via {GANs} for ranking applications}.
\bjournal{Data Mining and Knowledge Discovery}
\bvolume{36}
\bpages{1498-1520}.
\end{barticle}
\endbibitem

\bibitem[\protect\citeauthoryear{Dong et~al.}{2021}]{Dong2021}
\begin{barticle}[author]
\bauthor{\bsnm{Dong},~\bfnm{Weinan}\binits{W.}},
  \bauthor{\bsnm{Fong},~\bfnm{Daniel Yee~Tak}\binits{D.~Y.~T.}},
  \bauthor{\bsnm{Yoon},~\bfnm{Jin-sun}\binits{J.-s.}},
  \bauthor{\bsnm{Wan},~\bfnm{Eric Yuk~Fai}\binits{E.~Y.~F.}},
  \bauthor{\bsnm{Bedford},~\bfnm{Laura~Elizabeth}\binits{L.~E.}},
  \bauthor{\bsnm{Tang},~\bfnm{Eric Ho~Man}\binits{E.~H.~M.}} \AND
  \bauthor{\bsnm{Lam},~\bfnm{Cindy Lo~Kuen}\binits{C.~L.~K.}}
(\byear{2021}).
\btitle{Generative adversarial networks for imputing missing data for big data
  clinical research}.
\bjournal{BMC Medical Research Methodology}
\bvolume{21}
\bpages{78}.
\end{barticle}
\endbibitem

\bibitem[\protect\citeauthoryear{Doove, {Van Buuren} and
  Dusseldorp}{2014}]{CARTpaper1}
\begin{barticle}[author]
\bauthor{\bsnm{Doove},~\bfnm{Lisa~L.}\binits{L.~L.}}, \bauthor{\bsnm{{Van
  Buuren}},~\bfnm{Stef}\binits{S.}} \AND
  \bauthor{\bsnm{Dusseldorp},~\bfnm{Elise}\binits{E.}}
(\byear{2014}).
\btitle{Recursive partitioning for missing data imputation in the presence of
  interaction effects}.
\bjournal{Computational Statistics \& Data Analysis}
\bvolume{72}
\bpages{92-104}.
\end{barticle}
\endbibitem

\bibitem[\protect\citeauthoryear{Doretti, Geneletti and
  Stanghellini}{2018}]{MARthroughconditionalindependence}
\begin{barticle}[author]
\bauthor{\bsnm{Doretti},~\bfnm{Marco}\binits{M.}},
  \bauthor{\bsnm{Geneletti},~\bfnm{Sara}\binits{S.}} \AND
  \bauthor{\bsnm{Stanghellini},~\bfnm{Elena}\binits{E.}}
(\byear{2018}).
\btitle{Missing Data: A Unified Taxonomy Guided by Conditional Independence}.
\bjournal{International Statistical Review}
\bvolume{86}
\bpages{189-204}.
\end{barticle}
\endbibitem

\bibitem[\protect\citeauthoryear{Fang and Bao}{2023}]{directcompetitor2}
\begin{barticle}[author]
\bauthor{\bsnm{Fang},~\bfnm{Fang}\binits{F.}} \AND
  \bauthor{\bsnm{Bao},~\bfnm{Shenliao}\binits{S.}}
(\byear{2023}).
\btitle{{FragmGAN}: Generative adversarial nets for fragmentary data imputation
  and prediction}.
\bjournal{Statistical Theory and Related Fields}
\bvolume{0}
\bpages{1-14}.
\end{barticle}
\endbibitem

\bibitem[\protect\citeauthoryear{Frahm, Nordhausen and
  Oja}{2020}]{Mestimatormissingvalues}
\begin{barticle}[author]
\bauthor{\bsnm{Frahm},~\bfnm{Gabriel}\binits{G.}},
  \bauthor{\bsnm{Nordhausen},~\bfnm{Klaus}\binits{K.}} \AND
  \bauthor{\bsnm{Oja},~\bfnm{Hannu}\binits{H.}}
(\byear{2020}).
\btitle{M-estimation with incomplete and dependent multivariate data}.
\bjournal{Journal of Multivariate Analysis}
\bvolume{176}
\bpages{104569}.
\end{barticle}
\endbibitem

\bibitem[\protect\citeauthoryear{Gneiting and Raftery}{2007}]{gneiting}
\begin{barticle}[author]
\bauthor{\bsnm{Gneiting},~\bfnm{Tilmann}\binits{T.}} \AND
  \bauthor{\bsnm{Raftery},~\bfnm{Adrian~E}\binits{A.~E.}}
(\byear{2007}).
\btitle{Strictly Proper Scoring Rules, Prediction, and Estimation}.
\bjournal{Journal of the American Statistical Association}
\bvolume{102}
\bpages{359-378}.
\end{barticle}
\endbibitem

\bibitem[\protect\citeauthoryear{Hong and Lynn}{2020}]{RFimputationpaper}
\begin{barticle}[author]
\bauthor{\bsnm{Hong},~\bfnm{Shangzhi}\binits{S.}} \AND
  \bauthor{\bsnm{Lynn},~\bfnm{Henry~S.}\binits{H.~S.}}
(\byear{2020}).
\btitle{Accuracy of random-forest-based imputation of missing data in the
  presence of non-normality, non-linearity, and interaction}.
\bjournal{BMC Medical Research Methodology}
\bvolume{20}
\bpages{199}.
\end{barticle}
\endbibitem

\bibitem[\protect\citeauthoryear{Ibrahim, Lipsitz and
  Chen}{1999}]{sequentialapproach0noaccess}
\begin{barticle}[author]
\bauthor{\bsnm{Ibrahim},~\bfnm{Joseph~G.}\binits{J.~G.}},
  \bauthor{\bsnm{Lipsitz},~\bfnm{Stuart~R.}\binits{S.~R.}} \AND
  \bauthor{\bsnm{Chen},~\bfnm{Ming-Hui}\binits{M.-H.}}
(\byear{1999}).
\btitle{Missing covariates in generalized linear models when the missing data
  mechanism is non-ignorable}.
\bjournal{Journal of the Royal Statistical Society: Series B (Statistical
  Methodology)}
\bvolume{61}
\bpages{173-190}.
\end{barticle}
\endbibitem

\bibitem[\protect\citeauthoryear{Imaizumi, Ota and
  Hamaguchi}{2022}]{Wassersteintesting}
\begin{barticle}[author]
\bauthor{\bsnm{Imaizumi},~\bfnm{Masaaki}\binits{M.}},
  \bauthor{\bsnm{Ota},~\bfnm{Hirofumi}\binits{H.}} \AND
  \bauthor{\bsnm{Hamaguchi},~\bfnm{Takuo}\binits{T.}}
(\byear{2022}).
\btitle{Hypothesis Test and Confidence Analysis With Wasserstein Distance on
  General Dimension}.
\bjournal{Neural Computation}
\bvolume{34}
\bpages{1448-1487}.
\end{barticle}
\endbibitem

\bibitem[\protect\citeauthoryear{Jäger, Allhorn and
  Bießmann}{2021}]{awesomebenchmarkingpaper}
\begin{barticle}[author]
\bauthor{\bsnm{Jäger},~\bfnm{Sebastian}\binits{S.}},
  \bauthor{\bsnm{Allhorn},~\bfnm{Arndt}\binits{A.}} \AND
  \bauthor{\bsnm{Bießmann},~\bfnm{Felix}\binits{F.}}
(\byear{2021}).
\btitle{A Benchmark for Data Imputation Methods}.
\bjournal{Frontiers in Big Data}
\bvolume{4}.
\end{barticle}
\endbibitem

\bibitem[\protect\citeauthoryear{Kokla et~al.}{2019}]{metabolomics1}
\begin{barticle}[author]
\bauthor{\bsnm{Kokla},~\bfnm{Marietta}\binits{M.}},
  \bauthor{\bsnm{Virtanen},~\bfnm{Jyrki}\binits{J.}},
  \bauthor{\bsnm{Kolehmainen},~\bfnm{Marjukka}\binits{M.}},
  \bauthor{\bsnm{Paananen},~\bfnm{Jussi}\binits{J.}} \AND
  \bauthor{\bsnm{Hanhineva},~\bfnm{Kati}\binits{K.}}
(\byear{2019}).
\btitle{Random forest-based imputation outperforms other methods for imputing
  LC-MS metabolomics data: a comparative study}.
\bjournal{BMC Bioinformatics}
\bvolume{20}
\bpages{492}.
\end{barticle}
\endbibitem

\bibitem[\protect\citeauthoryear{Lee and Mitra}{2016}]{sequentialapproach1}
\begin{barticle}[author]
\bauthor{\bsnm{Lee},~\bfnm{Min~Cherng}\binits{M.~C.}} \AND
  \bauthor{\bsnm{Mitra},~\bfnm{Robin}\binits{R.}}
(\byear{2016}).
\btitle{Multiply imputing missing values in data sets with mixed measurement
  scales using a sequence of generalised linear models}.
\bjournal{Computational Statistics \& Data Analysis}
\bvolume{95}
\bpages{24-38}.
\end{barticle}
\endbibitem

\bibitem[\protect\citeauthoryear{Li et~al.}{2011}]{inverseweightingoverview}
\begin{barticle}[author]
\bauthor{\bsnm{Li},~\bfnm{Lingling}\binits{L.}},
  \bauthor{\bsnm{Shen},~\bfnm{Changyu}\binits{C.}},
  \bauthor{\bsnm{Li},~\bfnm{Xiaochun}\binits{X.}} \AND
  \bauthor{\bsnm{Robins},~\bfnm{James~M}\binits{J.~M.}}
(\byear{2011}).
\btitle{On weighting approaches for missing data}.
\bjournal{Stat Methods Med Res}
\bvolume{22}
\bpages{14--30}.
\end{barticle}
\endbibitem

\bibitem[\protect\citeauthoryear{Little}{1993}]{little_patternmixture}
\begin{barticle}[author]
\bauthor{\bsnm{Little},~\bfnm{Roderick J.~A.}\binits{R.~J.~A.}}
(\byear{1993}).
\btitle{Pattern-Mixture Models for Multivariate Incomplete Data}.
\bjournal{Journal of the American Statistical Association}
\bvolume{88}
\bpages{125-134}.
\end{barticle}
\endbibitem

\bibitem[\protect\citeauthoryear{Little and Donald}{1986}]{RubinLittlebook}
\begin{bbook}[author]
\bauthor{\bsnm{Little},~\bfnm{Roderick J.~A.}\binits{R.~J.~A.}} \AND
  \bauthor{\bsnm{Donald},~\bfnm{Rubin~B.}\binits{R.~B.}}
(\byear{1986}).
\btitle{Statistical Analysis with Missing Data}.
\bpublisher{John Wiley \& Sons, Inc.}
\end{bbook}
\endbibitem

\bibitem[\protect\citeauthoryear{Little and
  Rubin}{2019}]{little2019statistical}
\begin{bbook}[author]
\bauthor{\bsnm{Little},~\bfnm{Roderick J.~A.}\binits{R.~J.~A.}} \AND
  \bauthor{\bsnm{Rubin},~\bfnm{Donald~B.}\binits{D.~B.}}
(\byear{2019}).
\btitle{Statistical Analysis with Missing Data., 3rd Edition}.
\bpublisher{Wiley}.
\end{bbook}
\endbibitem

\bibitem[\protect\citeauthoryear{Liu et~al.}{2014}]{MICE_Results}
\begin{barticle}[author]
\bauthor{\bsnm{Liu},~\bfnm{Jingchen}\binits{J.}},
  \bauthor{\bsnm{Gelman},~\bfnm{Andrew}\binits{A.}},
  \bauthor{\bsnm{Hill},~\bfnm{Jennifer}\binits{J.}},
  \bauthor{\bsnm{Su},~\bfnm{Yu-Sung}\binits{Y.-S.}} \AND
  \bauthor{\bsnm{Kropko},~\bfnm{Jonathan}\binits{J.}}
(\byear{2014}).
\btitle{On the stationary distribution of iterative imputations}.
\bjournal{Biometrika}
\bvolume{1}
\bpages{155--173}.
\end{barticle}
\endbibitem

\bibitem[\protect\citeauthoryear{Malistov and
  Trushin}{2019}]{malistov2019gradient}
\begin{binproceedings}[author]
\bauthor{\bsnm{Malistov},~\bfnm{Alexey}\binits{A.}} \AND
  \bauthor{\bsnm{Trushin},~\bfnm{Arseniy}\binits{A.}}
(\byear{2019}).
\btitle{Gradient boosted trees with extrapolation}.
In \bbooktitle{2019 18th IEEE International Conference On Machine Learning And
  Applications}
\bpages{783--789}.
\end{binproceedings}
\endbibitem

\bibitem[\protect\citeauthoryear{Mattei and Frellsen}{2019}]{MIWAE}
\begin{binproceedings}[author]
\bauthor{\bsnm{Mattei},~\bfnm{Pierre-Alexandre}\binits{P.-A.}} \AND
  \bauthor{\bsnm{Frellsen},~\bfnm{Jes}\binits{J.}}
(\byear{2019}).
\btitle{{MIWAE}: Deep Generative Modelling and Imputation of Incomplete Data
  Sets}.
In \bbooktitle{Proceedings of the 36th International Conference on Machine
  Learning}
\bvolume{97}
\bpages{4413--4423}.
\end{binproceedings}
\endbibitem

\bibitem[\protect\citeauthoryear{Mealli and Rubin}{2015}]{whatismeant3}
\begin{barticle}[author]
\bauthor{\bsnm{Mealli},~\bfnm{Fabrizia}\binits{F.}} \AND
  \bauthor{\bsnm{Rubin},~\bfnm{Donald~B.}\binits{D.~B.}}
(\byear{2015}).
\btitle{{Clarifying missing at random and related definitions, and implications
  when coupled with exchangeability}}.
\bjournal{Biometrika}
\bvolume{102}
\bpages{995-1000}.
\end{barticle}
\endbibitem

\bibitem[\protect\citeauthoryear{Mohan, Pearl and Tian}{2013}]{Mohan2013}
\begin{bincollection}[author]
\bauthor{\bsnm{Mohan},~\bfnm{Karthika}\binits{K.}},
  \bauthor{\bsnm{Pearl},~\bfnm{Judea}\binits{J.}} \AND
  \bauthor{\bsnm{Tian},~\bfnm{Jin}\binits{J.}}
(\byear{2013}).
\btitle{Graphical Models for Inference with Missing Data}.
In \bbooktitle{Advances in Neural Information Processing Systems 26}
\bpages{1277--1285}.
\end{bincollection}
\endbibitem

\bibitem[\protect\citeauthoryear{Mohan and Pearl}{2014}]{mohan2013missing}
\begin{barticle}[author]
\bauthor{\bsnm{Mohan},~\bfnm{Karthika}\binits{K.}} \AND
  \bauthor{\bsnm{Pearl},~\bfnm{Judea}\binits{J.}}
(\byear{2014}).
\btitle{On the testability of models with missing data}.
\bjournal{Artificial Intelligence and Statistics}
\bpages{643--650}.
\bnote{PMLR}.
\end{barticle}
\endbibitem

\bibitem[\protect\citeauthoryear{Mohan and Pearl}{2021}]{mohan2021graphical}
\begin{barticle}[author]
\bauthor{\bsnm{Mohan},~\bfnm{Karthika}\binits{K.}} \AND
  \bauthor{\bsnm{Pearl},~\bfnm{Judea}\binits{J.}}
(\byear{2021}).
\btitle{Graphical models of processing missing data}.
\bjournal{Journal of the American Statistical Association}
\bpages{1--16}.
\end{barticle}
\endbibitem

\bibitem[\protect\citeauthoryear{Molenberghs et~al.}{2008}]{ourresult}
\begin{barticle}[author]
\bauthor{\bsnm{Molenberghs},~\bfnm{Geert}\binits{G.}},
  \bauthor{\bsnm{Beunckens},~\bfnm{Caroline}\binits{C.}},
  \bauthor{\bsnm{Sotto},~\bfnm{Cristina}\binits{C.}} \AND
  \bauthor{\bsnm{Kenward},~\bfnm{Michael~G.}\binits{M.~G.}}
(\byear{2008}).
\btitle{Every Missingness Not at Random Model Has a Missingness at Random
  Counterpart with Equal Fit}.
\bjournal{Journal of the Royal Statistical Society. Series B (Statistical
  Methodology)}
\bvolume{70}
\bpages{371--388}.
\end{barticle}
\endbibitem

\bibitem[\protect\citeauthoryear{Murray}{2018}]{greatoverview}
\begin{barticle}[author]
\bauthor{\bsnm{Murray},~\bfnm{Jared~S.}\binits{J.~S.}}
(\byear{2018}).
\btitle{Multiple Imputation: A Review of Practical and Theoretical Findings}.
\bjournal{Statistical Science}
\bvolume{33}
\bpages{142 -- 159}.
\end{barticle}
\endbibitem

\bibitem[\protect\citeauthoryear{Muzellec et~al.}{2020}]{OTimputation}
\begin{binproceedings}[author]
\bauthor{\bsnm{Muzellec},~\bfnm{Boris}\binits{B.}},
  \bauthor{\bsnm{Josse},~\bfnm{Julie}\binits{J.}},
  \bauthor{\bsnm{Boyer},~\bfnm{Claire}\binits{C.}} \AND
  \bauthor{\bsnm{Cuturi},~\bfnm{Marco}\binits{M.}}
(\byear{2020}).
\btitle{Missing Data Imputation using Optimal Transport}.
In \bbooktitle{Proceedings of the 37th International Conference on Machine
  Learning}
\bpages{7130-7140}.
\end{binproceedings}
\endbibitem

\bibitem[\protect\citeauthoryear{Nabi, Bhattacharya and
  Shpitser}{2020}]{nabi2020missing}
\begin{binproceedings}[author]
\bauthor{\bsnm{Nabi},~\bfnm{Razieh}\binits{R.}},
  \bauthor{\bsnm{Bhattacharya},~\bfnm{Rohit}\binits{R.}} \AND
  \bauthor{\bsnm{Shpitser},~\bfnm{Ilya}\binits{I.}}
(\byear{2020}).
\btitle{Full law identification in graphical models of missing data:
  Completeness results}.
In \bbooktitle{Proceedings of the Twenty Seventh International Conference on
  Machine Learning (ICML-20)}.
\end{binproceedings}
\endbibitem

\bibitem[\protect\citeauthoryear{Nabi et~al.}{2025}]{nabi2025define}
\begin{barticle}[author]
\bauthor{\bsnm{Nabi},~\bfnm{Razieh}\binits{R.}},
  \bauthor{\bsnm{Bhattacharya},~\bfnm{Rohit}\binits{R.}},
  \bauthor{\bsnm{Shpitser},~\bfnm{Ilya}\binits{I.}} \AND
  \bauthor{\bsnm{Robins},~\bfnm{James}\binits{J.}}
(\byear{2025}).
\btitle{Causal and counterfactual views of missing data models}.
\bjournal{Statistica Sinica}.
\end{barticle}
\endbibitem

\bibitem[\protect\citeauthoryear{Nazábal et~al.}{2020}]{VAE1}
\begin{barticle}[author]
\bauthor{\bsnm{Nazábal},~\bfnm{Alfredo}\binits{A.}},
  \bauthor{\bsnm{Olmos},~\bfnm{Pablo~M.}\binits{P.~M.}},
  \bauthor{\bsnm{Ghahramani},~\bfnm{Zoubin}\binits{Z.}} \AND
  \bauthor{\bsnm{Valera},~\bfnm{Isabel}\binits{I.}}
(\byear{2020}).
\btitle{Handling incomplete heterogeneous data using {VAEs}}.
\bjournal{Pattern Recognition}
\bvolume{107}
\bpages{107501}.
\end{barticle}
\endbibitem

\bibitem[\protect\citeauthoryear{Nelsen}{2006}]{Nelsen2006}
\begin{bbook}[author]
\bauthor{\bsnm{Nelsen},~\bfnm{Roger~B.}\binits{R.~B.}}
(\byear{2006}).
\btitle{An Introduction to Copulas},
\bedition{2nd} ed.
\bseries{Springer Series in Statistics}.
\bpublisher{Springer}, \baddress{New York}.
\end{bbook}
\endbibitem

\bibitem[\protect\citeauthoryear{Näf et~al.}{2023}]{ImputationScores}
\begin{barticle}[author]
\bauthor{\bsnm{Näf},~\bfnm{Jeffrey}\binits{J.}},
  \bauthor{\bsnm{Spohn},~\bfnm{Meta-Lina}\binits{M.-L.}},
  \bauthor{\bsnm{Michel},~\bfnm{Loris}\binits{L.}} \AND
  \bauthor{\bsnm{Meinshausen},~\bfnm{Nicolai}\binits{N.}}
(\byear{2023}).
\btitle{{Imputation scores}}.
\bjournal{The Annals of Applied Statistics}
\bvolume{17}
\bpages{2452 -- 2472}.
\end{barticle}
\endbibitem

\bibitem[\protect\citeauthoryear{Oberst et~al.}{2020}]{overlap1}
\begin{binproceedings}[author]
\bauthor{\bsnm{Oberst},~\bfnm{Michael}\binits{M.}},
  \bauthor{\bsnm{Johansson},~\bfnm{Fredrik}\binits{F.}},
  \bauthor{\bsnm{Wei},~\bfnm{Dennis}\binits{D.}},
  \bauthor{\bsnm{Gao},~\bfnm{Tian}\binits{T.}},
  \bauthor{\bsnm{Brat},~\bfnm{Gabriel}\binits{G.}},
  \bauthor{\bsnm{Sontag},~\bfnm{David}\binits{D.}} \AND
  \bauthor{\bsnm{Varshney},~\bfnm{Kush}\binits{K.}}
(\byear{2020}).
\btitle{Characterization of Overlap in Observational Studies}.
In \bbooktitle{Proceedings of the Twenty Third International Conference on
  Artificial Intelligence and Statistics}
(\beditor{\bfnm{Silvia}\binits{S.}~\bsnm{Chiappa}} \AND
  \beditor{\bfnm{Roberto}\binits{R.}~\bsnm{Calandra}}, eds.).
\bseries{Proceedings of Machine Learning Research}
\bvolume{108}
\bpages{788--798}.
\end{binproceedings}
\endbibitem

\bibitem[\protect\citeauthoryear{Peyre and Cuturi}{2019}]{ComputationalOT}
\begin{barticle}[author]
\bauthor{\bsnm{Peyre},~\bfnm{Gabriel}\binits{G.}} \AND
  \bauthor{\bsnm{Cuturi},~\bfnm{Marco}\binits{M.}}
(\byear{2019}).
\btitle{Computational Optimal Transport}.
\bjournal{Foundations and Trends in Machine Learning}
\bvolume{11}
\bpages{355--607}.
\end{barticle}
\endbibitem

\bibitem[\protect\citeauthoryear{Pfister and
  Bühlmann}{2024}]{pfister2024extrapolationaware}
\begin{barticle}[author]
\bauthor{\bsnm{Pfister},~\bfnm{Niklas}\binits{N.}} \AND
  \bauthor{\bsnm{Bühlmann},~\bfnm{Peter}\binits{P.}}
(\byear{2024}).
\btitle{Extrapolation-Aware Nonparametric Statistical Inference}.
\bjournal{Preprint arXiv:2402.09758}.
\end{barticle}
\endbibitem

\bibitem[\protect\citeauthoryear{Potthoff et~al.}{2006}]{RMAR2006}
\begin{barticle}[author]
\bauthor{\bsnm{Potthoff},~\bfnm{Richard~F}\binits{R.~F.}},
  \bauthor{\bsnm{Tudor},~\bfnm{Gail~E}\binits{G.~E.}},
  \bauthor{\bsnm{Pieper},~\bfnm{Karen~S}\binits{K.~S.}} \AND
  \bauthor{\bsnm{Hasselblad},~\bfnm{Vic}\binits{V.}}
(\byear{2006}).
\btitle{Can one assess whether missing data are missing at random in medical
  studies?}
\bjournal{Stat Methods Med Res}
\bvolume{15}
\bpages{213--234}.
\end{barticle}
\endbibitem

\bibitem[\protect\citeauthoryear{Qiu, Zheng and Gevaert}{2020}]{VAE2}
\begin{barticle}[author]
\bauthor{\bsnm{Qiu},~\bfnm{Yeping~Lina}\binits{Y.~L.}},
  \bauthor{\bsnm{Zheng},~\bfnm{Hong}\binits{H.}} \AND
  \bauthor{\bsnm{Gevaert},~\bfnm{Olivier}\binits{O.}}
(\byear{2020}).
\btitle{{Genomic data imputation with variational auto-encoders}}.
\bjournal{GigaScience}
\bvolume{9}
\bpages{giaa082}.
\end{barticle}
\endbibitem

\bibitem[\protect\citeauthoryear{Rabe-Hesketh and
  Skrondal}{2023}]{RMAR_Psychometrika}
\begin{barticle}[author]
\bauthor{\bsnm{Rabe-Hesketh},~\bfnm{Sophia}\binits{S.}} \AND
  \bauthor{\bsnm{Skrondal},~\bfnm{Anders}\binits{A.}}
(\byear{2023}).
\btitle{Ignoring Non-ignorable Missingness}.
\bjournal{Psychometrika}
\bvolume{88}
\bpages{31-50}.
\end{barticle}
\endbibitem

\bibitem[\protect\citeauthoryear{Redko
  et~al.}{2019}]{WassersteinDomainadaptation}
\begin{binproceedings}[author]
\bauthor{\bsnm{Redko},~\bfnm{Ievgen}\binits{I.}},
  \bauthor{\bsnm{Courty},~\bfnm{Nicolas}\binits{N.}},
  \bauthor{\bsnm{Flamary},~\bfnm{R\'emi}\binits{R.}} \AND
  \bauthor{\bsnm{Tuia},~\bfnm{Devis}\binits{D.}}
(\byear{2019}).
\btitle{Optimal Transport for Multi-source Domain Adaptation under Target
  Shift}.
In \bbooktitle{Proceedings of the Twenty-Second International Conference on
  Artificial Intelligence and Statistics}.
\bseries{Proceedings of Machine Learning Research}
\bvolume{89}
\bpages{849--858}.
\end{binproceedings}
\endbibitem

\bibitem[\protect\citeauthoryear{Ren et~al.}{2023}]{directcompetitor0}
\begin{barticle}[author]
\bauthor{\bsnm{Ren},~\bfnm{Boyu}\binits{B.}},
  \bauthor{\bsnm{Lipsitz},~\bfnm{Stuart~R}\binits{S.~R.}},
  \bauthor{\bsnm{Weiss},~\bfnm{Roger~D}\binits{R.~D.}} \AND
  \bauthor{\bsnm{Fitzmaurice},~\bfnm{Garrett~M}\binits{G.~M.}}
(\byear{2023}).
\btitle{Multiple imputation for non-monotone missing not at random data using
  the no self-censoring model}.
\bjournal{Stat Methods Med Res}
\bvolume{32}
\bpages{1973--1993}.
\end{barticle}
\endbibitem

\bibitem[\protect\citeauthoryear{Rizzo and Szekely}{2022}]{energypackage}
\begin{bmanual}[author]
\bauthor{\bsnm{Rizzo},~\bfnm{Maria}\binits{M.}} \AND
  \bauthor{\bsnm{Szekely},~\bfnm{Gabor}\binits{G.}}
(\byear{2022}).
\btitle{energy: E-Statistics: Multivariate Inference via the Energy of Data}
\bnote{R package version 1.7-11}.
\end{bmanual}
\endbibitem

\bibitem[\protect\citeauthoryear{Rubin}{1976}]{Rubin_Inferenceandmissing}
\begin{barticle}[author]
\bauthor{\bsnm{Rubin},~\bfnm{Donald~B.}\binits{D.~B.}}
(\byear{1976}).
\btitle{{Inference and missing data}}.
\bjournal{Biometrika}
\bvolume{63}
\bpages{581-592}.
\end{barticle}
\endbibitem

\bibitem[\protect\citeauthoryear{Schafer}{1997}]{schafer1997analysis}
\begin{bbook}[author]
\bauthor{\bsnm{Schafer},~\bfnm{Joseph~L}\binits{J.~L.}}
(\byear{1997}).
\btitle{Analysis of incomplete multivariate data}.
\bpublisher{Chapman and Hall/CRC}.
\end{bbook}
\endbibitem

\bibitem[\protect\citeauthoryear{Schouten, Lugtig and Vink}{2018}]{ampute}
\begin{barticle}[author]
\bauthor{\bsnm{Schouten},~\bfnm{Rianne~Margaretha}\binits{R.~M.}},
  \bauthor{\bsnm{Lugtig},~\bfnm{Peter}\binits{P.}} \AND
  \bauthor{\bsnm{Vink},~\bfnm{Gerko}\binits{G.}}
(\byear{2018}).
\btitle{Generating missing values for simulation purposes: a multivariate
  amputation procedure}.
\bjournal{Journal of Statistical Computation and Simulation}
\bvolume{88}
\bpages{2909-2930}.
\end{barticle}
\endbibitem

\bibitem[\protect\citeauthoryear{Seaman and
  Vansteelandt}{2018}]{Seaman2018-IPW}
\begin{barticle}[author]
\bauthor{\bsnm{Seaman},~\bfnm{Shaun~R}\binits{S.~R.}} \AND
  \bauthor{\bsnm{Vansteelandt},~\bfnm{Stijn}\binits{S.}}
(\byear{2018}).
\btitle{Introduction to Double Robust Methods for Incomplete Data}.
\bjournal{Stat Sci}
\bvolume{33}
\bpages{184--197}.
\end{barticle}
\endbibitem

\bibitem[\protect\citeauthoryear{Seaman et~al.}{2013}]{whatismeant}
\begin{barticle}[author]
\bauthor{\bsnm{Seaman},~\bfnm{Shaun}\binits{S.}},
  \bauthor{\bsnm{Galati},~\bfnm{John}\binits{J.}},
  \bauthor{\bsnm{Jackson},~\bfnm{Dan}\binits{D.}} \AND
  \bauthor{\bsnm{Carlin},~\bfnm{John}\binits{J.}}
(\byear{2013}).
\btitle{{What is meant by “Missing at Random”?}}
\bjournal{Statistical Science}
\bvolume{28}
\bpages{257-268}.
\end{barticle}
\endbibitem

\bibitem[\protect\citeauthoryear{Shadbahr et~al.}{2023}]{Naturepaper}
\begin{barticle}[author]
\bauthor{\bsnm{Shadbahr},~\bfnm{T.}\binits{T.}},
  \bauthor{\bsnm{Roberts},~\bfnm{M.}\binits{M.}},
  \bauthor{\bsnm{Stanczuk},~\bfnm{J.}\binits{J.}} \betal{et~al.}
(\byear{2023}).
\btitle{The impact of imputation quality on machine learning classifiers for
  datasets with missing values}.
\bjournal{Communications Medicine}
\bvolume{3}
\bpages{139}.
\end{barticle}
\endbibitem

\bibitem[\protect\citeauthoryear{Shpitser}{2016}]{Shpitser_2016}
\begin{binproceedings}[author]
\bauthor{\bsnm{Shpitser},~\bfnm{Ilya}\binits{I.}}
(\byear{2016}).
\btitle{Consistent Estimation of Functions of Data Missing Non-Monotonically
  and Not at Random}.
In \bbooktitle{Advances in Neural Information Processing Systems}
\bvolume{29}.
\end{binproceedings}
\endbibitem

\bibitem[\protect\citeauthoryear{Stekhoven and
  Bühlmann}{2011}]{stekhoven_missoforest}
\begin{barticle}[author]
\bauthor{\bsnm{Stekhoven},~\bfnm{Daniel~J.}\binits{D.~J.}} \AND
  \bauthor{\bsnm{Bühlmann},~\bfnm{Peter}\binits{P.}}
(\byear{2011}).
\btitle{{MissForest—non-parametric missing value imputation for mixed-type
  data}}.
\bjournal{Bioinformatics}
\bvolume{28}
\bpages{112-118}.
\end{barticle}
\endbibitem

\bibitem[\protect\citeauthoryear{Sun and Tchetgen}{2018}]{MARinverseweighting}
\begin{barticle}[author]
\bauthor{\bsnm{Sun},~\bfnm{BaoLuo}\binits{B.}} \AND
  \bauthor{\bsnm{Tchetgen},~\bfnm{Eric J.~Tchetgen}\binits{E.~J.~T.}}
(\byear{2018}).
\btitle{On Inverse Probability Weighting for Nonmonotone Missing at Random
  Data}.
\bjournal{Journal of the American Statistical Association}
\bvolume{113}
\bpages{369--379}.
\end{barticle}
\endbibitem

\bibitem[\protect\citeauthoryear{Székely}{2003}]{EnergyDistance}
\begin{btechreport}[author]
\bauthor{\bsnm{Székely},~\bfnm{Gabor~J.}\binits{G.~J.}}
(\byear{2003}).
\btitle{E-statistics: the energy of statistical samples}
\btype{Technical Report} No. \bnumber{05},
\bpublisher{Bowling Green State University, Department of Mathematics and
  Statistics}.
\end{btechreport}
\endbibitem

\bibitem[\protect\citeauthoryear{Takai and
  Kano}{2013}]{MLEConsistencyunderMissing}
\begin{barticle}[author]
\bauthor{\bsnm{Takai},~\bfnm{Keiji}\binits{K.}} \AND
  \bauthor{\bsnm{Kano},~\bfnm{Yutaka}\binits{Y.}}
(\byear{2013}).
\btitle{Asymptotic Inference with Incomplete Data}.
\bjournal{Communications in Statistics - Theory and Methods}
\bvolume{42}
\bpages{3174--3190}.
\end{barticle}
\endbibitem

\bibitem[\protect\citeauthoryear{Tang and Ishwaran}{2017}]{Tang2017-on}
\begin{barticle}[author]
\bauthor{\bsnm{Tang},~\bfnm{Fei}\binits{F.}} \AND
  \bauthor{\bsnm{Ishwaran},~\bfnm{Hemant}\binits{H.}}
(\byear{2017}).
\btitle{Random Forest Missing Data Algorithms}.
\bjournal{Stat Anal Data Min}
\bvolume{10}
\bpages{363--377}.
\end{barticle}
\endbibitem

\bibitem[\protect\citeauthoryear{Tian}{2017}]{RecoveringProbabilityDistributionsfrommissing}
\begin{binproceedings}[author]
\bauthor{\bsnm{Tian},~\bfnm{Jin}\binits{J.}}
(\byear{2017}).
\btitle{Recovering Probability Distributions from Missing Data}.
In \bbooktitle{Proceedings of the Ninth Asian Conference on Machine Learning}
\bpages{574--589}.
\end{binproceedings}
\endbibitem

\bibitem[\protect\citeauthoryear{Van~Buuren}{2007}]{FCS_Van_Buuren2007}
\begin{barticle}[author]
\bauthor{\bsnm{Van~Buuren},~\bfnm{Stef}\binits{S.}}
(\byear{2007}).
\btitle{Multiple imputation of discrete and continuous data by fully
  conditional specification}.
\bjournal{Stat Methods Med Res}
\bvolume{16}
\bpages{219--242}.
\end{barticle}
\endbibitem

\bibitem[\protect\citeauthoryear{Waljee et~al.}{2013}]{Waljee2013}
\begin{barticle}[author]
\bauthor{\bsnm{Waljee},~\bfnm{Akbar~K.}\binits{A.~K.}},
  \bauthor{\bsnm{Mukherjee},~\bfnm{Ashin}\binits{A.}},
  \bauthor{\bsnm{Singal},~\bfnm{Amit~G.}\binits{A.~G.}},
  \bauthor{\bsnm{Zhang},~\bfnm{Yiwei}\binits{Y.}},
  \bauthor{\bsnm{Warren},~\bfnm{Jeffrey}\binits{J.}},
  \bauthor{\bsnm{Balis},~\bfnm{Ulysses}\binits{U.}},
  \bauthor{\bsnm{Marrero},~\bfnm{Jorge}\binits{J.}},
  \bauthor{\bsnm{Zhu},~\bfnm{Ji}\binits{J.}} \AND
  \bauthor{\bsnm{Higgins},~\bfnm{Peter~Dr}\binits{P.~D.}}
(\byear{2013}).
\btitle{Comparison of imputation methods for missing laboratory data in
  medicine}.
\bjournal{BMJ open}
\bvolume{3}.
\end{barticle}
\endbibitem

\bibitem[\protect\citeauthoryear{Wang et~al.}{2022}]{wang2022deep}
\begin{barticle}[author]
\bauthor{\bsnm{Wang},~\bfnm{Zhenhua}\binits{Z.}},
  \bauthor{\bsnm{Akande},~\bfnm{Olanrewaju}\binits{O.}},
  \bauthor{\bsnm{Poulos},~\bfnm{Jason}\binits{J.}} \AND
  \bauthor{\bsnm{Li},~\bfnm{Fan}\binits{F.}}
(\byear{2022}).
\btitle{Are deep learning models superior for missing data imputation in
  surveys? {E}vidence from an empirical comparison}.
\bjournal{Survey Methodology}
\bvolume{48}.
\end{barticle}
\endbibitem

\bibitem[\protect\citeauthoryear{Xu, Daniels and
  Winterstein}{2016}]{sequentialapproach2}
\begin{barticle}[author]
\bauthor{\bsnm{Xu},~\bfnm{Dandan}\binits{D.}},
  \bauthor{\bsnm{Daniels},~\bfnm{Michael~J.}\binits{M.~J.}} \AND
  \bauthor{\bsnm{Winterstein},~\bfnm{Almut~G.}\binits{A.~G.}}
(\byear{2016}).
\btitle{Sequential {BART} for imputation of missing covariates}.
\bjournal{Biostatistics}
\bvolume{17}
\bpages{589--602}.
\end{barticle}
\endbibitem

\bibitem[\protect\citeauthoryear{Yoon, Jordon and van~der Schaar}{2018}]{GAIN}
\begin{binproceedings}[author]
\bauthor{\bsnm{Yoon},~\bfnm{Jinsung}\binits{J.}},
  \bauthor{\bsnm{Jordon},~\bfnm{James}\binits{J.}} \AND
  \bauthor{\bparticle{van~der} \bsnm{Schaar},~\bfnm{Mihaela}\binits{M.}}
(\byear{2018}).
\btitle{{GAIN}: Missing Data Imputation using Generative Adversarial Nets}.
In \bbooktitle{Proceedings of the 35th International Conference on Machine
  Learning}
\bpages{5689--5698}.
\end{binproceedings}
\endbibitem

\bibitem[\protect\citeauthoryear{Yuan et~al.}{2021}]{VAE3}
\begin{binproceedings}[author]
\bauthor{\bsnm{Yuan},~\bfnm{Yuan}\binits{Y.}},
  \bauthor{\bsnm{Shen},~\bfnm{Yulong}\binits{Y.}},
  \bauthor{\bsnm{Wang},~\bfnm{Jian}\binits{J.}},
  \bauthor{\bsnm{Liu},~\bfnm{Yang}\binits{Y.}} \AND
  \bauthor{\bsnm{Zhang},~\bfnm{Le}\binits{L.}}
(\byear{2021}).
\btitle{{VAEM}: a Deep Generative Model for Heterogeneous Mixed Type Data}.
In \bbooktitle{Advances in Neural Information Processing Systems 34}
\bpages{4044--4054}.
\end{binproceedings}
\endbibitem

\bibitem[\protect\citeauthoryear{Zhu and Raghunathan}{2015}]{MICE_Results2}
\begin{barticle}[author]
\bauthor{\bsnm{Zhu},~\bfnm{Jian}\binits{J.}} \AND
  \bauthor{\bsnm{Raghunathan},~\bfnm{Trivellore~E.}\binits{T.~E.}}
(\byear{2015}).
\btitle{Convergence Properties of a Sequential Regression Multiple Imputation
  Algorithm}.
\bjournal{Journal of the American Statistical Association}
\bvolume{110}
\bpages{1112-1124}.
\end{barticle}
\endbibitem

\end{thebibliography}


\end{document}